\documentclass[12pt]{amsart}
\usepackage[bookmarks=true,colorlinks=true,linkcolor=blue,filecolor=magenta,urlcolor=cyan,citecolor=black]{hyperref}
\usepackage{amsmath}
\usepackage{amscd}
\usepackage{amssymb}
\usepackage{amsfonts}
\usepackage{bbm}
\usepackage{bm}
\usepackage{cancel}
\usepackage{comment}
\usepackage{enumerate}
\usepackage{fancyhdr} 
\usepackage{graphicx}
\usepackage[margin=1in]{geometry}
\usepackage{mathrsfs}
\usepackage{mathtools}
\usepackage{setspace}
\usepackage{subfigure}
\usepackage{tikz}
\usepackage{tikz-cd}
\usepackage{textcomp}
\usepackage{ulem}
\usepackage[all]{xy}

\usetikzlibrary{matrix,arrows,calc,intersections,fit,decorations.pathmorphing}
\usetikzlibrary{shapes}
\tikzset{snake it/.style={decorate, decoration=snake}}

\newcommand{\bb}{\mathbb}
\newcommand{\BB}{\mathcal{B}}
\newcommand{\Cal}{\mathcal}
\newcommand{\Z}{\mathbb{Z}}
\newcommand{\R}{\mathbb{R}}
\newcommand{\C}{\mathbb{C}}

\newcommand{\D}{\mathbb{D}}
\newcommand{\EE}{\mathcal{E}}
\newcommand{\GG}{\mathcal{G}}
\newcommand{\HH}{\mathbb{H}}

\newcommand{\F}{\mathcal{F}}

\newcommand{\LL}{\mathcal{L}}

\newcommand{\dbar}{\bar\partial}

\newcommand{\del}{\partial}
\newcommand{\Lag}{\mathcal{L}\mathit{ag}}
\newcommand{\CM}{\mathcal{M}}
\newcommand{\CW}{\mathcal{CW}}
\newcommand{\Moore}{\mathcal{P}}
\newcommand{\tLag}{\widetilde{\Lag}}
\newcommand{\bfD}{\mathbf{D}}
\newcommand{\ro}{{\mathrm o}}
\newcommand{\RR}{\mathcal{R}}
\newcommand{\sS}{\mathcal{S}}

\newcommand{\RRbar}{\overline{\RR}}
\newcommand{\UU}{\mathcal{U}}
\newcommand{\ZZ}{\mathcal{Z}}

\newcommand{\ZZbar}{\overline{\mathcal{Z}}}
\newcommand{\NN}{\mathcal{N}}

\newcommand{\sH}{\mathcal{H}}

\newcommand{\JJ}{\mathcal{J}}
\newcommand{\Disc}{{\mathcal M}}
\newcommand{\Discbar}{\overline{\CM}}
\newcommand{\Disct}{\tilde{\mathcal M}}

\newcommand{\Wrap}{\mathcal{W}}

\newcommand{\Chord}{\mathcal{X}}

\newcommand{\vsigma}[1][\!]{\vec{\sigma}^{\, #1}}
\newcommand{\vtau}[1][\!]{\vec{\tau}^{\, #1}}

\newcommand{\into}{\hookrightarrow}
\newcommand{\iso}{\cong}

\newcommand{\RNum}[1]{\uppercase\expandafter{\romannumeral #1\relax}}

\newcommand{\di}{\operatorname{dim}}

\newtheorem{thm}{Theorem}[section]
\newtheorem{cor}[thm]{Corollary}
\newtheorem{lem}[thm]{Lemma}
\newtheorem{prop}[thm]{Proposition}
\newtheorem{defin}[thm]{Definition}

\theoremstyle{remark}
\newtheorem{rem}[thm]{Remark}

\numberwithin{equation}{section}

\numberwithin{mytheorem}{subsection}

\begin{document}


\title{On Wrapped Fukaya Category and the loop space of Lagrangians in a Liouville manifold}
\author{Alex Zhongyi Zhang}

\maketitle


\section{Abstract}
	We introduce an $A_\infty$ map from the cubical chain complex of the based loop space of Lagrangian submanifolds with Legendrian boundary in a Liouville Manifold $C_{*}(\Omega_{L} \Lag)$ to wrapped Floer cohomology of Lagrangian submanifold $\CW^{-*}(L,L)$. In the case of a cotangent bundle and a Lagrangian co-fiber, the composition of our map with the map from $\CW^*(L,L) \to C_{*}(\Omega_q Q) $ as defined in \cite{Ab12} shows that this map is split surjective.

\setcounter{tocdepth}{1}



\section{Introduction}

Originally arising from attempts to understand the geometry of Hamiltonian mechanics, symplectic geometry is the study of even dimensional manifolds that are equipped with a non-degenerate $2$-form. The canonical model of a symplectic manifold is the cotangent bundle of a smooth manifold, where the symplectic form arises from pairing tangent and cotangent vectors.
The study of symplectic geometry underwent a revolution after Gromov \cite{Gromov85} introduced pseudo-holomorphic curves. Soon after, Floer introduced his famous homology theory to solve Arnold's conjecture on fixed points of Hamiltonian systems and Fukaya changed the understanding of symplectic structures by introducing Fukaya Categories, whose objects are compact Lagrangians and whose morphisms are generated by their intersections; one key discovery is that the operations in this category are encoded in holomorphic discs.

One of the motivating problems of symplectic topology is Arnold's nearby Lagrangian conjecture: any closed exact Lagrangian in a cotangent bundle is Hamiltonian isotopic to the zero section.
Work related to this conjecture has spanned over two decades: in \cite{Vit94},\cite{Vit96} Viterbo proved that the Floer homology of the cotangent bundle of a closed manifold  (a symplectic object) is isomorphic to the cohomology of the loop space of the zero section (a topological object). Abbondandolo and Schwarz showed the existence of a chain level isomorphism from the Morse homology of the based loop space of a closed manifold $H_*(\Omega_q Q)$ to the Lagrangian version of Hamiltonian Floer homology of the cotangent bundle $T^*Q$ \cite{AbSc10}, and that this map preserves the ring structures in that there is a correspondence between the Chas-Sullivan loop product and pair-of-pants product. This version of Floer homology is now known as the wrapped Floer homology, which was developed by Abouzaid and Seidel, who used wrapping to give control at infinity. This approach enlarged the Fukaya category to allow certain non-compact exact Lagrangians, such as the cotangent fibers $T_q^*Q$, where the morphisms given by two Lagrangians are given by Floer Homology modified by a large Hamiltonian Perturbation at infinity; 
Abouzaid further showed that the wrapped Fukaya category $\CW(T^*Q)$ is generated by a single Lagrangian cofiber $T_q^*Q$ \cite{Abcotangen}.
Building on the correspondence between the Morse homology of loop space and symplectic cohomology of cotangent bundles, in \cite{Ab12} Abouzaid showed that $\CW^*(T_q^*Q, T_q^*Q)$ is quasi-isomorphic as an $A_\infty$-algebra to $C_{-*}(\Omega_q Q, \eta)$, the chains of based loop space of the zero section twisted by a canonical local system, which is a dg-algebra under the concatenation of loops. This is the inverse to the map constructed in \cite{AbSc10} on the level of homology. Recently Kragh also showed homotopy equivalence by studying directly the geometry of an exact Lagrangian through intersection with cotangent fiber \cite{Kragh}. 

We would like to understand the relations between the wrapped Floer homology of any open Liouville manifold $M$ and the loop space of exact Lagrangians inside.

In order to understand the topological structure of the space $\Lag$ of exact Lagrangian submanifolds with Legendrian boundary of a Liouville manifold, whose homotopic information can be obtained from its based loop space, we construct an $A_\infty$ functor between the\textit{ Pontryagin category} of decorated Lagrangians to the wrapped Fukaya category of the Liouville Manifold.

Here the Pontryagin category of a topological space $X$ is a DG category whose objects are the points of $X$ and morphisms between two points are given by cubical chains on the space of paths in $X$ between the points. The associative product structure is given by concatenation of paths. The exact definition of Pontryagin Category of a path connected space is in Section (\ref{sec: PreviewFloer}).

Consider exact Lagrangian $\iota: L \into M $ that is the completion of one with Legendrian boundary near cylindrical ends. Let $ b \in H^2(M; \Z/2)$ be a class whose pullback under $\iota$ is the second Stiefel-Whitney class of $L$:
\begin{equation}
 \iota^*b = w_2(L).
\end{equation}
Let $\Lag_b$ denote the space of Lagrangians that satisfy the above equation.
We may define a relative Pin structure on $L$ relative to $M$ from $b$, which is a free $H^1(L; \Z/2) $-torsor. Therefore we consider the space $\tLag_b$ of decorated Lagrangian $\LL = (L, \beta)$ such that $L \in \Lag_b$ and $\beta \in H^1(L;\Z/2)$; in other words, it is the space of all possible relative Pin structures on a Lagrangian.
\[
\begin{tikzcd}   
H^1(L;\Z/2) \ar{r} &\tLag_b\ar{d}{\pi} \\
&\Lag_b
\end{tikzcd}
\]

Let $\tLag$ denote the space of all decorated Lagrangians, we get the main theorem of this paper.
\begin{thm}\label{thm: main theorem in introduction}
	For any open Liouville manifold $(M,\omega)$ with cylindrical ends, there exists an $A_\infty$ functor from the Pontryagin category of decorated Lagrangians to the wrapped Fukaya category of the ambient symplectic manifold:
	\begin{equation}\label{Map:functor}
	\Moore (\tLag)  \xrightarrow{\F} \Wrap^*(M).
	\end{equation}
\end{thm}

\begin{rem}
	The degree $0$ map of the $A_\infty$ morphism is a chain map $\F^1$:
	\begin{equation}
	C_*(\Omega_{\LL_0,\LL_1}\tLag) \xrightarrow{\F^1} \CW^{-*}(\LL_0,\LL_1).
	\end{equation}
	When $\LL_0 = \LL_1 = \LL$, this induces a map from the homology of the based loop space of Lagrangians with Legendrian boundary to the wrapped Floer homology of the base point Lagrangian.
	\begin{equation}\label{homology}
	H_{*}(\Omega_{\LL} \tLag) \xrightarrow{H(\F)} \Cal{HW^{-*}}(\LL).
	\end{equation} 
	If $M  = T^*Q$ and $\LL = T_q^*Q$ this is surjective as shown by the theorem below, but not necessarily surjective in the general case.
\end{rem}

The above theorem implies that we can use Wrapped Floer homology to detect nontrivial families of Lagrangians. As an example for computation, let the Liouville manifold $M$ be the cotangent bundle $T^*Q$ of a compact manifold $Q$. In this case, M. Abouzaid constructed an $A_\infty$ functor, which, in the case of a cotangent fiber $\LL = T^*_q Q$, gives a map on morphism spaces $f:\Cal{CW}^*(\LL,\LL; H) \to C_{-*}(\Omega_q Q)$ \cite{Ab12}.
Here each cotangent fibre $T_q^* Q$ is contractible, thus the map $\tLag_b \to \Lag_b$ is an isomorphism.
By taking the above $A_\infty$ functor on the level of automorphisms of a single Lagrangian and composing with the chain level homorphism from $f$ mentioned above, we obtain the following homotopy commutative diagram:

\begin{thm}\label{thm:cotangent_bundle}
	With $L = T_q^*Q$, the map $(f\circ \F \circ \eta)_*$ in the diagram below equals the identity map on $H_{-*}(\Omega_q Q)$.
	
	\[\xymatrix{
		&C_{-*}(\Omega_q Q)\ar@{.>}[dr] \ar^{id}[d] \ar^{\eta}[r]  &C_{-*}(\Omega_{L}{\Lag}) \ar^{\F}[d]\\
		&C_{-*}(\Omega_q Q)	&\Cal{CW}^*(\LL,\LL; H)  \ar^{f}[l]
	}
	\]
	Here the map $\eta$ sends every loop in the base to the corresponding loop of cotangent fibers in the space of Lagrangian embeddings.
	
\end{thm}


\begin{rem}
	Recently in a series of papers\cite{GPS01, GPS02, GPS03}, Ganatra, Pardon and Shende used the theory of microlocal sheaves to study the symplectic structure of Liouville manifolds and it is very likely that one can also use their machinery to achieve the results of this paper. 
\end{rem}

Our paper first reviews the theory of wrapped Fukaya category and Pontryagin category, then construct the moduli spaces that give rise to the functor. Afterwards we take the special case of a cotangent bundle of a compact manifold as an example of this $A_\infty$ functor.  Finally we attach brief explanations of relative Pin structures, signed operations, $C^0$ estimates and gluing theorem in the appendix.

\section{Wrapped Fukaya Category} \label{sec:Wrapped}

\subsection{Review of Floer theory in open Liouville manifold} \label{sec: PreviewFloer}

In \cite{AbSe09} Abouzaid and Seidel constructed the wrapped Floer cohomology on Liouville domains in analogy with the Lagrangian Floer cohomology of a closed symplectic manifold. They defined $HW^*(L;H)\cong \underrightarrow{\text{\it lim}}_w \, HF^*(L;wH)$ from a direct system of continuation maps from Lagrangian Floer cohomology with Hamiltonians of increasing slope at infinity. Later in \cite{Ab12}, Abouzaid used Hamiltonians that are quadratic at infinite ends to define wrapped Floe homology. We are going to follow this definition. 

\subsubsection{Geometric Setup} \label{subsec:geometric setup} From now on we work on open Liouville manifold $(M,\omega)$ with Liouville $1$-form $\theta$ and infinite cylindrical ends. Namely there is a compact Liouville domain $(M^{in}, \omega)$ with boundary, together with a one-form $\theta$ such that $\omega = d\theta$ is symplectic, and the dual Liouville vector field $Z$, uniquely determined by the requirement that $i_Z \omega = \theta$ that points strictly outwards along $\del M^{in}$. This implies that $\theta|_{\del M^{in}}$ is a contact one-form. By flowing inwards from the boundary along $Z$, one obtains a collar embedding $(0, 1] \times \del M^{in} \into M^{in}$. One can then complete $M^{in}$ to obtain $M$ by attaching the big half of symplectization:
\begin{equation}\label{eq:completion}
 M = M^{in} \cup_{\partial M^{in}} ([1,\infty) \times \partial M^{in}).
\end{equation}
The piece $[1,\infty) \times \partial M^{in}$ is called the \textit{infinite cone}. The completion carries a natural one-form $\hat{\theta}
$ such that $d\hat{\theta} = \hat{\omega}$ is symplectic and an associated Liouville vector field $\hat{Z}$. On $M^{in} \subset M$ these restrict to the previous data, while on the overlapping piece $(0, \infty) \times \del M^{in}$, $\hat{\theta} = r(\theta|_{\del M^{in}})$ and $\hat{Z} = r\del_r$, where $r$ is the variable on $(0, \infty)$.  

Let $L \in M^{in}$ be a Lagrangian intersecting $\del M^{in}$ transversally, and which has the following property
\begin{align}\label{eq:exactness}
	& \parbox{30em} {$\theta|_L$ is exact, $\theta|_L = df.$ Moreover, $\theta|_L$ vanishes near the boundary $\del L = L \cap \del M^{in}$.}
\end{align}
The second part implies that $\del L $ is a Legendrian submanifold of $\del M^{in}$. Furthermore, by attaching an infinite cone to this boundary, one can extend $L$ to a smooth noncompact Lagrangian submanifold 
\begin{equation}
	\hat{L} = L \cup_{\partial L} ([1,\infty) \times \partial L) \subset M.
\end{equation}
In the remainder of this paper, we only use $\theta, \omega, L$ to denote their extensions (resp. completions) in $M$ since there won't be any confusion. 

 We use $\psi^\rho$ to denote the image of the Liouville flow for time log$(\rho)$ outwards along the cylindrical end. Namely, $$\psi^\rho(r,x)=(\rho \cdot r, x).$$

Let $\sH (M)\subseteq C^\infty(M,\R)$ denote the set of smooth functions $H$ such that 
\begin{equation}\label{eq:quadratic Hamiltonian}
 H(q,r)=r^2
\end{equation}
for any $(q, r)$ (where $q \in \del M^{in}$) on the cylindrical end and $r \gg 1$. Namely $H$ is quadratic with respect to the radial coordinate away from some compact subset of $M$.  We fix such a function for the purpose of defining Floer homology, and let $X_H$ denote the Hamiltonian flow of $H$ defined by the equation 
\begin{equation} \label{eqn:Hamiltonian_equation}
i_{X_H} \omega =-dH.
\end{equation}

We denote by $\Lag$ the space of all Lagrangians with vanishing Maslov index that also satisfies Condition (\ref{eq:exactness}), namely those that are equipped with Legendrian boundary and the restriction of the one form $\theta$ on $L$ is exact,
\begin{equation} \label{eq:primitive_of_one_form}
\theta|_L=df,
\end{equation}
for some $f\in C^\infty(L,\R).$ When $L$ moves in $\Lag$, we shall choose a primitive $f_L$ for the restriction of $\theta$ to each such Lagrangian that is continuous with respect to $L$.

For each pair of Lagrangians $L_0, L_1 \in \Lag$, we define $\Chord (L_0,L_1)$ to be the set of time-1 flow lines of $X_H$ that start from $L_0$ and stop at $L_1$, namely a map $x:[0,1] \to M$ such that
\[
\begin{cases}
x(0) \in L_0\\
x(1) \in L_1\\
\frac{dx}{dt}=X_H\\
\end{cases}
\]
In order to get well-defined moduli spaces, we assume that 
\begin{equation}
\label{eq:non-degenerate_chord}
\parbox{36em}{all  time-$1$ Hamiltonian chords of $H$ with boundaries on $L_0$ and $L_1$  are non-degenerate.}
\end{equation}

It is convenient to point out that the Hamiltonian chords above correspond bijectively to the intersection points of $L_1 $ with the time-1 flow $\phi^1_H(L_0)$ of $L_0$ under the Hamiltonian flow of $H$. In particular, the non-degeneracy of the chords is the same as transversal intersection of $L_1$ and $\phi^1_H(L_0)$. Condition (\ref{eq:non-degenerate_chord}) is therefore true after we do a generic Hamiltonian perturbation on the Lagrangians that intersect transversally with $\del M^{in}$.  Since Lagrangian Floer Homology is invariant under Hamiltonian perturbation, the above assumption is valid.

\subsubsection{Brief discussion of relative Pin and brane structures} \label{subsec: Pin and brane structures}

We briefly introduce the definition of Lagrangian branes for the sake of a proper definition of Lagrangian Floer homology.
Let $(M,\omega,J)$ be symplectic with $2c_1(M)=0$, equipped with a non-vanishing quadratic volume form $\eta_M^2$, namely a global section of the line bundle $(\Omega^{n,0}M)^{\otimes 2}=(\wedge^n_\C T^*M)^{\otimes 2}$. Likewise, given a symplectic vector bundle $\varphi: E \to B$ of rank $r$ such that $2c_1(E)=0$, we have a global non-vanishing complex volume form $\eta_E^2 \in H^0(B,(\wedge^{r}_\C E)^{\otimes 2})$. 
Let $Gr(E)$ be the set of Lagrangian sub-bundles of $E$,  and $v_1,\cdots, v_r$ be a basis of $E_x$ for any $x \in B$.
 We define the squared phase map of $E$ as follows:

\begin{defin}\label{defin:phase_map}

	Given a symplectic vector bundle $\varphi: E \to B$ of rank $r$ such that $2c_1(E)=0$, and any Lagrangian sub bundle $\phi:F\subset E$, consider the following squared phase map of $E$:
	\begin{align*}
		\alpha_E: &Gr(E)\times B\to S^1, \\
		\alpha_E(F,x)&:=\frac{\eta^2_x(v_1\wedge v_2\wedge \cdots v_r)}{|\eta^2_x(v_1\wedge v_2\wedge \cdots v_r)|}.
	\end{align*}
	 
\end{defin}
\begin{rem}
	A simple linear algebra exercise shows that $\alpha_E$ is independent of the choices of the set of basis; thus it is well defined.
\end{rem}

Moreover, if the Maslov class of $F$ vanishes, namely $\alpha_E(F,\cdot)_* : H_1(B)\to H_1(S^1) \simeq \Z$ is zero, we lift the map $\alpha_E(F,\cdot)$ to a map $\alpha^\#$ from $B$ to the universal cover $\R$ of $S^1$. In this way we obtain the following definition:

\begin{defin}\label{defin:gradings_on_Lagrangian_bundles_with_vanishing_Maslov_class}
	For a Lagrangian sub-bundle $F$ with vanishing Maslov class inside a symplectic fibration $E \to B$, an associated grading $\alpha^\#$ on $F$ is a function:
\begin{align}\label{eq:grading}
	\alpha^\#: & B \to \R,\\
 \text{such that   }	\exp(2\pi i \alpha^\#(x))&:=\alpha_E(F,x). \label{eq:grading_condition}
\end{align}
\end{defin}

\begin{rem}
	The choice of grading is not a canonical one, in fact, it is a free torsor of $H^0(B; \Z)$. 
\end{rem}

We further assume that for some $b \in H^2(B;\Z/2)$, the second Stiefel-Whitney class $w_2(\phi)\in H^2(B;\Z/2)$ is the same as $b$, namely $b=w_2(\phi)$. Fixing a triangulation of $B$, we can choose an oriented vector bundle over the three skeleton $B_{[3]},\psi: E_b \to B_{[3]}$ such that $w_2(\psi)=b|_{B_{[3]}}$. Then our assumption on $w_2(\phi)$ gives 
\begin{equation}\label{eq:relative Pin structure}
	w_2(\phi|_{B_{[3]}} \oplus \psi)=0.
\end{equation}
From the standard literature(\cite[Section 11i]{seidel-book}, \cite{kirby}), we know that second Stiefel-Whitney class is precisely the obstruction to the existence of Pin structures:
\begin{defin}\label{def:Relative_Pin_structures}
	A relative Pin structure on a Lagrangian sub-bundle $F$ is the choice of a Pin structure on the vector bundle $F|_{B_{[3]}}\oplus E_b$.
\end{defin}

\begin{defin}[brane structure]\label{def: brane structure}
	Given a symplectic vector bundle $E\to B$ and $b \in H^2(B;\Z/2)$, consider a Lagrangian sub-bundle $F\to B$ such that $2\mu(F)=0, \ w_2(F)+b=0$. A brane structure with respect to $b$ consists of a pair ($\alpha^\#, P^\#$), where $\alpha^\#$ satisfies the equation (\ref{eq:grading_condition}) and $P^\#$ is a relative Pin structure on $F$.
\end{defin}
\begin{rem}
	The obstruction to the existence of a grading is the Maslov class $\mu(F) \in H^1(B;\Z)$ represented by the map $\alpha_E(F,\cdot): B \to S^1$; if $b=0$, then we have a Pin-structure on $F$ in the sense of \cite{seidel-book}, whose obstruction is the second Stiefel-Whitney class $w_2(E) \in H^2(B,\Z/2)$. Here we take the direct sum with $E_b$ so as to ``kill" the obstruction given by possibly non-zero $b$. 
\end{rem}
Now in our situation, the base $B$ of the fibration is an exact Lagrangian sub manifold $L\in M$ with Legendrian boundary, the sub-bundle is just $TL \subset TM|_L$ whose Maslov class $\mu(TL)$ vanishes.

Therefore, to each Hamiltonian chord $x \in \Chord(L_0, L_1)$, we can assign a grading $\mu(x)$ from the gradings on $L_0, L_1$ for given transverse brane structures on $L_0,L_1$. Here we define $\mu(x):=\lfloor \alpha_1^\#-\alpha^\#_0 \rfloor+1$, where $\lfloor\cdot \rfloor$ is the floor function that lies in $\Z$. We also define the orientation space $o_x=Hom(P^\#_0,P^\#_1)\otimes TL_1^{\otimes \mu(x)}$.
 For details we refer to \cite[Section 11]{seidel-book}. 
 
 We denote the space of exact Lagrangians with Legendrian boundary condition with a relatively Pin structure with respect to the background class $b \in H^2(M)$ by
 \begin{equation}\label{eq: $Lag_b$}
 \Lag_b
 \end{equation}
  Namely Lagrangians with vanishing $w_2(L)+b$. We know that the isomorphism classes of Pin structures on $TL$ is a free torsor of $H^1(L;\Z/2)$ by examining the long exact sequence
 \begin{equation}
  \cdots \to H^1(L; \Z/2) \to H^1(L; Pin(n))\to H^1(L, O(n)) \to H^2(L;\Z/2) \to \cdots
 \end{equation}
 that arises from the exact sequence of classifying spaces:
 \begin{equation}
 B\Z/2 \to BPin(n)\to BO(n) \to B^2\Z/2 \to \cdots
 \end{equation} 
 Thus for each $L \in \Lag_b$, together with a Pin structure $P^\#$ on $TL$ and any class $\beta \in H^1(L;\Z/2)$, we may twist $P^\#$ by $\beta$ as follows: take a line bundle $\xi$ on $L$ whose first Stiefel-Whitney class is $\beta$ and its associated double cover $S(\xi)$. The Pin structure $P^\#$ is a Pin(n)-principal bundle $P^\# \to L$ with an isomorphism $P^\#\times_{Pin(n)}\R^n \iso TL$. Let $\pm e$ denote the kernel of the double cover $Pin(n) \to O(n)$. We take the product $S(\xi)\times_{\Z/2} P^\#$ where $\Z/2$ acts on $S(\xi)$ by $\pm1$ and on $P^\#$ by $\pm e$. This gives a new Pin(n)-structure on $TL$.
 
  A brane structure on $L$, by Definition (\ref{def: brane structure}), consists of a grading $\alpha^\#: L \to \R$ and Pin structure $P^\#$; the gradings $\alpha^\#$ are free torsors of $H^0(L,\Z)$. Notice in our situation we don't require that $w_2(L)=0$, yet all the Lagrangians we consider have to satisfy $w_2(TL)=i_L^*b$ for some background class $b\in H^2(M,\Z/2)$.
 Denoting by $\tLag_b$ the space of Lagrangians equipped with a specific Pin structure ($L, P^\#$) with a fixed background class $b$ (we may omit this in later section and just denote the space $\tLag$ when it is clear from the context), we obtain a fibration:
 \begin{equation*}
 \begin{tikzcd}
 H^1(L;\Z/2) \ar{r} &\tLag_b\ar{d}{\pi} \\
 &\Lag_b
 \end{tikzcd}
\end{equation*}
Note that for Lagrangians of different first cohomology groups, this fibration gives different fibers, thus in practice we only look at a component of any given $L$ such that the fibration $\pi$ is surjective with the same fiber $H^1(L;\Z/2)$.

\subsection{Almost complex structures}\label{sec:almost_complex_structure_in_general}

In this paper we work with almost complex structures of contact type, namely those satisfy 
\begin{equation}\label{eq:contact type}
	dr \circ J =-\theta 
\end{equation}
near cylindrical ends.
Let $\JJ(M)$ denote the space of such almost complex structures on $M$ which are also compatible with the symplectic form on $M$. This is an infinite-dimensional manifold. Its tangent space $T\JJ(M)$ at any point $J$ consists of endomorphisms $K$ of $TM$ which anti-commute with $J$, and which have the property that $\omega(\cdot, K \cdot)$ is symmetric. We call these endomorphisms \textit{infinitesimal deformations} of the almost complex structure. One can get an actual deformation by exponentiating $K$, namely setting 
\begin{equation} \label{eqn:almost_complex_structure_by_perturbation}
	J'= J \exp(-J K).
\end{equation}
 As in \cite[Section 3]{AbSe09}, for any disc $S$ with punctures on the boundary that are equipped with a set of strip-like ends $\{\epsilon^k\} : Z_{\pm} \to S$, and weights $\{w^k\}$ there exists a contractible choice of almost complex structures parametrized by $S$ such that for each strip-like end $\epsilon^k : Z_{\pm} \to S$ with $(s, t)$ denoting the coordinates of $Z_{\pm}$, the almost complex structure $J_{\epsilon^k(s,t)}$ is only dependent on $t$ for $|s| \gg 0$. In other words, if we go sufficiently far along any end, $J$ becomes independent of $s$, hence compatible with translations in the $s$ direction.

\subsection{Moduli spaces}\label{sec: Moduli}
We consider the space of maps from the strip $Z = \R \times [0,1]$ to a symplectic manifold $M$ satisfying certain equations. First of all, we need to fix some data for the equations: a family $J_t \in \JJ(M)$ of almost complex structures parametrized by $t \in [0,1]$; a closed one-form $\gamma$ on $Z$ such that $\gamma(s, t) = dt$ when $|s| \gg 0$ and a Hamiltonian $H: M \to \R$ that satisfies Condition (\ref{eq:non-degenerate_chord}).
Given a pair $x_0,x_1 \in \Chord(L_0,L_1)$, we have the following definition:

\begin{defin}
	
 $\Disct(x_0;x_1)$ is the moduli space of maps 
\begin{equation*}
	u: Z = \R \times [0,1] \to M,
\end{equation*}
satifying the following conditions
\begin{equation*}
\left \{
\begin{aligned}
&u(\R \times \{1\}) \subseteq L_1,\\
&u(\R \times \{0\}) \subseteq L_0,\\
&\lim_{s\to -\infty}u(s,\cdot) =x_0(\cdot),\\
&\lim_{s\to \infty}u(s,\cdot) =x_1(\cdot).\\
\end{aligned}
\right.
\end{equation*}
such that $u$ solves the Cauchy-Riemann equation 
\begin{equation} \label{eq:CauchyRiemannEq}
(du-X_H\otimes \gamma)^{0,1}=0
\end{equation}
which can also be written explicitly with the family of almost complex structure $J_t$:
\begin{equation*}
	J_t(\del_t u -X_H \cdot \gamma) - (\del_t u -X_H \cdot \gamma)\circ j =0.
\end{equation*}
\end{defin}
Since equation (\ref{eq:CauchyRiemannEq}) is invariant under translation of the variable $s$, there is an action of $\R$ on $\Disct(x_0;x_1)$. After taking the quotient by the free $\R$-action, we get $\CM(x_0;x_1)$; and let it be empty if the action is not free. Standard arguments of the gluing theorem such as those in Appendix (6.3) of \cite{Ab08} give the following:

\begin{prop}
	The moduli space $\Disc(x_0;x_1)$ is regular for a generic family of almost complex structures $J_t$ and has dimension $\mu(x_0)-\mu(x_1) -1$. The Gromov compactification $\Discbar(x_0;x_1) $ is a topological manifold with boundary. The boundary is stratified into topological manifolds with the closure of codimension $1$ strata given by the images of embeddings
	\begin{equation}
		\Discbar(x_0; y) \times \Discbar(y; x_1) \to \Discbar(x_0; x_1),
	\end{equation}
	for all possible $y$ in $\Chord(L_0, L_1)$. For each pair $(y_1, y_2)$ that gives a boundary stratum of higher codimension, we obtain a commutative diagram:
	\begin{equation}\label{eq:corner_chart_moduli_space}
	\xymatrix{
	\Discbar (x_0; y_1) \times \Discbar(y_1; y_2) \times \Discbar(y_2; x_1) \ar[r]  \ar[d] & \Discbar (x_0; y_1) \times \Discbar (y_1; x_1) \ar[d] \\
	\Discbar (x_0; y_2) \times \Discbar(y_2; x_1) \ar[r] & \Discbar (x_0; x_1).
	}
	\end{equation}
	\qed
\end{prop}

In particular, the codimension $k$ boundary strata consist of disjoint unions over all possible sequences of products of moduli spaces of the following form:
\begin{equation}
\Discbar(x_0; y_1) \times \Discbar(y_1; y_2)\times \cdots \times \Discbar(y_k;x_1).
\end{equation}

The standard arguments of Gromov compactness and the energy estimate arguments given by \cite[lemma 7.2]{AbSe09} gives the compactness of $\Disc(x_0, x_1)$, thus we have the following corollary.
\begin{cor}\label{cor: compactification_strip_moduli}
	For each chord $x_1(t)$, the moduli space $\Discbar(x_0;x_1)$ is empty for all but finitely many choices of $x_0$, and is a compact manifold with boundary of dimension $\mu(x_0)-\mu(x_1) - 1$ whenever $J_t$ is a generic family of almost complex structure. Moreover, the boundary is covered by the images of the natural inclusions.
	\begin{equation*}
		\Discbar(x_0; y) \times \Discbar(y;x_1) \to \Discbar(x_0;x_1).
	\end{equation*}
	\qed
\end{cor}
From now on through the end of this section, we fix the almost complex structure $J_t$ for which the above corollary holds. 

\subsection{Wrapped Floer complex}\label{sec: FloerComplex}
We define the graded vector space underlying the Floer chain complex to be the direct sum
\begin{equation} \label{eq: FloerComplex}
	\CW_b^*(L_0; L_1):=\bigoplus_{x \in \Chord(L_0;L_1)}|o_x|,
\end{equation}
where $o_x$ is the determinant line bundle given by the paths of Lagrangians under the Hamiltonian flow associated with $x$. The details are explained in \cite[Appendix A]{Ab12}
We may simply assume that it is just the vector space freely generated by the sets of chords unless one wants to worry about signs.

The differential is a count of the solutions of rigid moduli spaces of equation (\ref{eq:CauchyRiemannEq}), namely those with $\mu(x_0)=\mu(x_1)+1$. Every element $u$ of the moduli space $\Disc(x_0, x_1)$ is rigid and thus defines an isomorphism:
\begin{equation*}
	o_{x_0} \iso o_{x_1}.
\end{equation*}
Thus an orientation of $o_{x_0}$ is induced from one on $o_{x_1}$, and the map on the orientation lines is denoted $\mu_u$. We define the map:
	\begin{align}
		\label{eq: wrap_differential}
 	\mu_1: \CW_b^i(L_0; L_1) &\to \CW_b^{i+1}(L_0, L_1) \\
 	[x_1] &\mapsto (-1)^i \sum_u \mu_u ([x_1]).
	\end{align}	
The proof that this count is a well defined differential of chain complexes is standard (see eg. \cite{AD}). Moreover, Corollary (\ref{cor: compactification_strip_moduli}) shows that each chord can be the input of only finitely many solutions to (\ref{eq:CauchyRiemannEq}), thus we know that on the right hand side of (\ref{eq: wrap_differential}) there are only finitely many summands.

\begin{rem}
	The graded vector space underlying $\CW_b^*(L_0; L_1)$ depends on $L_0, L_1, H,\omega$, meanwhile the differential depends on $J_t$. When we want to distinguish them as in the next lemma, we write $\CW_b^*(L_0,L_1; \omega, H, J_t)$.
\end{rem}

\begin{lem}
	 If $\psi : M \to M$ satisfies $\psi^*(\omega)=\rho \omega$ for some non-zero constant $\rho$, then we have a canonical isomorphism
	 \begin{align}
	 	\label{eq:rescale}
	 	\CW(\psi): \CW_b^*(L_0, L_1; \omega, H, J_t) \iso \CW_b^*(\psi(L_0), \psi(L_1), \frac{H\circ \psi}{\rho}, \psi^*(J_t)).
	 \end{align}
\end{lem}
This was shown as \cite[Lemma 3.4]{Ab12}, which I rewrite here for the convenience of the reader.
\begin{proof}
	Given a diffeomorphism $\psi$, we obtain an isomorphism of chain complexes
	\begin{equation*}
		\CW_b^*(L_0, L_1; \omega, H, J) \iso \CW_b^*(\psi(L_0), \psi(L_1), \psi^* \omega,\psi^*H, \psi^*(J)).
	\end{equation*}
	This comes from the composition of each chord from $L_0$ to $L_1$ with $\psi$ and the composition of each solution to equation (\ref{eq:CauchyRiemannEq}) with $\psi$; thereby obtaining an equation with $\psi^*J_t$. Since $\psi^*\omega=\rho \omega$, the Hamiltonian flow of $H\circ \psi$ with respect to $\psi ^* \omega$ agrees with the flow of $\frac{H\circ \psi}{\rho}$ with respect to $\omega$. Therefore we obtain an identification:
	\begin{equation} \label{eq:isomorphic_complexes_rescale}
	\CW_b^*(\psi(L_0),\psi(L_1); \psi^*\omega, H\circ \psi, \psi^*J_t) \iso \CW_b^*(\psi(L_0), \psi(L_1),\omega, \frac{H\circ \psi}{\rho}, \psi^*J_t).
	\end{equation}
	From now on, we omit the Floer data and define 
	\[\CW_b^*(\psi(L_0),\psi(L_1)) := \CW_b^*(\psi(L_0),\psi(L_1);\omega;\frac{H\circ \psi}{\rho}, \psi^*J_t).\]
\end{proof}

\subsection{Product Structure}\label{sec:product}

We would like to study the product operation on wrapped Floer complexes,
\begin{equation}\label{eq:product_structure}
	\CW_b^*(L_1,L_2; H)\otimes \CW_b^*(L_0,L_1; H) \to \CW_b^*(L_0, L_2;H).
\end{equation}

However, the naturally defined map would take value in $\CW_b^*(L_0;L_2; 2H)$.
And the usual continuation map from $\CW_b^*(L_0;L_2; 2H)$ to $\CW_b^*(L_0;L_2; H)$ is not well defined. To solve this problem, we appeal to the rescaling trick in \cite{Ab12}, where we first define a map
\begin{equation} \label{eq:premultiplication}
 \mu_2^{\psi^2}:	\CW_b^*(L_1,L_2;H)\otimes \CW_b^*(L_0,L_1;H) \to \CW_b^*(\psi^2(L_0),\psi^2(L_2); (\psi^2)^*(H)),
\end{equation}
where $\psi^2$ is the time-log$(2)$ Liouville flow. Secondly we compose with the inverse of equation (\ref{eq:isomorphic_complexes_rescale}) to get 
\begin{equation} \label{eq: product_pair_of_pants}
 \mu_2:	\CW_b^*(L_1,L_2;H)\otimes \CW_b^*(L_0,L_1;H) \to \CW_b^*(L_0,L_2; H).
\end{equation}
The map  (\ref{eq:premultiplication}) counts solutions to the following version of the Cauchy-Riemann equation:
\begin{equation}\label{eq:C-R_eqn_for_product}
	(du-X_S\otimes \gamma_s)^{0,1}=0,
\end{equation}
where $S$ is the surface obtained by removing $3$ points $\xi^0,\xi^1,\xi^2$ from the boundary of $\D^2$. In the above equation, $\gamma_S$ is a closed 1-form on $S$ while $X_S$ is the Hamiltonian vector field of $H_S$ on $M$, which depends on $S$. The Hamiltonians satisfy quadratic growth near cylindrical ends, the closed one-form $\gamma_s$ are translation invariant near the strip-like ends of the domain $S$ and the almost complex structures parametrized by the domain $J_S$ agrees with $J_t$ near strip-like ends; the details are in \cite[Section 3]{Ab12}.

We could thus go on to define the moduli space for $\mu^2$.
\begin{defin}\label{def:Floer_data_for_RR_2}
	The moduli space $\Disc(x_0; x_1, x_2)$ is the space of solutions to equation (\ref{eq:C-R_eqn_for_product}) with boundary conditions in $L_0, L_1, L_2$ and such that the image of $u$ converges to $x_{1}$, $x_2$ and $x_0$ at the corresponding strip-like ends.
\end{defin}

The standard arguments show that the Gromov compactification of the moduli space $\Disc_{2}(x_0;x_1,x_2)$  is obtained by adding the strata
\begin{align}
\label{eq:cover_boundary_disc_2}
& \coprod_{y \in \Chord(L_0,L_1)}  \Disc_{2}( x_0;y,x_2) \times \Discbar(y;x_1),   \\
& \coprod_{y \in \Chord(L_1,L_2)}  \Disc_{2}( x_0;x_1,y) \times \Discbar(y;x_2),  \\  \label{eq:cover_boundary_disc_2-end}
& \coprod_{y \in \Chord(L_0,L_2)}  \Discbar( x_0;  y) \times \Disc_{2}( y;x_1,x_2).
\end{align}
The above three cases correspond to the breaking of a holomorphic strip at each strip-like end of the the punctured disc. Note that since we are working with exact Lagrangians, there is no bubbling.
\begin{lem} \label{lem:regularity_compactness_2_inputs}
		For any fixed pair ($x_1,x_2$), the moduli space $\Disc_2(x_0; x_1, x_2)$ is empty for all but finitely many $x_0$, and for a generic family of almost complex structures $J_S$ and Hamiltonians $H_S$, $\Discbar_2(x_0; x_1, x_2)$ is a compact manifold of dimension $\mu(x_0)-\mu(x_1)-\mu(x_2)$, the boundary of which is the union of the codimension 1 strata listed in equations (\ref{eq:cover_boundary_disc_2})-(\ref{eq:cover_boundary_disc_2-end}).
\end{lem}

 Transversality is given by standard Sard-Smale arguments going back to \cite{FHS} in the case of Hamiltonian Floer cohomology; compactness of the moduli spaces is given by similar arguments of Lemma (\ref{cor: compactification_strip_moduli}), and is explicitly explained in  \cite[Lemma 3.2]{Ab08}.
Whenever the dimension of the moduli space $\Disc_2(x_0; x_1,x_2)$ is $0$, it consists of finitely many elements, each of which defines an isomorphism
\begin{equation}
o_{x_1}\otimes o_{x_2} \to o_{x}.
\end{equation}
We also denote by $\mu_u$ this isomorphism induced on orientation lines; now we are set to define a product on the wrapped Floer complex as follows:
\begin{equation}
	\begin{aligned}
	\mu_2: \CW_b^*(L_2,L_1) \otimes \CW_b^*(L_1,L_0) &\to \CW_b^*(L_2, L_0)\\
    \mu_2([x_1],[x_2]) &= \sum_{ \substack{ |x_0|  = |x_1| + |x_2| \\  u \in \Disc_{2}( x_0;x_1,x_2)} } (-1)^{\mu(x_1)}\mu_u([x_2],[x_1]).
	\end{aligned}
\end{equation}

\subsection{\texorpdfstring{$A_\infty$}{A-infinity} structure}\label{sec:A_inftysubsection} 
\begin{figure} \label{fig:Three_puncture}
	\centering
	\includegraphics{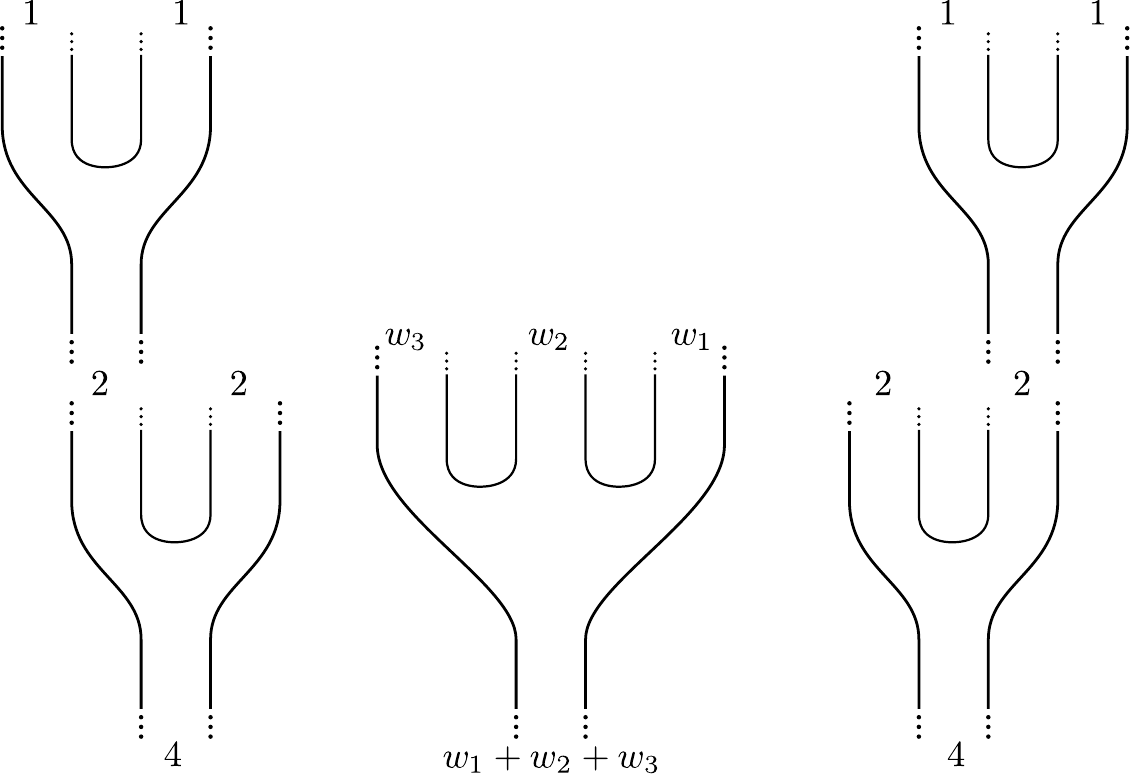}
	\caption{ }
	
\end{figure}
We want to equip each strip-like end with Floer data that would facilitate the definition of an $A_\infty$ structure on the wrapped Fukaya category. First, denote by $\RR_d$ the moduli space of disks that is equipped with $d$ incoming boundary punctures $\{\xi^k\}_{k=1}^d$, ordered clockwise and a outgoing boundary puncture $\xi^0$. Denote by $\RRbar_d$ the Deligne-Mumford compactification of this moduli space; as in \cite[Section 9]{seidel-book}, we fix a \textit{universal and consistent} choice of strip-like ends. In other words, for each surface $S$ and each puncture, there exists a map
\begin{equation*}
	\epsilon^k: Z_{\pm} \to S,
\end{equation*}
whose source is $Z_-$ if it is the outgoing puncture, and $Z_+$ otherwise, and such a choice varies smoothly in the interior of $\RRbar_d$. Each boundary stratum $\sigma$ of $\RRbar_d$ has the following form:
\begin{equation*}
\sigma=\RR_{d_1} \times \RR_{d_2}\times \cdots \times \RR_{d_j}.
\end{equation*}
A neighborhood of $\sigma$ looks like gluing of strip-like ends chosen on these lower dimensional moduli spaces:
\begin{equation*}
 \sigma \times [1,\infty)^{j-1} \to \RR_d.
\end{equation*}
Let us assign to each end $\xi^i$ of a punctured disc a weight $w_i$, which is a positive real number, with the convention that any $\dbar$ operator on such a disc must pull back under $\epsilon^i$ to the Cauchy Riemann equation (\ref{eq:CauchyRiemannEq}) up to applying $\psi^{w_{i}}.$ In the product case where there are two incoming ends and one outgoing end, the inputs have weights 1 while the output has weight 2. 

If we have three inputs and one output, the moduli space of all such possible configurations $\RRbar_3$ is actually an interval whose two ends are represents by nodal discs obtained by gluing  the two possible different ways as shown in Figure (\ref{fig:Three_puncture}), two discs with two incoming ends of weight 1 and the other with weight 2.
In the previous section (\ref{sec:product}), we defined the product (\ref{eq: product_pair_of_pants}) to be the composition of (\ref{eq:premultiplication}) with rescaling. Instead of rescaling first and then take composition, one pulls back by $\psi^2$ of all the data used to define $\Disc(x_0; x_1, x_2)$, and thus obtain a new $\dbar$ operator where the weights are now $2$ at the incoming ends and $4$ at the outgoing end. Then one composes $\CW_b^*(\psi^2(L_3), \psi^2(L_2))$ with $\CW_b^*(\psi^2(L_2), \psi^2(L_0) )$ after applying (\ref{eq:premultiplication}) with $L_0,L_1,L_2$. In this way we define the product with three inputs directly:
\begin{equation}\label{eq:multiplication_disc_with_three_inputs}
\mu_3: \CW_b^*(L_3,L_2) \otimes \CW_b^*(L_2,L_1) \otimes \CW_b^*(L_1, L_0) \to \CW_b^*(L_3,L_0).
\end{equation}

\begin{rem}
	$\mu_2$ is not associative at the chain level. Yet the two ways of composition give the same element after taking cohomology on $\CW_b^*(L_3,L_0)$.
\end{rem}

Similar to the above construction, higher products
\begin{equation}\label{eq: Higher_Prod}
	\CW^*(L_d, L_{d-1}) \otimes \CW^*(L_{d-1}, L_{d-2}) \otimes \cdots \otimes \CW^*(L_1,L_0) \to \CW^*(L_d, L_0),
\end{equation}
come from the count of virtual dimension zero moduli spaces $\Disc_d(x_0; \vec{x})$ of solutions of Equation \eqref{eq:C-R_eqn_for_product}.

%

%

If we don't take into account the strip-breaking phenomenon at the ends, the virtual codimension $1$ strata of the Gromov compactification $\Discbar_{d}(x_0; \vec{x})$ lie over the codimension $1$ strata of $\RRbar_{d}$. The consistency condition on the space of Floer data implies that whenever a disc breaks, each component is a solution to equation (\ref{eq:C-R_eqn_for_product}) for the Floer data up to applying $\psi^C$ for some constant $C$ that depends on the modulus in $\RRbar_{d}$. Such rescaling identifies the solutions of rescaled equation with moduli space of the solutions to the original equation, so we conclude that for each integer $k$ between $0$ and $d-d_2$ and chord  $y \in \Chord(L_{k+1}, L_k)$, we get a natural inclusion:
\begin{equation} \label{eq:codim_1_strata_discs_1_output}
\Discbar_{d_1}(x_0; \vec{x}^1)  \times   \Discbar_{d_2}(y; \vec{x}^2) \to  \Discbar_{d}(x_0; \vec{x})
\end{equation}
where the sequences of inputs in the respective factors are given by $\vec{x}^2= (x_{k+1},\cdots x x_{k+d_2})$ and $\vec{x}^1= (x_1, \cdots , x_k ,y, x_{k+d_2+1}, \cdots, x_d)$.

Thus similar to Lemma (\ref{lem:regularity_compactness_2_inputs}), we get the following lemma.
\begin{lem}\label{lem:boundary_moduli_spaces_a_infty}
	The moduli space $\Discbar_{d}(x_0; \vec{x})$ is compact and is empty for all but finitely many $x_0$ once the inputs $\vec{x}$ are fixed. For a generic choice of Floer data, they form manifolds of dimension 
	\[ |x_0| +d-2-\sum_{1\leq k\leq d} |x_k|,
	\] whose boundary is covered by the images of inclusions (\ref{eq:codim_1_strata_discs_1_output}).\qed
\end{lem}
Whenever $|x_0|=x-d + \sum_{1\leq k\leq d}|x^k|$, there are therefore only finitely many elements of $\Discbar_{d}(x_0; \vec{x})$. Every such element $u \in \Discbar_{d}(x_0; \vec{x}): S \to M$ induces an isomorphism:
\begin{equation}\label{key}
	\ro_{x_d} \otimes \cdots \otimes \ro_{ x_1} \to \ro_{x_0}.
\end{equation} 
Let us write $\mu_u$ for the induced map on orientation lines, omitting composition with $\CW^*(\psi^{w_{k,S}})$ or its inverse from the notation. We define the $d$th higher product:
\begin{equation}
	\mu_d :  CW^{*}_{b}(L_{d-1}, L_d)    \otimes \cdots \otimes    CW^{*}_{b}(L_1, L_2)  \otimes  CW^{*}_{b}(L_0,L_1 )  \to   CW^{*}_{b}( L_0, L_d)
\end{equation}
Explicitly, 
\begin{equation}
\mu_{d}([x_d], \ldots, [x_1])  = \sum_{\substack{|x_0| =  2 - d + \sum_{1 \leq k \leq d} |x_k| \\ u \in \Disc_{d}(x_0; \vec{x}) }} (-1)^{\dagger}  \mu_{u}([x_d], \ldots,  [x_1]),
\end{equation}
where the sign is given by 
\begin{equation} \label{eq:dagger_sign}
\dagger = \sum_{k=1}^{d} k |x_k|.
\end{equation}
This is just saying we count the elements of $\Discbar_{d}(x_0; \vec{x})$ with the right signs; the derivation of $\dagger$ can be found in \cite[Chapter 12]{seidel-book}

If we now consider the dimension 1 moduli spaces, Lemma \eqref{lem:boundary_moduli_spaces_a_infty} claims that their boundaries are the strata in equation \eqref{eq:codim_1_strata_discs_1_output} which are rigid and thus correspond to the composition of operations $\mu_d$. Taking into account of signs, we get the following proposition:
\begin{prop}\label{prop:a_infty_structure}
The operations $\mu_d$ define an $A_\infty$ structure on the category $\CW^*(M)$. Explicitly, we have 
\begin{equation}
 \label{eq:a_infty_property}
 \sum_{\substack{d_1 + d_2 = d +1 \\ 0 \leq k <d_1}} (-1)^{\maltese_{1}^{k}} \mu_{d_1}\left(x_d, \ldots, x_{k+d_2+1}, \mu_{d_2}(x_{k+d_2}, \ldots, x_{k+1}) ,   x_k, \ldots , x_1 \right) = 0,
\end{equation}
where the sign is given by
\begin{equation*}\label{Maltese_sign}
	\maltese_1^k=k+ \sum_{1\leq j\leq k}|x_j|.
\end{equation*}
\qed
\end{prop}

\section{The Pontryagin Category}\label{sec: Pontryagin_Cat}

Let $\LL$ be any path connected topological space. Consider the topological category with objects the points of $\LL$, and morphisms from $L^0$ to $L^1$ is given by the Moore path space
\begin{equation}
\Omega(L^0,L^1) \equiv \{\gamma: [0,R] \to \LL|   \gamma(0)= L^0, \gamma(1) = L^1\}
\end{equation}
Here $R$ is allowed to vary between $0$ and infinity including infinity. 
The composition is given by concatenation:
\begin{equation}
	\begin{aligned}
	\Omega(L^0, L') \times \Omega(L', L^1) \to &\Omega(L^0, L^1),\\
	(\gamma_1, \gamma_2)\to \gamma_2 \cdot \gamma_1(l) &\equiv 
	\begin{cases}
	\gamma_1(l) &\textrm{ if } 0 \leq l \leq R_1 \\
	\gamma_2(l-R_1) &\textrm{ if } R_1 \leq l \leq R_1 + R_2,
	\end{cases}
		\end{aligned}
\end{equation}
where $\gamma_i$ has domain $[0,R_i]$.

This formula defines an associative composition of paths. In order for this operation to induce the structure of a DGA on chains, we use cubical chains as in \cite{Ab12} instead of singular chains. As in \cite{Mas} we define a map from a cube to a topological space to be \textit{degenerate} if it factors through the projection to a face. The underlying vector space for the cubical chain complex is:
\begin{equation}\label{eq:cubical_chains}
	C_i(X)=\frac{\Z [\text{Map}([0,1]^i, X)]}{\Z[\text{degenerate maps}]}.
\end{equation}
We write $\delta_{k,\epsilon}$ to denote the inclusion of the codimension -$1$ face where the $k$th coordinate is equal to $\epsilon \in \{0,1\}$. The differential of $C_*(X)$ is defined as 
\begin{equation}
	\del \sigma = \sum_{k=1}^i \sum_{\epsilon=0,1}\del_{k,\epsilon} \sigma = \sum_{k=1}^i \sum_{\epsilon=0,1}(-1)^{k+\epsilon} \sigma\circ \del_{k,\epsilon}.
\end{equation}
Since the product of cubes is still a cube, it is easy to define the following map
\begin{equation}
	C_*(X) \times C_*(Y) \to C_*(X\times Y),
\end{equation}
which may easily checked to be associative. Now we apply this to the path space $\Omega_\LL$, and obtain a DG category. Let 
\begin{equation}\label{def:Pontryagin_category}
	\Moore(\LL)
\end{equation} denote the differential graded category whose objects are points of $\LL$ and morphism spaces
\begin{equation}
	\text{Hom}_*(L^0,L^1)=C_*(\Omega(L^0,L^1)).
\end{equation}
The differentials and the products are given by
\begin{align}
	\mu^P_1 \sigma& = \del \sigma,\\
	\mu^P_2(\sigma_2, \sigma_1)&= (-1)^{|\sigma_1|}\sigma_2 \cdot \sigma_1.
\end{align}
We may consider this as an $A_\infty$ category by defining $\mu^P_d = 0$ for $d >2$.
Notice morphisms in this category satisfy $A_\infty$ relations:
\begin{equation}
	\sum_{m,n}(-1)^{\maltese_n} \mu^P_{d-m+1}(\sigma_d,\cdots,\sigma_{n+m+1},\mu^P_m(\sigma_{n+m}, \cdots \sigma_{n+1}), \sigma_n,\cdots, \sigma_1)=0,
\end{equation}
namely \[\mu_1(\mu_2(\sigma_2,\sigma_1))+\mu_2(\sigma_2,\mu_1(\sigma_1))+(-1)^{|\sigma_1|+1}\mu_2(\mu_1(\sigma_2), \sigma_1)=0
\] 
in this situation.
\begin{rem}
Since we assumed that the space $\LL$ is path connected, every object is quasi-isomorphic to any other object. In particular, if $L \in \LL$ we may fully faithfully embed the $A_\infty$ category with the single object  $L$ and morphism $C_{*}(\Omega(L, L)$ to the Pontryagin category:
\begin{equation}
(L,C_{*}(\Omega(L, L)) )\to \Moore(\LL).
\end{equation}
\end{rem}
\begin{rem}
 In the next section $\LL$ is the space $\tLag_b$ of exact Lagrangians with Pin structure relative to $b$ as in Equation (\ref{eq: $Lag_b$}); we obtain the differential graded category $\Moore(\tLag)$ from the above constructions, from which we would like to construct a functor to the wrapped Fukaya category of the ambient symplectic manifold in the following sections.
\end{rem}
\begin{rem}
	The functor $f$ constructed by Abouzaid in \cite{Ab12} is mapped into the triangulated closure of the image of $\Moore(\LL)$ in its category of modules under the Yoneda embedding. A practical model is the category of twisted complexes (\cite[Section 3l]{seidel-book}), where the objects are direct sums of shifted objects in $\Moore(\LL)$ and the differentials incorporate the morphisms between the shifted objects in the same direct sum. However, in our situation, since the loops of Lagrangians are generated by cotangent fibers of loops on the zero section of a cotangent bundle, $f$ just maps each loop of Lagrangians back to the loop on the zero section.
\end{rem}

\section{Construction of the functor}\label{sec:Functor}

Suppose we are given a open Liouville manifold $(M, \omega)$ that comes from the completion of a Liouville domain $M^{in}$,  with cylindrical end $\del M^{in} \times (0, \infty)$. From the two previous sections, we obtain its wrapped Fukaya category $\CW^*(M)$, with operations $\mu_d$ satisfying the $A_\infty$ relations \eqref{prop:a_infty_structure},
and the differential graded category $\Moore(\tLag)$ of exact Lagrangians with Legendrian boundaries with morphisms $\mu_1^{\Moore}, \mu_2^{\Moore}$.
Our goal in this section is to construct a functor from $\Moore(\tLag)$ to $\CW^*(M)$ that also preserves the $A_\infty$ structures. This is just Theorem \eqref{thm: main theorem in introduction} in the introduction. We write it here for the reader's convenience.
\begin{thm}\label{main theorem}
	For any open Liouville manifold $(M, \omega)$ with cylindrical ends, there exists an $A_\infty$ functor $\F$ from the triangulated closure of the Pontryagin category $\Moore(\tLag)$ of exact Lagrangians with Legendrian boundaries decorated with a Pin structure to the wrapped Fukaya category $\CW^*(M)$ of the Liouville manifold:
	\begin{equation}
		\Moore(\tLag) \xrightarrow{\F}  \CW^*(M).
	\end{equation}
On the level of objects, the functor sends every graded Lagrangian with Pin structure $\LL=(L, P^\#)$ to itself; on the level of homomorphisms, we have:
\begin{align*}
Mor(\Moore(\tLag_b)) &\xrightarrow{\mathcal{F}} \CW(M), \\
C_{*}(\Omega_{\LL_1,\LL_0} \Lag_b) &\mapsto \CW^{-*}(\LL_1,\LL_0;H,J)
\end{align*}
where the pair $(H, J)$ is a generic choice of Floer data that makes the moduli space of solutions to Floer equations regular.
\end{thm}
In the above theorem we have suppressed the notation of $(L,,P^\#)$ of Pin structures in the notations above when it is clear.

We shall give a proof of this theorem in the following sections.
\begin{rem}
	The above theorem generalizes the functors in \cite{AbSc10} and \cite{Ab12} to the based loop space of Lagrangian embeddings with Legendrian boundary.
\end{rem}

\subsection{Chain map}\label{sec:chain_map}
Before proving Theorem (\ref{main theorem}), let us consider the geometric set up of the domain.
 We work over the upper half plane $\bb{H}$ with 2 marked points  $z_1, z_2$ on the boundary, which is bi-holomorphic to $U=\bb{D}^2\setminus \{1 \text{ point on } \del D^2 \}$, let $\xi$ denote the punctured point on the boundary of $U$. The group of automorphisms of $\HH$ that preserve the point at infinity is the intersection of $PSL(2;\R)$ and the affine transformations Aff $(\R^2)$, namely $z \mapsto az+b$ where $a\in \R_{>0}, b\in \R$. Thus after an automorphism, we may assume $z_1=(0,0), z_2=(1,0)$. 

To have a proper definition of a moduli space of pseudo-holomorphic curves, we work over the space $\JJ(M)$ of almost complex structures of contact type along cylindrical ends of $M$ defined in Section (\ref{sec:almost_complex_structure_in_general}). Furthermore, we allow more freedom by taking a $k$-dimensional family of almost complex structures so we consider the set $\JJ_{U, k}(M)$, which denotes the space of families of smooth almost complex structures of contact type on $M$ parametrized by $U\times I^k$ and such that for any $\tau \in I^{k}$,
\[
J_{\tau, z} \in \JJ(M)
\] 
only depends on $t$ when $s\to \infty$ at the strip-like end $\epsilon^0: Z_{-} = (-\infty, 0]\times [0,1] \to U$ near the puncture $\xi$ parametrized by the $s, t$ coordinates and that these almost complex structures near the strip-like end $\epsilon^0$ is independent of $ \tau$; in other words, there is a $1$-parameter family of almost complex structure $J_t$ parametrized by $t \in [0,1]$ such that for every $\tau$,
\begin{equation}\label{eq:almost_complex_structure_at_strip_like_ends}
	J_{\tau, \epsilon^0(s,t)} = J_t \text{ when $s \ll 0$}.
\end{equation}
We also require that this $k$-family of almost complex structures satisfies the following: for any family $\{u_\tau\}: U \to M$ that are solutions of the Cauchy-Riemann Equation (\ref{eq:CauchyRiemannEq}) such that $u_\tau$ sends the boundary of $U$ between $z_1$ and $z_2$ to Lagrangians in the path given by $\sigma(\tau)$ (the exact condition is (\ref{eq:CR_for_upper_half_plane}) below), the corresponding linearized $\dbar$ operator for $J_{\tau, z}$ is surjective.
This requirement can be achieved inductively on $k$ (the dimension of the cube): for a single point we just pick a generic $J_z$ by standard arguments of transversality as in \cite[chapter 3]{MS04}; for higher $k$ suppose we have picked $J_{\tau, z}$ for one $k-1$ chain in the boundary strata of the family parametrized by $I^k$, we may simply extend it smoothly to other $k-1$ chains of $\del \sigma$, extend to a neighborhood of $\del \sigma \simeq \del I^k$ and then extend to the whole $k$-chain by contractibility of $\JJ(M)$. Note that during the extensions we can modify the almost complex structures $J_{\tau,z}$ via a family of infinitesimal deformations $\{K_\tau\}$ where $K_\tau \in End(TM, J_\tau, \omega)$ by taking almost complex structures of the form $J_{\tau}\exp(-J_{\tau}K_\tau)$ as in Equation \eqref{eqn:almost_complex_structure_by_perturbation} to make the linearized operator surjective.

The above space of almost complex structures has a subspace that is a Banach manifold as we can give a local neighborhood of every element $J \in \JJ_{U, k}$ by taking the exponential map of a small infinitesimal deformation in a suitable Banach space: $K \in C^\infty_{\vec{\epsilon}, M }(I^k, End(TM, J, \omega))$ of all sections such that the norm $||K\tau||_{\vec{\epsilon}}$ is finite; here the above norm $||\cdot||_{\vec{\epsilon}}$ is defined in \cite{Fl88} as
\begin{equation}\label{norm: Floer}
	||K_\tau||_{\vec{\epsilon}} := \sum_{l = 1}^{\infty}\epsilon_l ||K_\tau||_{C^l}.
\end{equation}
When the sequence $\epsilon_l$ converges to zero sufficiently rapidly, there exists sections $K \in C^\infty_{\vec{\epsilon}, M }(I^k, End(TM, J, \omega)) $ which are supported in arbitrarily small neighborhood of any point $x \in M$ and $\tau \in I^k$ and take arbitrary small values.

We denote a $k$ family of $U$-parametrized almost complex structures that satisfies the above conditions by $J_{\tau, z}$ or $J_{\tau}$ if there is no confusion.

We would like to consider a map:
\[ H :  U \to \Cal{H}(M),
\]
satisfying quadratic growth at infinity as in (\ref{eq:quadratic_Hamiltonian}). Let $X_U$ denote the Hamiltonian flow of $H$, and consider the following Cauchy-Riemann equation:
\begin{equation} \label{eq:CR_for_upper_half_plane}
	J_{\tau} (du-X_U\otimes \gamma) - (du-X_U\otimes \gamma)\circ j_U = 0
\end{equation}
for maps $u: U \to M$. 


Given $\sigma$ representing an element in $C_{k}(\Omega_{L_0, L_1} \tLag)$, namely $\sigma: [0,1]^{k} \to \Omega_{L_0, L_1} \tLag$, for any $\tau \in I^{k},$ where $\sigma(\tau)$ is a path of Lagrangians $[0,R] \to \Moore_{L_0, L_1} \tLag$ between $L_0, L_1$ for every $\tau \in I^{k}$, let $|\sigma(\tau)|$ denote the length of this path; we also fix $x \in \Chord(\LL_0, \LL_1)$.
	
	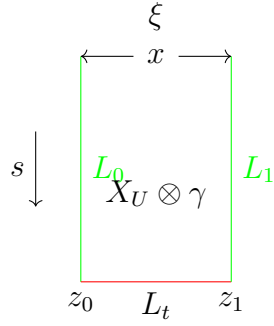
\begin{figure}\label{fig: upper_half_plane}
		\begin{tikzpicture}
		\draw[color = green](0,0) to (0,3);
		\draw[color = red](0,0) to (2,0);
		\draw[color = green](2,0) to (2,3);
		\coordinate[label = below: $z_0$]($z_0$) at (0,0);
		\coordinate[label = below:$z_1$]($z_1$) at (2,0);
		\coordinate[label = right: \textcolor{green}{$L_0$}]() at (0, 1.5);
		\coordinate[label = right: \textcolor{green}{$L_1$}]() at (2,1.5);
		\coordinate[label = below: $L_t$]() at (1,0);
		\coordinate[label = above: $\xi$]() at (1,3.2);
		\coordinate[label = below: $X_U\otimes \gamma$]() at (1, 1.5);
		\coordinate[label = center: $x$]() at (1,3);
		\draw[<-] (0,3) to (0.7,3);
		\draw[->](1.3,3) to (2,3);
		\draw[->](-.6, 2) to (-.6, 1);
		\coordinate[label = left:$s$] () at (-.6, 1.5);
		
		\end{tikzpicture}
		\caption{Upper half plane}
	\end{figure}
	 
	 \begin{defin}[Moduli space of upper half plane]\label{def:moduli_of_Disc_3}
	 	 We define the moduli space of pseudo-holomorphic maps with moving Lagrangian boundary $\CM(x, \sigma, H, J_{\tau})$ to be pairs $(u, \tau)$ where $\tau \in I^{k}$, and $u: U\to M$ satisfying the following condition:
	\begin{equation} \label{eqn: conditions_for_Disc_3}
	\left \{
	\begin{aligned}
	& (du-X_U\otimes \gamma)\circ j_U - J_{\tau, z} (du - X_U\otimes \gamma)=0, \\
	& u(z) \in L_0 \text{ if $z \in \del U$ lies on the segment between $\xi$ and $z_1$ },\\
	& u(z) \in L_1 \text{ if $z \in \del U$ lies on the segment between $\xi $ and $z_2$},\\
	& u(z) \in L_{\sigma(\tau)(\frac{t}{|\sigma(\tau)|})} \text{ if $z=(t,0) \in \del U$ lies between $z_1$ and $z_2$, }\\
	 &\lim_{s \rightarrow + \infty} u(\epsilon^0(s,t )) = x(t), \text{(asymptotic condition)}\\
	\end{aligned}
	\right \}
	\end{equation}
\end{defin}
Standard arguments involving Sard-Smale and transversality results in Floer theory give:
\begin{lem}
	For generic choices of Hamiltonians satisfying Equation (\ref{eq:quadratic Hamiltonian}) and a generic family of almost complex structures $J_{\tau} \in \JJ(M)$, $\CM(x,\sigma, H, J_{\tau})$ is a smooth manifold of dimension $\mu(x)+|\sigma|$, where $|\sigma| = k$ denote the dimension of the cubical chain.
\end{lem}


From the standard transversality results in Floer theory we realize that the dimension of this moduli space is independent of $H,J_{\tau}$ when it is regular, thus we can omit it and write $\CM(x,\sigma)$ when there is no confusion. 

Fixing a choice of $H, J_{\tau}$, we would like to study the boundary of the Gromov compactification of $\Disc(x,\sigma,H,J_{\tau})$.
Since there are three special points on the boundary of the domain (two markings and one puncture with a strip-like end), there is no modulus. If $\tau$ is fixed, the only strata that we need to add to the Gromov-compactification for the moduli space with boundary condition given by $\sigma(\tau)$ are obtained by considering breakings of strips at the ends. This gives a map:
\begin{equation}\label{eq:boundary component by breaking at ends for Disc_3}
	\Disc(\sigma; y) \times \Disc(y; x)  \to \Discbar(\sigma; x),
\end{equation}
for all possible $y \in \Chord(\LL_0, \LL_1)$. 

On the other hand, we also need to take into consideration of the boundary strata $\sigma' \subseteq \del \sigma$ of the $k$-dimensional cubical chain $\sigma$, which gives:
\begin{equation} \label{eq: boundary component by boundary of cubical chain for Disc_3}
	\Disc(\sigma'; x) \to \Discbar(\sigma; x).
\end{equation}

The standard compactness result gives the components of the boundary strata of the compactification $\Discbar(\sigma; x)$ in the following proposition.
\begin{prop}\label{prop:boundary strata of Discbar_3}
	The Gromov compactification $\bar{\Disc}(x; \sigma)$ is a compact manifold whose boundary is covered by the images of the following strata.
	
	\begin{align} \label{eq:3_pointed_disc_top_break}
	&\coprod_{y \in \Chord(L_0,L_1)} \Disc(y; \sigma) \times \Disc(x;y), \\
	&\coprod_{\sigma' \subseteq \del \sigma} \Disc(x; \sigma').\label{eq:3_pointed_disc_chain_break}
	\end{align}\qed
\end{prop}

\subsection{The linear term.} Consider the following evaluation map:
\begin{equation}\label{eq:evaluation_3_pointed_disc}
ev: \Discbar(x;\sigma) \to \Chord(L_0, L_1),
\end{equation}
which takes every punctured disc $u$ to its infinite end $x$, a Hamiltonian chord between $L_0$ and $L_1$. If we consider the boundary strata (\ref{eq:3_pointed_disc_top_break}), we obtain a commutative diagram:
\begin{equation}\label{eq: evaluation on neck breaking boundary}
\begin{tikzcd}
\Disc(y;\sigma) \times \Disc(x;y) \ar{d} \ar{r} & \Discbar(x; \sigma) \ar{d}{ev}\\
\Disc(x;y) \ar{r} &\Chord(L_0, L_1).
\end{tikzcd}
\end{equation}
The first left vertical arrow is projection to the second factor and the top map is the inclusion of a boundary stratum. 

Meanwhile, if we consider the boundary strata (\ref{eq:3_pointed_disc_chain_break}), we get the following commutative diagram:
\begin{equation}\label{eq: evaluation on interior}
\begin{tikzcd}
\Disc(x; \sigma) \ar{dr} \ar{r}& \Discbar(x;\sigma) \ar{d}\\
&\Chord(L_0,L_1).
\end{tikzcd}
\end{equation}
\begin{lem}\label{lem:chain_map}
	There exist fundamental cycles $[\Discbar(x;\sigma)] \in C_*(\Discbar(x;\sigma))$ such that the assignment
	\begin{equation}
		\F^1(\sigma)= \bigoplus_{x} (-1)^{\dag } ev_*([\Discbar(x;\sigma)])
	\end{equation} where $\dag = \mu(x) + \mu(\sigma)$, defines a chain map in \eqref{main theorem}.
\end{lem}
Consider moduli spaces $\Discbar(x;\sigma)$ whose boundary only has codimension 1 strata, which must be products of closed manifolds. By taking the product of the fundamental chains of the factors in each boundary stratum, we obtain a chain in $C_*(\Discbar(x; \sigma))$:
\begin{multline}\label{eq:boundary_strata}
	(-1)^{\mu(x) + \mu(\sigma)}\sum_{x' \in \Chord(L_0,L_1)} [\Disc(x';\sigma)] \times [\Disc(x, x')] 
	+ \sum_{\sigma' \subseteq \del \sigma}[\Disc(x;\sigma')].
\end{multline} 
We now define the fundamental chain 
\[
[\Discbar(x;\sigma)]\]
to be any chain whose boundary is the sum of chains represented by equation (\ref{eq:boundary_strata}).

\begin{proof}[Proof of Lemma \ref{lem:chain_map}]
	Consider $\sigma: I^k \to \Omega_{\LL_0,\LL_1} \tLag$, fix $x \in \Cal{CW}^{-k}(\LL_1, \LL_0;H,J)$ for some Hamiltonian $H$ that is quadratic with respect to the radial coordinate $|p|$ at cylindrical ends of $M$ outside a compact set. 
We know that, up to signs which we shall ignore, proving that $\F^1$ is a chain map is equivalent to proving that 
\begin{equation}
	\mu^1_\CW\circ \F^1(\sigma)=\F^1 \circ \mu_1^{Tw}(\sigma). 
\end{equation}
From the two previous sections we know $\mu^{Tw}_1(\sigma)$ is just $\bigcup_{\sigma' \subseteq \del \sigma}\sigma'$ and $\mu^1_\CW(x)$ is a sum of Hamiltonian chords of Maslov index 1 higher than that of $x$ as shown in Equation \eqref{eq: wrap_differential}.
Going through the stratification of $\del \Discbar(x; \sigma)$, we find that the strata \eqref{eq:3_pointed_disc_top_break} for rigid moduli spaces $\Disc(x; y)$ correspond to $\mu^1_\CW\circ \F^1(\sigma)$ and that the strata \eqref{eq:3_pointed_disc_chain_break} correspond to $\F^1 \circ \mu_1^P(\sigma)$.
	Thus $\F^1$ satisfies the first $A_\infty$ relationship, namely it it a chain map.

\end{proof}

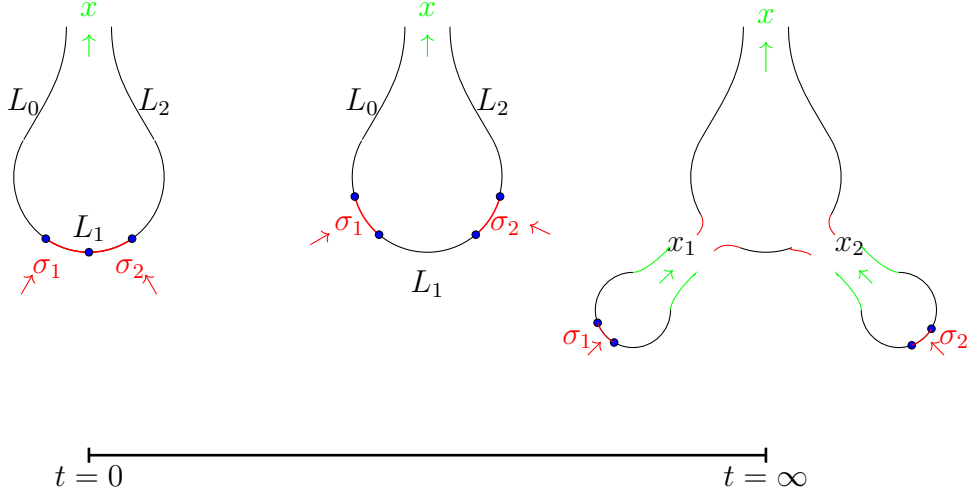
\begin{figure}\label{fig:2-product homotopy}
	\centering
	\begin{tikzpicture}
	\newcommand*{\radius}{1}
	\newcommand*{\tinyradius}{.05}
	\newcommand*{\bigradius}{1.5}
	\newcommand*{\smalldash}{.1}
	\newcommand*{\length}{3*\bigradius}
	\draw [ shift=(30:\radius)] (0,0) arc (29:-14:\radius);
	\draw[red,line width=0.6,shift=(-16:\radius)] (0,0) arc (-16:-49:\radius);
	\draw[shift=(-51:\radius)] (0,0) arc(-51:-129:\radius);
	\draw[red,line width=0.6,shift=(-131:\radius)] (0,0) arc (-131:-164:\radius);
	\draw[shift=(-166:\radius)] (0,0) arc(-166:-209:\radius);
	\draw (30:\radius) to [out= 120, in=-90] (.3*\radius, 2*\radius);
	\draw (-210:\radius) to [out=60, in =-90](-.3, 2);
	\coordinate[label=above:\textcolor{green}{$x $}] () at (0, 2);
	\coordinate[label=right:$L_2$]() at (.5, 1);
	\coordinate[label=left:$L_0$]() at (-.5,1);
	\coordinate[label=below:$L_1$]() at (0, -1.1);
	\coordinate[label=below: \textcolor{red}{$\sigma_1$}]() at (-160:1.1*\radius);
	\coordinate[label=below:\textcolor{red}{$\sigma_2$}]() at (-20:1.1*\radius);
	\draw[green, ->](0,1.6*\radius)--+(0,0.3 *\radius);
	\draw[red,<-] (-25:1.5*\radius)--+(-25:0.3cm);
	\draw[red,<-] (-150:1.5*\radius) --+(-150:0.3 cm);
	\draw [fill=blue](-15:\radius) circle (\tinyradius);
	\draw[fill=blue](-50:\radius) circle (\tinyradius);
	\draw[fill=blue](-165:\radius) circle(\tinyradius);
	\draw[fill=blue](-130:\radius) circle (\tinyradius);

	\begin{scope}[shift={(-\length, 0)}]
	\draw [shift=(29:\radius)](0,0) arc (29:-209:\radius);
	\draw (30:\radius) to [out=120, in=-90] (.3* \radius, 2*\radius);
	\draw (-210:\radius) to [out=60, in =-90](-.3*\radius, 2*\radius);
	\draw [red,line width=0.6, shift=(-55:\radius)] (0,0) arc (-55:-125:\radius);
	\draw [fill=blue](-55: \radius) circle (\tinyradius);
	\draw[fill=blue](-125:\radius) circle (\tinyradius);
	\draw[fill=blue](-90:\radius) circle (\tinyradius);
	\coordinate[label=above:\textcolor{green}{$x $}] () at (0, 2);
	\coordinate[label=right:$L_2$]() at (.5, 1);
	\coordinate[label=left:$L_0$]() at (-.5,1);
	\coordinate[label=above:$L_1$] () at (-90:\radius);
	\coordinate[label=below: \textcolor{red}{$\sigma_1$}]() at (-120:1.1*\radius);
	\coordinate[label=below:\textcolor{red}{$\sigma_2$}]() at (-60:1.1*\radius);
	\draw[green, ->](0,1.6*\radius)--+(0,0.3 *\radius);
	\draw[red,<-] (-60:1.5*\radius)--+(-60:0.3cm);
	\draw[red,<-] (-120:1.5*\radius) --+(-120:0.3 cm);
	\end{scope}
	
	\begin{scope}[shift={(\length,0)}]
	\newcommand*{\smallradius}{0.5*\radius}
	\draw [shift=(-30:\radius)](0,0) arc (-30:30:\radius);
	\draw [shift=(150:\radius)](0,0) arc (150:210:\radius);
	\draw [shift=(-70:\radius)](0,0) arc (-70:-110:\radius);
	\draw (30:\radius) to [out=120, in=-90] (.3* \radius, 2*\radius);
	\draw (-210:\radius) to [out=60, in =-90](-.3*\radius, 2*\radius);
	\coordinate (p1) at (220:1.2*\radius);
	\coordinate (p2) at (235:1.2*\radius);
	\coordinate (q1) at (-40:1.2*\radius);
	\coordinate (q2) at (-60: 1.2*\radius);
	\draw [red](210:\radius) to [out=-60, in =45](p1);
	\draw [red](250:\radius) to [out=160, in =45](p2);
	\draw [red](-30:\radius) to [out= -120, in = 135](q1);
	\draw [red](-70:\radius) to [out=-160, in =135](q2);
	\coordinate [red, label= left:$x_1$]() at (230:1.2*\radius);
	\coordinate [red, label= right:$x_2$]() at (-50:1.2*\radius);
	\coordinate [label=above:\textcolor{green}{$x$}]() at (0, 1.9*\radius);
	\draw [green, ->] (0,1.4*\radius) to (0,1.8*\radius);

	\begin{scope} [shift ={(-45:2.5*\radius)}]
	\draw [shift=(90:\smallradius)](0,0) arc (90:-180:\smallradius);
	\draw [fill=blue] (-30:\smallradius) circle (\tinyradius);
	\draw [fill=blue] (-70:\smallradius) circle (\tinyradius);
	\draw [red](-30:\smallradius) arc (-30:-70:\smallradius);
	\coordinate[label= right: \textcolor{red}{$\sigma_2$}] () at (-45:1.1*\smallradius);
	\draw[red, ->] (-45:1.7*\smallradius) --+(135:0.5*\smallradius);
	\draw [green] (90:\smallradius) [out=180, in=135]to (-0.9*\smallradius, 1.6*\smallradius);
	\draw[green](-180:\smallradius) [out=90, in =135] to (-1.6*\smallradius, 0.9*\smallradius);
	\draw [green, ->](135:\smallradius) to (135:1.5*\smallradius);
	\end{scope}
	
	\begin{scope}[shift={(-135:2.5*\radius)},rotate=-90]
	\draw [shift=(90:\smallradius)](0,0) arc (90:-180:\smallradius);
	\draw [fill=blue] (-30:\smallradius) circle (\tinyradius);
	\draw [fill=blue] (-70:\smallradius) circle (\tinyradius);
	\draw [red](-30:\smallradius) arc (-30:-70:\smallradius);
	\coordinate[red,label= left: \textcolor{red}{$\sigma_1$}] () at (-45:1.1*\smallradius);
	\draw[red, ->] (-45:1.7*\smallradius) --+(135:0.5*\smallradius);
	\draw [green] (90:\smallradius) [out=180, in=135]to (-0.9*\smallradius, 1.6*\smallradius);
	\draw[green](-180:\smallradius) [out=90, in =135] to (-1.6*\smallradius, 0.9*\smallradius);
	\draw [green, ->](135:\smallradius) to (135:1.5*\smallradius);
	\end{scope}

	\end{scope}

	\draw[line width=1]   (-\length,-\bigradius-2.2*\radius) -- (\length,-\bigradius-2.2*\radius);
	\begin{scope}[shift={(-\length,-\bigradius-2.2*\radius )}]
	\draw [line width=1] (0,-\smalldash)--(0,\smalldash);
	\coordinate [label=below:{$t=0$}] () at  (0,0);
	\end{scope}
	\begin{scope}[shift={(\length,-\bigradius-2.2*\radius)}]
	\draw [line width=1] (0,-\smalldash)--(0,\smalldash);
	\coordinate [label=below:{$t=\infty$}] () at  (0,0);
	\end{scope}
	\end{tikzpicture}
	\caption{boundary strata of one-punctured discs with 4 boundary marked points}	
\end{figure}

\subsection{Homotopy between compositions}\label{sec:homotopy_btwn_composition}
Having constructed the chain map $\F^1$, we would like to check it preserves product structures on the two chains, on the source it is given by the concatenation on the based loop space, namely the concatenation of paths; on the target it is the pair of pants product. As in the failure of pair of pants product to be associative, the preservation of product structure induced by $\F^1$ is only observed at the level of homology. At the level of chains, there is a homotopy between the $\F^1(\mu_2^P(\sigma_1,\sigma_2)))$ and $\mu^2_{\CW}(\F^1(\sigma_1), \F^1(\sigma_2) )$; these two compositions are represented by the two outermost diagram in Figure (\ref{fig:2-product homotopy}).
We shall therefore need to introduce a one-dimensional moduli space $\ZZ_{4}^{l_1,l_2}$ for each pair of positive $l_1, l_2$ representing the lengths of the path of moving Lagrangians in $M$, namely given $\sigma_1 \in C_{k_1}(\Omega_{\LL_0,\LL_1}\Lag), \tau_i \in I^{k_1} $, then $l_1 = |\sigma_1(\tau_1)|$ and similarly for $l_2$. The boundary of this moduli space is represented by the broken curve and the collapsing of two adjacent marked points as shown in Figure (\ref{fig:2-product homotopy}). As in the proof of the homotopy associativity of the product in the Fukaya category, we shall define a family of Cauchy-Riemann equations on this abstract moduli space, interpolating between the equations on the two boundary figures. The family of moduli spaces of solutions to this family of equations, with their corresponding boundary conditions, shall define the desired homotopy. For higher homotopies that is incorporated in $A_\infty$ functor, the constructions are similar and we describe them in one piece through Section (\ref{sec:abstract_moduli_space_of_upper_half_plane}) and Section (\ref{sec: proof_of_main_theorem}).

\subsection{Abstract moduli spaces of upper half planes}\label{sec:abstract_moduli_space_of_upper_half_plane}

Given a tuple of positive real numbers $l_1,\cdots, l_d$ write $\ZZ_{2d}^{l_1,l_2,\cdots, l_d}$ for the moduli space of upper half planes with $2d$ consecutive boundary marked points $z_1,\cdots, z_{2d}$, and that the distance $\tilde{l_i} =|z_{2i+1}-z_{2i}|$ satisfy $\frac{\tilde{l_i}}{\tilde{l_j}} =\frac{l_i}{l_j}$. This is bi-holomorphic to a disc with a puncture and $2d$ marked points on the boundary with fixed ratios. In addition, we fix an orientation on the moduli space $\ZZ_{2d}^{l_1,l_2,\cdots, l_d}$ using the conventions for Stasheff polyhedra \cite{seidel-book}.
In fact, we have the following lemma:
\begin{lem}
	The space $\ZZ_{2d}^{l_1,l_2,\cdots, l_d}$ is isomorphic to the moduli space $\RR_{d+1}$ of discs with $d+2$ boundary marked points.
	\begin{equation}\label{eq:isomorphism_between_ZZ_and_standard_Stasheff}
		\ZZ_{2d}^{l_1,l_2,\cdots, l_d}\iso \RR_{d+1} \subseteq \RR_{2d}.
	\end{equation} 
\end{lem}
\begin{proof}
We take the boundary points on the source indexed by odd numbers ($z_1,z_3.\cdots, z_{2d-1}$) to the first $d$ incoming marked points on $\RR_{d+1}$ and $z_2$ to the $d+1$-th incoming marked point (actually we may take any $z_{2k} \ \forall \ k \in \{1,\cdots, d\}$  instead of $z_2$ as the $(d+1)$th incoming marked point); the outgoing marked point for $\RR_{d+1}$ is mapped from the puncture at $\infty \in \bar{\HH}$. Note that the other even numbered marked points on the boundary of the upper half plane is fixed as long as we pick $z_2$ because the ratios between different $|z_{2k}-z_{2k-1}|$ are fixed. 
\end{proof}

Given the space of all configurations $\{(z_1,\cdots, z_{2d}) \}$ where $\{z_i\}$ lie on the boundary of $\HH$, to achieve stability, we need to take the quotient of the action of the automorphism of $\HH$ on the configurations $\{z_1,\cdots, z_{2d}\}$ which fixes the ratio of different lengths $\frac{l_i}{l_j}$. This automorphism group $Aut_{2d}^{\vec{l}}$ turns out to be $PSL(2; \R) \cap \text{Aff} (2;\R)$, thus we have the following equivalence relation on the space of configurations $\{(z_1,\cdots, z_{2d}) \}$ on $\del \HH$:
\begin{align} \label{eq: equivalence relation on upper half plane with markings}
 (z_1,\cdots, z_{2d}) \simeq (z'_1, \cdots, z'_{2d}) \iff z_i'=az_i+b \text{ for some $a \in \R_+, b \in \R$ and all $i$},
\end{align}

Therefore the dimension of the moduli space $\ZZ_{2d}^{l_1,\cdots, l_d}$ is:
\begin{equation}\label{eq:dimension_of_abstract_moduli}
2d-(d-1)-2=d-1.
\end{equation}
The first term $2d$ is the sum of $2d$ marked points on $\del {\HH}$; the term $(d-1)$ is due to the fact that as soon as we fix $z_2$, then the ratio $l_1:\cdots:l_d$ would fix the remaining $d-1$ marked points $z_4,\cdots, z_{2d}$; the last term $2$ is to mod out the automorphism group $Aut_{2d}^{\vec{l}}$.

Yet the compactification $\ZZbar_{2d}^{l_1,\cdots, l_d}$of $\ZZ_{2d}^{l_1,\cdots, l_d}$ inside $\RRbar_{2d}$ is not quite the same as that of $\RRbar_{d+1}$ as we shall see below.

From now on, if there is no confusion, we shall use $\ZZ_{2d}$ or $\ZZ^{\vec{l}}_{2d}$ to abbreviate for $\ZZ_{2d}^{l_1,\cdots, l_d}$.

There are two scenarios to be taken into account:

\begin{enumerate} \label{cond:scenarios of points collision in abstract moduli space of upper half plane}
\item  \label{case: normal point collision}
If two consecutive marked points $z_{2k}, z_{2k+1}$ come together, as in the picture on the left hand side of figure (\ref{fig:boundary_of_d_inputs}), the topological type is determined by sequences $z_1,\cdots,z_{2k}, z_{2k+1},\cdots, z_{d}$, thus we obtained a map:
\begin{equation}\label{eq:markings_collapse}
\ZZ_{2d-1}^{l_1,\cdots, l_{k}+l_{k+1}, \cdots, l_{d} }\to \ZZbar_{2d}^{l_1,\cdots,l_d},
\end{equation}
where we regard the original two markings $z_{2k}, z_{2k+1}$ as a single marked point $z_{2k}'$.

\item \label{case: exotic point collision}
 If two consecutive marked points indexed by $z_{2k-1}, z_{2k}$ come close together, then by the property of \textit{fixed ratio} of $z_{2i}-z_{2i-1}$ and $z_{2k}-z_{2k-1} \ \forall i, k$, all two consecutive markings with such configurations have to come close; thus there are multiple bubbles breaking off with corresponding gluing parameters such that the ratio is still fixed. This phenomenon corresponds to the right hand side of Figure (\ref{fig:boundary_of_d_inputs}). Note the topological type is determined by sequences $\{1,\cdots, d_1\}, \cdots, \{d_1+\cdots +d_{r-1}+1, \cdots, d\}$, a partition of $d$ into $r$ sets of consecutive integers. We get a map:
\begin{equation}\label{eq:bubbling_of_abstract_moduli}
\ZZ_{2d_1}\times \cdots \ZZ_{2d_r} \times \RR_{r} \to \ZZbar_{2d},
\end{equation}
\begin{rem}
	The dimension of the left hand side of \ref{eq:bubbling_of_abstract_moduli} is \[
	\sum_{i=1}^{r} (d_i-1) + (r+1-3) =(\sum d_i )-2 =d-2,
	\]
	which is $1$ less than that of $\ZZ_{2d}^{\vec{l}}$. Here the second term $(r+1-3)$ is the dimension of the moduli space of abstract discs $\RR_r$ with $r$ incoming marked points and a single outgoing marked point on the boundary.
\end{rem}

We shall denote the unique disc representing $\RR_r$ by level 1 disc, and call the other discs representing $\ZZ_{2d_i}$ level 2 discs.
\end{enumerate}

\subsection{Floer data for half planes}\label{sec:Floer_data_upper_half_plane}
To properly define the space of solutions of inhomogeneous Cauchy-Riemann equations on upper half planes, we need consistent Floer data.
Here we have the outgoing end at the puncture $\xi_0$ as before. 

\begin{defin}\label{def:Floer_data_on_upper_half_Disc}
	A \textit{Floer datum} $D_U$ on a stable upper half plane $U \in \ZZ_{2d}^{\vec{l}}$ consists of the following choices on each component:
	\begin{enumerate}
	\item Weights: a positive real number $w_{k,U}$, where $k \in \{0,\cdots, r \}$ associated to each end for the 2nd type of of boundary strata (\ref{eq:bubbling_of_abstract_moduli}) where the principal holomorphic disc is broken into a level 1 disc with $r$ incoming ends and a outgoing end and $r$ level 2 discs each equipped with a single outgoing ends, namely $\ZZ_{2s_1}\times \cdots \ZZ_{2s_r} \times \RR_{r} \subseteq \ZZbar_{2d+1}$ $w_i$ satisfy $1=w_0 \geq \sum_1^r w_{k,U}$. 
	
	For $U$ inside the open strata or first type of boundary strata where there is no breaking of holomorphic discs (\ref{eq:markings_collapse}), the weight is just $1$.
	\item time shifting maps: for the curve breaking boundary (\ref{eq:bubbling_of_abstract_moduli}),  map $\rho_U:\del \bar{U} \to [1,+\infty)$ which agrees with $w_{k,U}$ near $\xi^k$ where $\xi^k$ denote the infinite ends and equals $1$ for open strata and first type of boundary strata (\ref{eq:markings_collapse}).
	\item Hamiltonian perturbation: a map $H_U: U \to \sH(M)$  on each surface such that the restriction of $H_U$  to a neighborhood of each segment between $z_{2i+1}$ and $z_{2i+2}$ for $1\leq i\leq d$ takes value in $\sH(M)$, satisfying Condition (\ref{eq:quadratic Hamiltonian}), and whose value near $\xi^k$ for $0\leq k\leq r$  is : 
	\begin{equation}
	\frac{H\circ \psi^{w_{k,U}}}{w^2_{k,U}}
	\end{equation}
	for the boundary strata represented by (\ref{eq:bubbling_of_abstract_moduli}) and equals $H$ near $\xi^0$ for open strata and first type boundary strata (\ref{eq:markings_collapse}).
	\item Basic one-form: a  sub-closed one-form $\gamma_U$ whose restriction to the complement of a neighborhood of the intervals between $z_{2i-1}$ and $z_{2i} (1\leq i \leq d)$ vanishes; and whose pullback under $\epsilon^k$ for $0\leq k \leq r$ agrees with $w_{k,U} dt$.
	\item Almost complex structure: a family $J_{ \tau, z} \in \JJ(M)$ of almost complex structures parametrized by $U$ and $\tau \in I^k$ for any positive integer $k$,  whose pullback under $\epsilon_k$ agrees with $\psi^{w_{k,U}}d\tau$.
	\end{enumerate}
\end{defin}

As before we write $X_U$ for the Hamiltonian flow of $H_U$ and consider the differential equation (\ref{eq:CR_for_upper_half_plane}):
\[(du-X_U\otimes \gamma_U)^{0,1}=0
\]
with respect to $J_{\tau, z}$.
Note that the pullback of equation (\ref{eq:CR_for_upper_half_plane}) under the ends $\epsilon^k$ agrees with 
\begin{equation}\label{eq:CR_at_infinite_ends}
	(du\circ\xi^k - X_{\frac{H\circ \psi^{w_{k,U}}} {w_{k,U} }}\otimes dt)^{0,1}=0.
\end{equation}
In order for the count of solutions to equation (\ref{eq:CR_for_upper_half_plane}) to define the desired homotopies, we must choose its restriction to the boundary strata of the moduli spaces in a compatible way. In the case where $d=2$, the moduli space is an interval, whose two endpoints may be identified with the product $\ZZ_2\times \ZZ_2\times \RR_{2}$ and $\ZZ_3$ (see figure \ref{fig:2-product homotopy}).

Our discussion in sections (\ref{sec:Wrapped}) and (\ref{sec:chain_map}) fixed the Floer data, on the unique elements of $\RR_2$ and $\ZZ_2$. In the case of $\ZZ_2\times \ZZ_2\times \RR_2$ the equations on the out-coming ends of the level 2 discs and the incoming ends of the level 1 discs agree up to some rescaling. Thus we obtain Floer data in a neighborhood of $\del \ZZbar_4$, which may be extended to the interior of the moduli space. Having fixed this choice, we proceed inductively for the rest of the moduli spaces:

\begin{defin}\label{def:universal_floer_data_for_higher_homotopy}
	A universal and conformally consistent choice of Floer data for the homomorphism $\F$, is a choice $\bfD_{\F}$ of such Floer data for every integer $d\geq 1$ and every element of $\ZZ_{2d}$ which varies smoothly over this compactified moduli space, whose restriction to a boundary stratum is conformally equivalent to the product of Floer data coming from either $\bfD_{\mu}$ or lower dimensional moduli spaces $\ZZbar_{d_i}$ and which near such a boundary agrees to infinite order with the Floer data obtained by gluing.
	
\end{defin}
The consistency condition implies that each irreducible component of a curve representing a point in the stratum (\ref{eq:bubbling_of_abstract_moduli}) carries the restriction of the data $\bfD_{\mu}$ if it comes from the factor $\RRbar_{r}$, and the restriction of the datum $\bfD_{\F}$ if it comes from $\ZZ_{d_i}$, up to conformal equivalence.

\subsection{Moduli space of upper half planes }\label{sec:moduli_of_upper_half_plane}
In order to extend $\F^1$ to an $A_\infty$ functor $\{\F^d\}$, we need to construct corresponding moduli spaces. 
	First of all, let us briefly review the moduli space constructed for $\F^1$:
	Consider the upper half plane $\HH$ with $2$ marked points $z_1,z_2$ on the boundary $\R$, with length $l=z_2-z_1$. 
	\begin{rem}
	The domain is bi-holomorphic to a unit disc with a puncture at $(0,i) \in \C$ and $2$ marked points on the boundary where $\infty \in \bar{\HH}$ is identified with the puncture $(0, i)$. 
	\end{rem}

	For $\sigma \in C_{k}(\Omega_{\LL_0, \LL_1}(\Lag))$ and $x \in \Chord(\LL_0, \LL_1)$, we associated in Definition (\ref{def:moduli_of_Disc_3}) a moduli space  $\Disc_{3}(\sigma;x)$ of upper half planes with boundary marked points with moving Lagrangian boundary conditions given by $\sigma$.  It can be viewed as a family of spaces over $I^{k}$ such that $\forall \tau \in I^{k}$, the member of the family over $\tau$ is $\Disc_3 (\sigma(\tau);x)$, the space of holomorphic maps from the upper half space with two boundary markings $z_1 = (0,0), z_2 = (1,0)$ such that the path of Lagrangian submanifolds between them is dictated by $\sigma(\tau)$.
	
	Note in the definition of $\Disc(\sigma, x)$ and $\Disc(\sigma(\tau), x)$, we already fixed the action of automorphism of $\HH$ with two boundary markings. We can study the pair $(u, \tau)$ satisfying Condition (\ref{eqn: conditions_for_Disc_3}), with the small modification that instead of the strict requirement that $z_1 = (0,0), z_2 = (1,0)$ and the equation denoting the moving boundary to be 
	\begin{equation}
	u(z) \in L_{\sigma(\tau)(\frac{t}{|\sigma(\tau)|})} \text{ if $z=(t,0) \in \del U$ lies between $z_1$ and $z_2$ },
	\end{equation}
	we just let $z_1, z_2$ be two boundary markings and the component of the boundary of $\HH$ satisfies
	\begin{equation}
	u(z) \in L_{\sigma(\tau)(t)} \text{ if $z=(t,0) \in \del U$ lies between $z_{1}$ and $z_2$}
	\end{equation}
	This gives us a new space which we denote $\Disct_{3}(\sigma, x)$, whose quotient by the automorphism group $\text{PSL}(2;\R) \cap \text{Aff}(2, \R)$ on the domain is $\Disc_{3}(\sigma, x)$. Similarly we have a space $\Disct_{3}(\sigma(\tau), x)$ for any $\tau$.
	
  The above construction gives rise to following projection:
	\begin{equation}\label{eq:fibration_for_Disct_3}
	\begin{tikzcd}
%
		&\Disct_3(\sigma(\tau);x) \ar{d} \ar{r} &\Disct_{3}(\sigma;x) \ar{d}{\pi}\\
		& \{\tau\} \ar{r}[draw = none]{\in} &I^{k}.
 	\end{tikzcd}
	\end{equation}
  The automorphism group acts on $\Disct_3(\sigma(\tau);x)$ and $\Disct_{3}(\sigma;x)$ simultaneously and gives the following projection:
  \begin{equation}\label{eq:fibration_for_Disc_3}
  	\begin{tikzcd}
  	&\Disc_3(\sigma(\tau);x) \ar{r} \ar{d}&\Disc_{3}(\sigma;x) \ar{d}{\pi}\\
  	&\{\tau \} \ar{r}{\in}&I^{k}.
  	\end{tikzcd}
  \end{equation}
  
	We now set out to generalize the constructions to multiple markings in order to give the higher operations $\F^d$ of the $A_\infty$ functor.
	For any integer $d \geq 1$, we consider the upper half plane $\HH$ with $2d$ boundary markings, which is bi-holomorphic to the unit disc with 1 puncture on the boundary $\xi^0$ and $2d$ markings, such that lengths $l_i=z_{2i}-z_{2i-1}$ denote the length of the moving Lagrangian boundary condition. Let 
	\[{t_i}=z_{2i+1}-z_{2i}\]
	 represent the lengths of the boundary component of $\HH$ that lies in $L_i$, namely constant boundary condition. 
	 As in section (\ref{sec:chain_map}), we work with a Hamiltonian $H_U$ that is quadratic with respect to the radial coordinate near cylindrical ends as in condition (\ref{def:Floer_data_on_upper_half_Disc}); also similar to the discussion of families of almost complex structures, we consider the space $\JJ_{U, \vec{k} }$ of almost complex structures parametrized by the product of $I^{k_i}$ with domain $U$ such that for any $ \vec{\tau} \in (I^{k_1}, \cdots, I^{k_d}) $, $J_{\vec{\tau}, z} \in \JJ(M)$ is independent of $ \vec{\tau}$ when we go far enough in the strip-like ends near the puncture $\xi$ as in Equation (\ref{eq:almost_complex_structure_at_strip_like_ends}). 
	 Denote such a family of almost complex structure by $J_{\vec{\tau} }$ if there is no confusion.

	 Given $\sigma_i \in C_{k_i}(\Omega_{L_i, L_{i + 1}} \Lag)$ and $\tau_i \in I^{k_i}$, let $|\sigma_i(\tau_i)|$ denote the length of this path; fixing $x \in \Chord(\LL_i, \LL_{i+1})$, we define the space of pseudo-holomorphic maps with multiple moving Lagrangian boundary condition as follows:
	 
	 \begin{defin}\label{defin: moduli_space_of_curves_with_multiple_boundary_points}
	 	The space of perturbed pseudo-holomorphic maps with multiple moving Lagrangian boundary $\Disct(\sigma_1,\cdots, \sigma_d, x, H, J_{\vec{\tau}})$ is the space of pairs $(u, \vec{\tau})$ consisting of a map $u: U\to M$ and $\tau_i \in I^{k_i}$ satisfying
	 		\begin{equation}\label{eq:conditions_multiple_moving_Lagrangian_boundary}
	 	\left \{
	 	\begin{aligned}
	 	& (du-X_{H_U}\otimes dt)\circ j_U - J_{\vec{\tau}, z} (du - X_{H_U}\otimes dt)=0, \\
	 	& u(z) \in  \psi^{\rho(z)}(L_0) \text{ if $z \in \del U$ lies on the segment between $\xi^0$ and $z_1$ },\\
	 	& u(z) \in  \psi^{\rho(z)}(L_d) \text{ if $z \in \del U$ lies on the segment between $\xi^0 $ and $z_{2d}$},\\
	 	& u(z) \in L_{\sigma_i(\tau_i)(t)} \text{ if $z=(t,0) \in \del U$ lies between $z_{2i-1}$ and $z_{2i}$,}\\
	 	&\lim_{s \rightarrow + \infty} u(\epsilon^0(s,t )) = x(t) \text{(asymptotic condition).}\\
	 	\end{aligned}
	 	\right \}
	 	\end{equation}
	 \end{defin}
We shall write $\Disct_{2d+1}(\vec{\sigma}; x)$ to abbreviate it if there is no confusion.
 	
Fixing $\tau_1, \cdots, \tau_d$, we define $\Disct_{2d+1}(\sigma_1(\tau_1),\cdots, \sigma_d(\tau_d);x)$ to be the space of solutions to Equation (\ref{eq:conditions_multiple_moving_Lagrangian_boundary}) 
	then  $ \forall \  \tau_i \in I^{k_i}$ we get the following projection:
	\[
	\begin{tikzcd}
	&\Disct_{2d+1}(\sigma_1(\tau_1),\cdots, \sigma_d(\tau_d);x) \ar[r]\ar[d] & \Disct_{2d+1} (\sigma_1,\cdots, \sigma_d; x)\ar{d}{\pi}\\
	& \{\tau_1,\cdots, \tau_d\} \ar[r] &I^{k_1+\cdots+k_d}.
	\end{tikzcd}
	\]
	
	Note that the domain of the maps in any given fiber $\Disct_{2d+1}(\sigma_1(\tau_1),\cdots, \sigma_d(\tau_d) ;x)$ is $\HH$ with $2d$ markings $z_1, \cdots, z_{2d}$ such that the lengths $l_i$ between $z_{2i -1}$ and $z_{2i}$ are equal to $|\sigma(\tau_i)|$.
	Let $\vec{\sigma}(\tau)$ denote $\sigma_1(\tau_1),\cdots, \sigma_d(\tau_d)$. We mod out the automorphism group $Aut^{\vec{l}}_{2d}$ on the space of maps $\Disct_{2d+1}(\vec{\sigma}(\tau); x)$ with moving Lagrangian boundary conditions $\sigma_i(\tau_i)$ between $z_{2i-1}$ and $z_{2i}$ to achieve stability of the domain. Thus we get the moduli space of pseudo-holomorphic curves with moving Lagrangian conditions given by multiple paths $\sigma_i(\tau_i) \in \Omega_{\LL_{i-1}, \LL_i } \Lag$ on an element $\ZZ^{\vec{l}}_{2d + 1}$.
	
	Similarly, the space $\Disct_{2d+1} (\sigma_1,\cdots, \sigma_d; x)$ also allows the action of $Aut^{\vec{l}}_{2d}$ and the \textit{moduli space of pseudo-holomorphic curves with moving Lagrangian boundary conditions} $\Disc_{2d+1} (\sigma_1,\cdots, \sigma_d; x)$ is defined as the quotient of $Aut^{\vec{l}}_{2d}$ action. As in the discussion after Definition (\ref{def:moduli_of_Disc_3}), $\Disc_{2d+1}(\vec{\sigma};x)$ is a smooth manifold of dimension $\sum |\sigma| + \mu(x) + d -1 $ for generic choice of Floer data and its Gromov compactification is a stratified topological manifold with corners.
	
	There is the diagram similar to Equation (\ref{eq:fibration_for_Disc_3}):
	\[ \begin{tikzcd} \label{eq:fibration of disc with multiple markings}
		&\Disc_{2d+1}(\vec{\sigma}(\tau);x) \ar{r} \ar[d]& \Disc_{2d+1}(\vec{\sigma};x) \ar{d} \\
			&	\vec{\tau} \ar[r]	&I^{\sum k_i}.
	\end{tikzcd}
	\]

	In order to show that $\{\F^{d} \}_{d = 1}^{\infty}$ is a map of $A_\infty$ categories, we need to study the relevant strata of the Gromov compactification of $\Disc_{2d+1}(\vec{\sigma};x)$.
	

	First of all, we study the boundary strata of the compactification of the fibers $\Disc_{2d+1}(\vec{\sigma}(\tau); x)$ of the fibration (\ref{eq:fibration of disc with multiple markings}). From the discussion in Section (\ref{sec:abstract_moduli_space_of_upper_half_plane}) and the standard gluing arguments similar to that for Proposition (\ref{prop:boundary strata of Discbar_3}) we have the following lemma:
	\begin{lem}
	The boundary strata of $\Discbar_{2d+1}(\vec{\sigma}(\tau); x)$ consist of strata from the following scenarios:
	\begin{enumerate}
		\item  (The strata corresponding to Case (\ref{case: normal point collision})  of the boundary strata of $\ZZbar_{2d+1}$.)
		Let $\sigma_{\mathfrak{i}}(\tau_{i, i+1})$ denote the concatenation of two adjacent paths $\sigma_i(\tau_i)$ and $\sigma_{i+1}(\tau_{i+1})$ in $ \Lag$; in other words, $z_{2i}$ and $z_{2i+1}$ come together. This gives a moduli space 
		\begin{equation}\label{eqn:normal boundary strata of fiber}
		\Discbar_{2d}(\sigma_1(\tau_1), \cdots, \sigma_{\mathfrak{i}}(\tau_{i, i+1}),\cdots, \sigma_d(\tau_d); x)
		\end{equation}
		 of pseudo holomorphic maps from upper half plane with $2d-1$ boundary markings. 
		
		\item (The strata corresponding to Case (\ref{case: exotic point collision}) of $\del \ZZbar_{2d+1}$.)
		When $z_{2i -1}$ and $z_{2i}$ come close together, the upper half plane admits simultaneous bubble breaking with configuration given by $\{1,\cdots, s_1\}, \cdots, \{s_1+\cdots +s_{r-1}+1, \cdots, d\}$, a partition of $d$ into $r$ sets of consecutive integers. This gives rise to a boundary strata represented by 
		 \begin{align} \label{eqn: exotic boundary strata of fiber}
		\Disc_{d_1,\cdots, d_r}({\vec{\sigma}(\tau)}; x_1,\cdots, x_r, x):= &\Disc_{2d_1+1} ({\sigma_1(\tau_1),\cdots, \sigma_{d_1}(\tau_{d_1})}; x_1) \times \cdots\times \\ & \Disc_{2d_r+1}({\sigma_{\sum_{i=1}^{r-1}+1}(\tau_{\sum_{i=1}^{r-1}+1}) ,\cdots, \sigma_{d}(\tau_d); x_r })\times \Disc( x_1,\cdots, x_r; x).
		\end{align}
	\end{enumerate}
	\qed
		\end{lem}
  Standard arguments from gluing similar to the argument for Proposition (\ref{prop:boundary strata of Discbar_3}) give the strata of the Gromov compactification $\Discbar_{2d+1}(\vec{\sigma};x)$ of the total space.
  \begin{lem}
  The Gromov compactification $\Discbar_{2d+1}(\vec{\sigma};x)$ is a stratified topological manifold with corners that consists of the following pieces:

\begin{enumerate}
	\item The interior $\Disc_{2d+1}(\vec{\sigma};x)$.
	\item The union of the closures of strata representing maps of upper half plane with $2d-1$ boundary marked points corresponding to Case (\ref{case: normal point collision}). For each member $ \{\tau_1,\cdots, \tau_d \}$ of the family parametrized by $I^{\sum d_i }$, the intersections of such strata with $\Discbar_{2d +1}(\vsigma(\tau); x)$ are the unions of boundary strata (\ref{eqn:normal boundary strata of fiber}) where $i$ ranges from $1$ to $d$. We write $\Discbar_{2d}(\vsigma_{\mathfrak{i}}; x)$ to denote such a stratum whose intersections with $\{\Discbar_{2d +1}(\vsigma(\tau); x) \}_{\vtau \in I^{\sum k_i}}$ is the stratum represented by Equation \eqref{eqn:normal boundary strata of fiber}.
	
	\item The strata
	\begin{align}\label{eqn:exotic boundary of moduli space}
		\Disc_{d_1,\cdots, d_r}({\vec{\sigma}}; x_1,\cdots, x_r, x):= &\Disc_{2d_1+1} ({\sigma_1,\cdots, \sigma_{d_1} }; x_1) \times \cdots\times \\ & \Disc_{2d_r+1}({\sigma_{\sum_{i=1}^{r-1}+1} ,\cdots, \sigma_{d} }; x_r)\times \Disc(x; x_1,\cdots, x_r).
	\end{align}
	whose intersection with $\{\Discbar_{2d +1}(\vsigma(\tau); x)\}_{\vtau \in I^{\sum k_i}}$ are the boundary strata (\ref{eqn: exotic boundary strata of fiber}).
	This corresponds to the multiple bubbling in Equation (\ref{eq:bubbling_of_abstract_moduli}) and left side of Figure (\ref{fig:boundary_of_d_inputs}).
	
	\item $\Discbar_{2d+ 1}(\sigma_1, \cdots, \sigma', \cdots, \sigma_d; x)$ where $\sigma' \subset \del \sigma_i$, this stratum is from the boundary of $i$-th cube in the cubical chain.
\end{enumerate}\qed
 \end{lem}
\subsection{Compatible choices of fundamental chains} In order to define operations using the moduli space $\Discbar_{2d +1}(\vsigma; x)$ we must compare the orientation of the boundary of this moduli space with that of the interior. As with other sign related issues, the proof of this result is similar to that in \cite[Appendix A]{Ab12} and we omit the calculations.

\begin{lem} \label{lemma: sign difference of multiple moving lagrangians}
	The product orientation of the strata (\ref{eqn:exotic boundary of moduli space}) differs from the boundary orientation by the following equation: 
	\begin{equation}\label {eq: signs for multiple breakings}
	\sum_{k = 1}^{r} d_k \cdot [\sum_{i = 1}^{d_{k-1}} \mu(\sigma_i) + \sum_{j= 1}^{k-1} \mu(x_j)] 
	+ \sum_{k = 1}^{r-1} \mu(x_k) \cdot[\Sigma_k] + \sum_{k = 1}^{r} (\sum_1^k d_i) (\sum_1^{k-1} d_i),
	\end{equation}
	where $\Sigma_k$ denote the sum of all $\mu(\sigma_k)$ for $k$ in the last $k$ tuples of the partition $\{1,\cdots, d_1\}, \cdots, \{d_1 +\cdots, d_{r-1} + 1, \cdots, d\}$ for $d$.

	The difference between the orientation on the stratum $\Discbar_{2d}$ corresponding to \eqref{eqn:normal boundary strata of fiber} and the boundary orientation is 
	\begin{equation}
		2d + 1.
	\end{equation}
\end{lem}

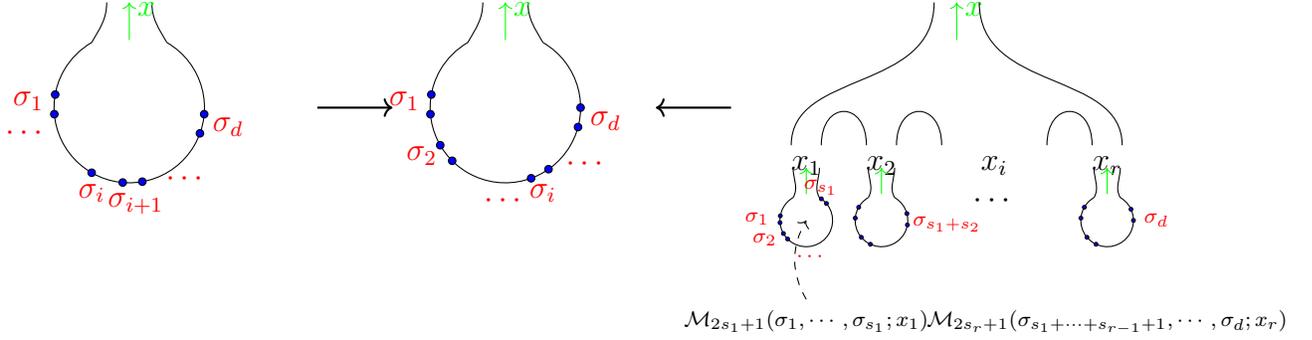
\begin{figure}\label{fig:boundary_of_d_inputs}
	\centering
	\begin{tikzpicture}
	\newcommand*{\radius}{1}
	\newcommand*{\smallradius}{.5*\radius}	
	\newcommand*{\tinyradius}{0.025}
	\newcommand*{\length}{5*\radius}

	\draw (60:\radius) arc (60:-240:\radius);
	\draw (60:\radius) to [out=120, in=-90] (0.3*\radius, 1.4*\radius);
	\draw (120:\radius) to [out=60, in=-90](-.3*\radius, 1.4*\radius);
	\draw [green,->](90:.9*\radius) to (0,1.4*\radius);
	\coordinate[label=above:\textcolor{green}{$x$}]() at (78:1.1*\radius);
	
	\draw [fill=blue] (170:\radius) circle (2*\tinyradius);
	\draw [fill=blue] (185:\radius) circle (2*\tinyradius);
	\coordinate [label=left:\textcolor{red}{$\sigma_1$}]() at (175:\radius);
	\begin{scope}[rotate= 40]
	\draw [fill=blue] (170:\radius) circle (2*\tinyradius);
	\draw [fill=blue] (185:\radius) circle (2*\tinyradius);
	\coordinate [label=left:\textcolor{red}{$\sigma_2$}]() at (180:\radius);
	\end{scope}
	
	\coordinate [label=below:\textcolor{red}{$\cdots$}]() at (-90:\radius);
	\begin{scope}[rotate= 120]
	\draw [fill=blue] (170:\radius) circle (2*\tinyradius);
	\draw [fill=blue] (185:\radius) circle (2*\tinyradius);
	\coordinate [label=below:\textcolor{red}{$\sigma_i$}]() at (180:\radius);
	\end{scope}
	
	\coordinate [label=below:\textcolor{red}{$\cdots$}]() at (-25:1.2*\radius);
	
	\begin{scope}[rotate= 175]
	\draw [fill=blue] (170:\radius) circle (2*\tinyradius);
	\draw [fill=blue] (185:\radius) circle (2*\tinyradius);
	\coordinate [label=right:\textcolor{red}{$\sigma_d$}]() at (175:\radius);
	\end{scope}

	\draw[line width=.8,->] (.6*\length,0) to (.4*\length, 0);
	\draw [line width=.8,->](-.5*\length,0) to (-.3*\length,0);
	
	\begin{scope}[shift={(-\length,0)}]
	\draw (60:\radius) arc (60:-240:\radius);
	\draw (60:\radius) to [out=120, in=-90] (0.3*\radius, 1.4*\radius);
	\draw (120:\radius) to [out=60, in=-90](-.3*\radius, 1.4*\radius);
	\draw [green,->](90:.9*\radius) to (0,1.4*\radius);
	\coordinate[label=above:\textcolor{green}{$x$}]() at (78:1.1*\radius);
	
	\draw [fill=blue] (170:\radius) circle (2*\tinyradius);
	\draw [fill=blue] (185:\radius) circle (2*\tinyradius);
	\coordinate [label=left:\textcolor{red}{$\sigma_1$}]() at (175:\radius);
	
	\coordinate [label=left:\textcolor{red}{$\cdots$}]() at (200:\radius);
	
	\begin{scope}[rotate=30]
	
	\draw [fill=blue] (210:\radius) circle (2*\tinyradius);
	\draw [fill=blue] (235:\radius) circle (2*\tinyradius);
	\coordinate [label=below:\textcolor{red}{$\sigma_i$}]() at (210:\radius);
	\draw[fill=blue](250:\radius) circle (2*\tinyradius);
	\coordinate [label=below:\textcolor{red}{$\sigma_{i+1}$}] () at (245:\radius);
	\end{scope}
	
	\begin{scope}[rotate= 170]
	\draw [fill=blue] (170:\radius) circle (2*\tinyradius);
	\draw [fill=blue] (185:\radius) circle (2*\tinyradius);
	\coordinate [label=right:\textcolor{red}{$\sigma_d$}]() at (175:\radius);
	\end{scope}
	
	\coordinate [label=right:\textcolor{red}{$\cdots$}] () at (-70:\radius);
	\end{scope} 

	\begin{scope}[shift={(1.2*\length, 0)}]
	\draw [green,->](90:.9*\radius) to (0,1.4*\radius);
	\coordinate[label=above:\textcolor{green}{$x$}]() at (78:1.1*\radius);
	\coordinate[label=below:$x_1$](I1) at (-2*\radius, -.5*\radius);
	\coordinate [label=below: $x_2$](I2) at (-\radius, -.5*\radius);
	\coordinate (I3) at (0, -.5*\radius);
	\coordinate [label=below:$x_r$](I4) at (2*\radius, -.5*\radius);
	\draw (-.3*\radius, 1.4*\radius) .. controls +(-90:1.2) and +(90:1) .. (-2.2*\radius, -.5*\radius);
	\draw (.3*\radius, 1.4*\radius) .. controls +(-90:1.2) and +(90:1) .. (2.2*\radius, -.5*\radius);
	\draw (-1.8*\radius,-.5*\radius) .. controls + (90:0.6) and +(90:0.6) .. (-1.2*\radius, -.5*\radius) ;
	\draw (-.8*\radius,-.5*\radius) .. controls + (90:0.6) and +(90:0.6) .. (-.2*\radius, -.5*\radius) ;
	\draw (1.8*\radius, -.5*\radius) .. controls + (90:0.6) and +(90:0.6) .. (1.2*\radius, -.5*\radius);
	\coordinate [label=below:$\cdots$] () at (.5*\radius, -\radius);
	\coordinate [label=below:$x_i$] () at (.5*\radius, -.5*\radius);
	
	\renewcommand{\smallradius}{.35*\radius}
	
	\begin{scope}[shift={(-2*\radius, -1.5*\radius)}]
	\draw (60:\smallradius) arc (60:-240:\smallradius);
	\draw (60:\smallradius) [out=120, in =-90] to (.5*\smallradius, 2*\smallradius);
	\draw (120:\smallradius) [out=60, in= -90] to (-.5*\smallradius, 2*\smallradius);
	\coordinate[label=below:\tiny{$\Disc_{2s_1+1}(\sigma_1,\cdots,\sigma_{s_1};x_1)$}]() at (0,-3*\smallradius);
	\draw [dashed,->](0,-3*\smallradius) to [bend left] (0,0);
	\draw [green, ->](0,\smallradius) to (0, 2*\smallradius);
	
	\draw [fill=blue] (170:\smallradius) circle (\tinyradius);
	\draw [fill=blue] (185:\smallradius) circle (\tinyradius);
	\coordinate [label=left:\textcolor{red}{\tiny{$\sigma_1$}}]() at (175:\smallradius);
	\begin{scope}[rotate=40]
	\draw [fill=blue] (170:\smallradius) circle (\tinyradius);
	\draw [fill=blue] (185:\smallradius) circle (\tinyradius);
	\coordinate [label=left:\textcolor{red}{\tiny{$\sigma_2$}}]() at (185:\smallradius);
	\end{scope}
	
	\coordinate [label=below:\textcolor{red}{\tiny{$\cdots$}}]() at (-75:.8*\smallradius);
	
	\begin{scope}[rotate=230]
	\coordinate [label=above:\textcolor{red}{\tiny{$\sigma_{s_1}$}}]() at (175:.8*\smallradius);	
	\draw [fill=blue] (170:\smallradius) circle (\tinyradius);
	\draw [fill=blue] (185:\smallradius) circle (\tinyradius);
	\end{scope}
	\end{scope}
	
	\begin{scope}[shift={(-1*\radius, -1.5*\radius)}]
	\draw (60:\smallradius) arc (60:-240:\smallradius);
	\draw (60:\smallradius) [out=120, in =-90] to (.5*\smallradius, 2*\smallradius);
	\draw (120:\smallradius) [out=60, in= -90] to (-.5*\smallradius, 2*\smallradius);
	\draw [green, ->](0,\smallradius) to (0, 2*\smallradius);
	\draw [fill=blue] (150:\smallradius) circle (\tinyradius);
	\draw [fill=blue] (175:\smallradius) circle (\tinyradius);
	\begin{scope}[rotate=70]
	\draw [fill=blue] (150:\smallradius) circle (\tinyradius);
	\draw [fill=blue] (175:\smallradius) circle (\tinyradius);
	\end{scope}
	
	\begin{scope}[rotate=200]
	\draw [fill=blue] (150:\smallradius) circle (\tinyradius);
	\draw [fill=blue] (175:\smallradius) circle (\tinyradius);
	\coordinate [label=right:\textcolor{red}{\tiny{$\sigma_{s_1+s_2}$}}]() at (150:0.8*\smallradius);
	\end{scope}
	
	\end{scope}

	\begin{scope}[shift={(2*\radius, -1.5*\radius)}]
	\draw (60:\smallradius) arc (60:-240:\smallradius);
	\draw (60:\smallradius) [out=120, in =-90] to (.5*\smallradius, 2*\smallradius);
	\draw (120:\smallradius) [out=60, in= -90] to (-.5*\smallradius, 2*\smallradius);
	\draw [green, ->](0,\smallradius) to (0, 2*\smallradius);
	\draw [fill=blue] (150:\smallradius) circle (\tinyradius);
	\draw [fill=blue] (175:\smallradius) circle (\tinyradius);
	\begin{scope}[rotate=70]
	\draw [fill=blue] (150:\smallradius) circle (\tinyradius);
	\draw [fill=blue] (175:\smallradius) circle (\tinyradius);
	\end{scope}
	
	\begin{scope}[rotate=210]
	\draw [fill=blue] (150:\smallradius) circle (\tinyradius);
	\draw [fill=blue] (175:\smallradius) circle (\tinyradius);
	\coordinate [label=right:\textcolor{red}{\tiny{$\sigma_{d}$}}]() at (150:\smallradius);
	\end{scope}
	\coordinate[label=below:\tiny{$\Disc_{2s_r+1}(\sigma_{s_1+\cdots+s_{r-1} +1},\cdots,\sigma_d;x_r)$}]() at 2*\smallradius,-3*\smallradius);
	\draw [dashed,->](0,-3*\smallradius) to [bend right] (0,0);
	
	\end{scope}
	
	\end{scope}

	\end{tikzpicture}
	\caption{boundary strata of one-punctured discs with $2d$ boundary marked points}
\end{figure}

\subsection{Proof of main theorem \eqref{main theorem}}\label{sec: proof_of_main_theorem}
	We now extend $\F^1$ to a sequence of multi-linear maps of all orders $d\geq 1$:
	\begin{align*}
	\F^d:C_*(\Omega_{\LL_d,\LL_{d-1}}\tLag) \otimes \cdots \otimes C_*(\Omega_{\LL_1,\LL_0}\tLag)  \to \CW_b^*(\LL_d,\LL_0)[1-d]\\
	\end{align*}
	To be more precise, \begin{equation}
		\F^d:C_{k_d}(\Omega_{\LL_d,\LL_{d-1}}\tLag)\otimes \cdots \otimes C_{k_1}(\Omega_{\LL_1,\LL_0}\tLag)  \to \CW_b^{-|\vec{k}|-d+1}(\LL_d,\LL_0),\\
	\end{equation}where $|\vec{k}|=\sum_{i=1}^{d}k_i$.
	The collection $\{\F^d\}$ should satisfy the following equations:
	\begin{multline}\label{FunctorA_infty relations}
	\sum_r \sum_{\substack{s_1,\cdots, s_r}}\mu^r_\CW(\F^{s_r}(\sigma_d,\cdots \sigma_{d-s_r+1}),\cdots, \F^{s_1}(\sigma_{s_1},\cdots, \sigma_1))\\
	=\sum_{\substack{m,n}}(-1)^{\maltese_n}\F^{d-m+1}(\sigma_d,\cdots,\sigma_{n+m+1},\mu_m^P(\sigma_{n+m},\sigma_n,\cdots, \sigma_{n+1}),\cdots, \sigma_1),
	\end{multline}
	where the left hand side is the sum over all $ r \geq 1$ and partitions $s_1+\cdots +s_r=d$, the right hand sum is over all possible $1\leq m\leq d, 0\leq n\leq d-m$, and $\maltese_n$ is defined in the same way as in (\ref{Maltese_sign}).
	Note that in our definition of the Pontryagin category $\Moore(\tLag)$, the only nontrivial $A_\infty$ operations are $\mu^P_1,\mu^P_2$. Thus  the above equation (\ref{FunctorA_infty relations})can be rewritten as 
	\begin{align}\label{eq:shortened_A_infty_relations}
	 &\sum_{r}\sum_{\vec{s}} \mu_{\CW}^r(\F^{s_r}(\cdots),\cdots,\F^{s_1}(\cdots))\\
	 =&(-1)^{\maltese_n}\F^{d}(\sigma_d,\cdots,\sigma_{n+2},\mu^P_1(\sigma_{n+1}),\sigma_n, \cdots, \sigma_1)\\
	 +&(-1)^{\maltese_n}\F^{d-1}(\sigma_d,\cdots,\sigma_{n+3},\mu^P_2(\sigma_{n+2},\sigma_{n+1}), \sigma_n,\cdots,\sigma_1).
	\end{align}
	Following the same strategy as the proof of Lemma (\ref{lem:chain_map}), we may equip these moduli spaces with appropriate cubical fundamental chains:
	\begin{lem}\label{lem:higher_homotopy_fundamental_chains}
		There exists a family of fundamental chains:
		\begin{equation}\label{eq:fundamental_chains_higher_homotopy}
			[\Discbar_{2d+1}(\vec{\sigma};x)] \in C_*(\Discbar_{2d+1}(\vec{\sigma};x)),
		\end{equation}
		such that 
		\begin{multline}\label{eq:relations_for_higher_chains}
			\del [\Discbar_{2d+1}(\vec{\sigma};x)] =\sum (-1)^{2d +1}[\Discbar_{2d}(\sigma_d\cdots \sigma_{i\cup (i+1)},\cdots, \sigma_1;x)]  \\
			+\sum (-1)^{\dagger}[\Discbar_{2d_r+1}(\sigma_d,\cdots, \sigma_{\sum_{i=1}^{r-1} d_i +1}; x_r)]
			 \times \cdots \times[\Discbar_{2d_1+1}(\sigma_{d_1},\cdots, \sigma_1;x_1)] \times [\Discbar_{r+1}(x_r,\cdots, x_1; x)]\\
			+\sum_{i=1}^{d} \sum_{\sigma' \subset \del \sigma_i} [\Discbar_{2d+1}(\sigma_d,\cdots, \sigma', \cdots, \sigma_1; x)],
		\end{multline}
		where $\dagger$ is an expression such that Equation (\ref{eq: signs for multiple breakings}) holds.
	\end{lem}

Similar to the definition of evaluation map of $\Discbar(\sigma; x)$ in \eqref{sec:chain_map}, we also define the evaluation on $\Discbar_{2d +1}(\vsigma; x)$ as follows:
\begin{equation}
	ev: \Discbar_{2d+1}(\vsigma; x) \to \Chord(L_0, L_1),
\end{equation}
which takes the punctured disc to its infinite end $x$. We obtain commutative diagrams similar to \eqref{eq: evaluation on neck breaking boundary} and \eqref{eq: evaluation on interior}.

With the evaluation map, we are set to define the higher operations $\F^d$:
	\begin{defin}\label{def:definition_of_higher_morphism}
		The operations $\F^d$ for higher homotopies are given by the following equation:
		\begin{align}
			\F^d:C_{k_d}(\Omega_{\LL_d,\LL_{d-1}}\tLag)\otimes &\cdots \otimes C_{k_1}(\Omega_{\LL_1,\LL_0}\tLag)  \to \CW_b^{-|\vec{k}|-d+1}(\LL_d,\LL_0)\\\label{eq:defining_equation_for_higher_homotopies}
			&\sigma_d\otimes \cdots \otimes \sigma_1 \mapsto \sum_{\di (\Disc(\sigma_d,\cdots,\sigma_1;x))=0} (-1)^{\dagger} ev_*([\Discbar_{2d+1}(\vsigma; x)])
		\end{align} 
	\end{defin}
Now we may proceed to give the proof of Theorem \ref{main theorem} based on Lemma \ref{lem:chain_map} and \ref{lem:higher_homotopy_fundamental_chains}.
\begin{proof}[Proof of Theorem \ref{main theorem}]
	From the description of the moduli spaces $\ZZbar_{2d}$ and their boundary strata from section (\ref{sec:moduli_of_upper_half_plane}) and the definition above; we realize that this is quite similar to the proof of Lemma (\ref{lem:chain_map}). We just need to identify the boundaries of the moduli space (\ref{eq:fundamental_chains_higher_homotopy}) and that of the equations of the $A_\infty$ relations:
	\begin{enumerate}
		\item the left hand side of Equations (\ref{eq:shortened_A_infty_relations}) corresponds to the 2nd terms on the right hand side of Equation (\ref{eq:fundamental_chains_higher_homotopy}) (when $r\geq 2$) and the left hand side of Equation (\ref{eq:fundamental_chains_higher_homotopy}) (when $r=1$).
		\item the 2nd terms on the right hand side of Equation (\ref{eq:shortened_A_infty_relations}) corresponds to the Left hand side of Equation (\ref{eq:fundamental_chains_higher_homotopy}).
		\item the first summands on the right hand side of Equation (\ref{eq:shortened_A_infty_relations}) corresponds to the 3rd summands of Equation (\ref{eq:fundamental_chains_higher_homotopy}).
	\end{enumerate}
\end{proof}


\section{Proof of Homotopy Equivalence on cotangent bundles}	

In this section we work with the cotangent bundle $T^*Q$ of a closed manifold $Q$. It can be realized at the completion of the unit disc bundle and the boundary $\del M^{in}$ is the unit circle bundle; the radial coordinate $r$ in this case is $|p|$ where $p$ denotes the cotangent vectors. 
Note that there is a natural map $\eta: C_*(\Omega_{q_0, q_1}Q) \to C_{*}(\Omega_{T_{q_0}^*Q, T_{q_1}^*Q}{\Lag})$ by taking a path in the zero section to the path of cotangent fibers in the space of Lagrangians with Legendrian boundary. We shall construct a chain homotopy equivalence between $Id_{C_{*}(\Omega_q Q)}$ and $f\circ F\circ \eta$ in the following commutative diagram:
	\[\xymatrix{
	C_{*}(\Omega_q Q)\ar@{.>}[dr] \ar^{\eta}[r]  &C_{*}(\Omega_{T_q^*Q}{\Lag}) \ar^{\F}[d]\\
	&\Cal{CW}^{-*}(L,L; H) \ar@/^2pc/|-{f}[ul]
}
\]
We use the concatenation of paths on the zero section $Q$ and evaluation of moduli spaces of pseudo-holomorphic discs with moving Lagrangian boundary condition defined below to show the homotopy equivalence in the later part of this section.

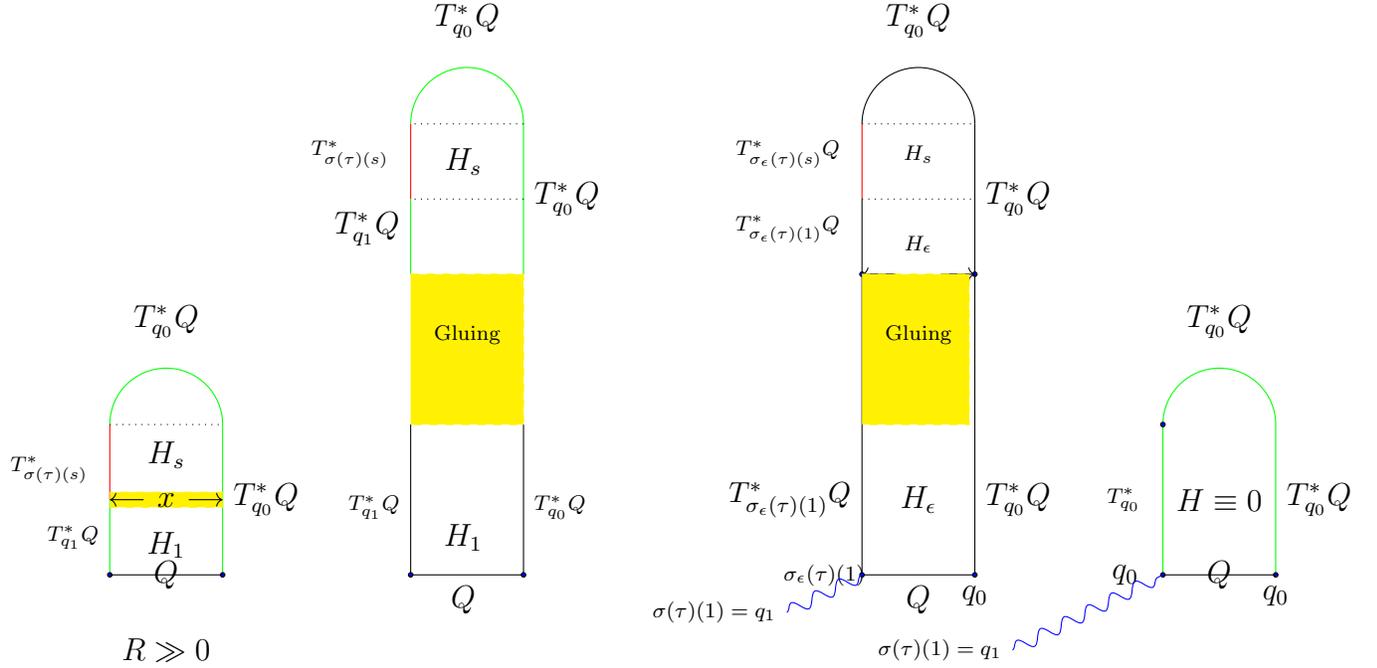
\begin{figure}\label{fig: Gluing of two moduli spaces}
	\centering
	\begin{tikzpicture}
	\begin{scope}[shift={(-7, 0)}]
	\newcommand*{\length}{1.5}
	\newcommand*{\tinyradius}{.03}
	\newcommand*{\height}{2}
	\draw[green](0,0) to (0, .5*\height);
	\draw[green](\length,0) to (\length, \height);
	\draw[green] (\length, \height) arc (0:180: .5*\length);
	\coordinate[label = right: $T^*_{q_0}Q$]() at (\length, .5*\height);
	\coordinate[label = above:$T^*_{q_0}Q$]() at (.5*\length, 1.5*\height);
	\coordinate[label = left: \tiny{$T^*_{\sigma(\tau)(s)}$}] () at (-.1*\length, .7*\height);
	\coordinate[label = left: \tiny{$T^*_{q_1}Q$}]() at (0, .25*\height);
	\draw[red] (0,\height) to (0,.5*\height);
	\draw[fill = blue](0,0) circle(\tinyradius);
	\draw[fill = blue](\length, 0) circle (\tinyradius);
	\draw[dotted](0,\height) to(\length, \height);
	\coordinate[label = center: $H_s$]() at (.5*\length, .8*\height);
	\draw[dotted](0,.5*\height) to (\length, .5*\height);
	\coordinate[label = center:$H_1$]() at (.5*\length, .2*\height);
	\draw (0,0) to (\length, 0);
	
	\coordinate[label = center: $Q$] () at (.5*\length, 0);
	\filldraw [yellow, dashed, line width =.2] (0,.45*\height) --(\length, 0.45*\height)--(\length, .55*\height)--(0,.55*\height);
	\draw[<-](0, .5*\height) to (.3*\length,.5*\height);
	\draw[->](.7*\length, .5*\height) to (\length, .5*\height);
	\coordinate[label = center: $x$]() at (.5*\length,.5*\height);

	\coordinate[label = center: $R \gg 0$] () at (0.5*\length, -1);
	\end{scope}	
	
	\newcommand*{\length}{1.5}
	\newcommand*{\tinyradius}{.03}
	\newcommand*{\height}{2}

	\begin{scope}[shift = {(-3,0)}]
	\draw (0,0) to (0, \height);
	\draw (0,0) to (\length, 0);
	\draw (\length, 0) to (\length, \height);
	\draw[fill = blue](0,0) circle(\tinyradius);
	\draw[fill = blue](\length, 0) circle (\tinyradius);
	\coordinate[label = below: $Q$] () at (.7, 0);
	\coordinate[label = left: \tiny{$T_{q_1}^*Q$}]() at (0, .9);
	\coordinate[label = right: \tiny{$T_{q_0}^*Q$}]() at (\length, .9);
	\coordinate[label = center: $H_1$]() at (0.7, 0.5);
	
	\begin{scope}[shift = {(0,4)}]
	\draw[green](0,0) to (0, .5*\height);
	\draw[green](\length,0) to (\length, \height);
	\draw[green] (\length, \height) arc (0:180: .5*\length);
	\coordinate[label = right: $T^*_{q_0}Q$]() at (\length, .5*\height);
	\coordinate[label = above:$T^*_{q_0}Q$]() at (.5*\length, 1.5*\height);
	\coordinate[label = left: $T_{q_1}^*Q$]() at (0, .3 *\height);
	\coordinate[label = left: \tiny{$T^*_{\sigma(\tau)(s)}$}] () at (-.1*\length, .8*\height);
	
	\draw[red] (0,\height) to (0,0.5*\height);
	\draw[dotted] (\length, \height) to (0, \height);
	\draw[dotted] (\length, .5*\height) to (0,.5* \height);
	\coordinate[label = center: $H_s$]() at (.7, 0.75* \height);
	
	\filldraw [yellow, dashed, line width =.4] (0,0) --(\length, 0)--(\length, -2)--(0,-2);
	\coordinate[label = center: \tiny{Gluing}]() at (.5*\length, -.8);
	\end{scope}
	\end{scope}

	\begin{scope}[shift = {(3,0)}]
	\draw [line width =.4](0,0) to (0,4);
	\draw [line width =.4](\length,0) to (\length,4);
	\draw (0,0)to (\length,0);
	\coordinate[label = below: $Q$]() at (.5*\length,0);
	\coordinate[label = below:$q_0$] () at (\length, 0);
	\coordinate[label = center: \tiny{$\sigma_\epsilon(\tau)(1)$}] () at (-.5,0);
	\draw[fill=blue](0,0) circle (\tinyradius);
	\draw[fill=blue](\length, 0) circle (\tinyradius);
	\coordinate[label = left:\small{$T_{\sigma_\epsilon(\tau)(1)}^*Q$}] () at (0,1);
	\coordinate[label = right: $T_{q_0}^*Q$] () at (\length,1);
	\coordinate[label = center: $H_\epsilon$]() at (.5*\length, 1);
	\path[draw = blue, snake it] (-1, -.5) --(0,0);
	\coordinate[label = left:\tiny{$\sigma(\tau)(1) = q_1$}]() at (-1, -.5);
	
	\begin{scope}[shift = {(0,4)}]
	\draw(0,0) to (0, .5*\height);
	\draw(\length,0) to (\length, \height);
	\draw(\length, \height) arc (0:180: .5*\length);
	\coordinate[label = right: $T^*_{q_0}Q$]() at (\length, .5*\height);
	\coordinate[label = above:$T^*_{q_0}Q$]() at (.5*\length, 1.5*\height);
	\coordinate[label = left: \tiny{$T^*_{\sigma_\epsilon(\tau)(s)}Q$}] () at (-.1*\length, .8*\height);
	\draw[red] (0,\height) to (0,0.5*\height);
	\draw[fill = blue](0,0) circle(\tinyradius);
	\draw[fill = blue](\length, 0) circle (\tinyradius);
	\coordinate[label = left: \tiny{$T^*_{\sigma_\epsilon(\tau)(1)}Q$}] () at (-.1*\length, .3*\height);
	\draw[dotted](0, \height) to (\length, \height);
	\coordinate[label = center: \tiny{$H_s$}] () at (.5*\length, .8*\height);
	\draw[dotted](0,.5*\height) to (\length, .5*\height);
	\coordinate[label = center: \tiny{$H_\epsilon$}] () at (.5*\length, .2*\height);
	
	\draw[->] (.7*\length, 0) to (\length, 0);
	\draw[->](.3*\length, 0) to (0,0);
	
	\filldraw [yellow, dashed, line width =.4] (0,0) --(0.95*\length, 0)--(0.95*\length, -2)--(0,-2);
	\coordinate[label = center: \tiny{Gluing}]() at (.5*\length, -.8);
	\end{scope}
	\end{scope}

	\begin{scope}[shift = {(7,0)}]
	\draw[green](0,0) to (0, 1*\height);
	\draw[green](\length,0) to (\length, \height);
	\draw[green] (\length, \height) arc (0:180: .5*\length);
	\coordinate[label = right: $T^*_{q_0}Q$]() at (\length, .5*\height);
	\coordinate[label = above:$T^*_{q_0}Q$]() at (.5*\length, 1.5*\height);
	\coordinate[label = left: \tiny{$T^*_{q_0}$}] () at (-.1*\length, .5*\height);
	\draw[fill = blue](0,0) circle(\tinyradius);
	\draw[fill = blue](\length, 0) circle (\tinyradius);
	\draw[fill = blue](0,\height) circle(\tinyradius);
	\draw (0,0)to (\length, 0);
	\coordinate[label = center: $Q$]() at (.5*\length, 0);
	\coordinate[label = center: $H \equiv 0$] () at (.5*\length, .5*\height);
	\coordinate[label = center: $q_0$] () at (-.5, 0);
	\coordinate[label = below: $q_0$]() at (\length, 0);
	\path[draw = blue, snake it] (-2, -1) --(0,0);
	\coordinate[label = left:\tiny{$\sigma(\tau)(1) = q_1$}]() at (-2, -1);
	\end{scope}
	\end{tikzpicture}
	\caption{Homotopy between $Id$ and $f\circ F\circ \eta$}
		
\end{figure}

\subsection{Gluing with moduli space of half-strips}

 From \cite[Section~4]{Ab12} we have a moduli space $\mathcal{H}(x;q_0,q_1 )$ of pseudo holomorphic maps from the upper half space with 2 punctures on the boundary to $T^*Q$ such that the two punctures are mapped to $q_0, q_1 \in Q$ and the boundary component of the upper half plane between the two punctures is mapped into $Q$. The evaluation map from the compactification $\bar{\Cal{H}}(q_0,q_1,x) \to \Omega(q_0,q_1)$ gives the chain map $f^1: \CW^k_b(T_{q_0}^*Q, T_{q_1}^*Q) \to C_{k}(\Omega_{q_0,q_1}Q)$.


%
%

 Let $Z$ denote the Riemann surface that is a disk with two punctures $\xi_0, \xi_1$ on its boundary, which is bi-holomorphic to a strip $\R \times I$. 
 
 We denote the domain of the moduli spaces $\{\Cal{H}(x, q_0,q_1)\}$ by $S_+$ for any $x \in \Chord(T_{q_0}^*Q, T_{q_1}^*Q)$. $S_+$ is equipped with three strip-like ends $\epsilon_{\pm}: Z_{\pm} \to S_{+},$ where the positive ends are near the incoming points $q_0, x(t)$ and negative ends near  $q_1$ as defined in \cite[Section4]{Ab12}; the domain $S_{-}$ of $\{\Disc(\sigma,x,H,J_{\tau})\}$ has a negative strip-like end for $x(t)$. For any gluing parameter $\delta \in (0, 1)$, we set the associated gluing length to be $R = - \log(\delta)$. One then defines another domain 
 \begin{equation}\label{eqn: glued_strip}
 	S_R = S_+ \#_R S_{-}
 \end{equation}
 by first removing $\epsilon_{+}([R, +\infty) \times [0,1])$ in $S_+$ and $\epsilon_{-}((-\infty, -R] \times [0,1])$ in $S_{-}$, both corresponding to strip-like ends near the puncture $x$, from $S_{+} \sqcup S_{-}$, and then identifying the remaining finite pieces of the ends through the formula $\epsilon_{+}(s, t)\sim \epsilon_{-}(s - R, t)$. The positive and negative ends at infinity in $S$ are the ones inherited from $S_{+}$.
 
 We had imposed conditions on Hamiltonians (\ref{eq:quadratic Hamiltonian}) and the behavior of almost complex structures near strip-like ends (\ref{eq:almost_complex_structure_at_strip_like_ends}); by \cite[chapter 9]{seidel-book}, there is a consistent choice of Floer data and strip like-ends that for small values of the gluing parameter $\delta$. Therefore we may impose a new inhomogeneous Cauchy-Riemann equation on this new Riemann surface $S_R$, and this equation agrees with the original equations in the definition of the two moduli spaces when restricted to $S_{+} \setminus \epsilon_{+}([R, +\infty) \times [0,1])$ and $S_{-} \setminus \epsilon_{-}((-\infty, -R] \times [0,1])$ respectively and interpolates in the overlapping region of gluing. For every $R \in [1,\infty)$, let us denote the moduli space of the solutions to the new equation on $S_R$ by $\GG(\sigma; R)$. 
 
 Note that the domain $S_R$ is bi-holomorphic to the strip $Z$ with one component of the boundary mapped to the zero section $Q$, a constant Lagrangian. It is illustrated in the second picture of Figure (\ref{fig: Gluing of two moduli spaces}). The component that is mapped to $Q$ is drawn as the bottom interval in the picture, and the arc in the picture forms the other component. 

%
 
   When $R \to \infty$, the gluing theorem (eg. \cite{FOOO}) implies that the moduli space is homeomorphic to the union over all orbits $x$ of the Cartesian product of the moduli space of upper half space with 2 punctures $\Cal{H}(x; q_0, q_1)$ and that of our moduli space $\Disc(\sigma, x)$:
 \begin{equation}\label{eqn:isomorphism of gluing when R is infty}
 \GG(\sigma,R) \simeq
 \bigsqcup_{x\in \Chord(T_{q_0}^*Q, T_{q_1}^*Q)}\Disc(\sigma,x) \times \Cal{H}(x; q_0, q_1).
  \end{equation}
 Therefore, we may also define $\GG(\sigma,\infty)$ to be $\bigsqcup_{x\in \Chord(T_{q_0}^*Q, T_{q_1}^*Q)}\Disc(\sigma,x) \times \Cal{H}(x; q_0, q_1)$ for all given $\sigma \in C_{*}(\Omega_{(q_0, q_1)}Q)$.
 
 

%
%
 

 We may pick a smooth monotone map $ \nu: [1/2, 1] \to [1, \infty]$ such that $\nu$ maps a small neighborhood of $\frac{1}{2}$ in this half interval constantly to $1$.  We are going to work with this half interval from now on and still name its parameter $R$.
 Gluing  induces a family of maps parametrized by $\tau \in I^k$ on the boundary of $Z$ 
 \begin{equation} \label{eqn: Lagrangian moving boundary for gluing}
 \Phi: \del Z \times [1/2,1] \times I^k \to \Lag(T^*Q)
 \end{equation} 
 such that for each $\tau$, $\Phi_\tau: \del Z \times [1/2, 1]$ is locally constant near $1/2$ by our construction.
 Therefore, Equation (\ref{eqn:isomorphism of gluing when R is infty}) holds when $R$ is close to $1$ under this new parametrization.
 
\subsection{Moduli space of family of moving Lagrangian boundary conditions}\label{sec: sigma_epsilon}
 Given $\sigma : [0,1]^{k} \to \Omega_{(q_0,q_1)} Q$, namely $ [0,1]^{k}\times [0,1] \to Q$, a chain of paths on the compact manifold $Q$, we will define $\sigma_\epsilon: [0,1]^{k} \times [0,1] \to Q$ such that $\sigma_\epsilon(\tau) $ moves from $q_0$ to $q_\epsilon$ smoothly, where $q_\epsilon = \sigma(\tau)(\epsilon)$.
 We want to define a moduli space $\Disc(\sigma_\epsilon, q_0, H)$ of J-holomorphic curves with boundary conditions given by the above one parameter family. As for the Floer data, we are concerned with the moving boundary conditions and Hamiltonian perturbation as we require smoothness.
 
We use linear Hamiltonian flow to generate the moving Lagrangian co-fibers. Consider the following mollifier function:
 \[
 \phi: \R \to [0,1], 
 \phi = 
\begin{cases}
2e\cdot \exp(-\frac{1}{1-(2s)^2} ) & s \in (-1,1)\\
0 & |s|\geq 1\\
\end{cases}
 \]
 
 Notice $\phi|_{I}$ has compact support in $(-1/2, 1/2)$ and it is $C^\infty$ in this open interval.
 By convolution of the characteristic function $\chi$ of the unit ball $B_{1}$ inside $\R$ with the mollifier $\phi$ defined above, we obtain the function $\psi_{1/2}$:
 \[
 \psi_{1/2} (s) = \chi \ast \phi (s) = \int_{\R} \chi(s-t) \phi(t) dt =\int_{B_{1/2}} \chi(s-t)\phi(t)dt
 \]
 which is smooth and equal to $1$ inside $B_{1/2}$ and is supported inside $B_{3/2}$. This is observed with the fact that if $|s|\leq 1/2$ and $|t| \leq 1/2$, then $|s-t| \leq 1$, hence for $|s| \leq 1/2$,
 \[
 \int_{B_{1/2}} \chi_{1/2}(s-t)\phi(t)dt = \int_{B_{1/2}}\phi\  dt = 1.
 \]
\begin{defin}[]\label{defin:bump_function}
	$\psi(s):= \psi_{1/2}(s-\frac{1}{2})$ for $s \in [0,1]$.
\end{defin}
 So far, we have constructed a smooth function $\psi$ which is $1$ at $s =0$, $0$ at $s = 1$  and all order right (resp. left) differentials vanish at that these two points.
 We extend $\psi$ by letting $\psi(s) = 1$ for $s < 0$ and $\psi(s) = 1$ for $s >1$ and get a smooth function $\R \to [0,1]$ that is a smooth diffeomorphism restricted to $[0,1]$. 
 
   Now for any $\sigma, \tau$, $\sigma(\tau)$ is a loop in $Q$. In other words, it is a map $I \to Q$ where both ends are mapped to the same point and that differentials of all orders on both ends are exactly opposite. If we compose it with the inverse of $\psi$ we still get a loop $\sigma(\tau)\circ \psi^{-1}$ in Q that is smooth. Thus when $\tau$ traverses $I^k$, we get another cubical chain of based loops in $Q$. Let us denote it by $\tilde{\sigma}$. From the above construction we may write $\sigma $ as $\tilde{\sigma} \circ \psi$ to abbreviate for $\sigma(\tau)(s) = \tilde{\sigma}(\tau)(\psi(s))$ for any given $\tau \in I^{k}$.
 
 \begin{defin}[]\label{def: sigma_epsilon}
 	The boundary condition $\sigma_\epsilon : I^{k} \times \R \to [0,1]$ is given by:
 	\[\sigma_\epsilon  = \tilde{\sigma} \circ (\epsilon\cdot \psi)
 	\]
 	namely:
 	\[
 	\sigma_\epsilon(\tau)(s) := \tilde{\sigma}(\tau)(\epsilon \cdot \psi(s))\]
 	where $\psi$ is defined as above.
 \end{defin} 
 
 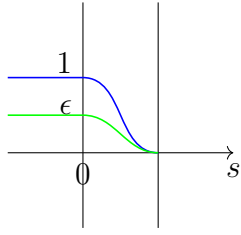
\begin{figure}\label{fig: psi_function}
 	\begin{tikzpicture}
 	\draw(0,-1) to (0,2);
 	\draw(1, -1) to (1,2);
 	\draw[->] (-1,0) to (2,0);
 	\coordinate[label = below: $s$] () at (2,0);
 	\coordinate[label = below:$0$]() at (0,0);
 	\draw [blue, line width =0.6] (0,1) [out =0, in = 180] to (1,0);
 	\draw[blue, line width =.6](-1,1) to (0,1);
 	\coordinate[label = left: $1$]() at (0, 1.2);
 	\draw[green, line width = .6](0, .5) [out = 0, in =180] to (1,0);
 	\draw[green, line width =.6](-1, .5) to (0, .5);
 	\coordinate[label = left:$\epsilon$]() at (0, .6);
 	\end{tikzpicture}
 	\caption{$\psi$ and $\psi_\epsilon$}
 \end{figure}
 
 From the construction above, we get the smoothness of $\sigma_\epsilon$ with respect to $s$. It is easy to see that it is also smooth with respect to $\epsilon$ as both $\tilde{\sigma}(\tau)$ and $\psi $ are smooth. Therefore we have the following lemma:
 \begin{lem}\label{lem: smoothness_in_s}
 	$\sigma_\epsilon$ is smooth with respect to both $s$ and $\epsilon$.
 \end{lem}

 Now that we have dealt with smoothness, we are set to define the moduli space of the family of moving Lagrangian boundary conditions.

 First we require the Hamiltonian function $H$ to be quadratic in the distance to the zero section and large enough slope with respect to $s$ as in (\ref{eq:quadratic_Hamiltonian}), namely $H = c(s)\cdot |p|^2 $ for some smooth function $c(s)$ that vanishes outside $[-1,2]$. The reason for such a choice is for the result of a $C^0$ bound of the pseudo holomorphic curves given by the solutions of the inhomogeneous equations, the details are in Appendix \ref{Appendix: C^0_bound}.
 
 As for the family of almost complex structures $J_{\tau, z,\epsilon}$, they satisfy the condition laid out in Section \ref{sec:chain_map} and they are also parametrized by $\epsilon$ for the Cauchy-Riemann equation in the following definition.

 \begin{defin}
 	Given Floer data $(H, J_\tau, \gamma)$, we define the moduli space of the family of pseudo-holomorphic maps with moving Lagrangian boundary condition $\Disc(\sigma_\epsilon, q_0, H, J_\tau)$ to be pairs $(u, \tau)$ consisting of maps $u : Z = \R \times [0,1]$(parametrized by (s, t)) to $T^*Q$ and $\tau \in I^k$ satisfying the following conditions:
 	\begin{equation}\label{eqn:family_Lagrangian_boundary}
 		\left \{
 		\begin{aligned}
 		& (du-X_H\otimes \gamma)^{0,1}=0, \\
 		& u(z) \in Q \text{ if $z \in \R \times \{1\} \subseteq \del Z$ lies on the upper boundary of $Z$ },\\
 		& u(z) \in T_{q_0}^*Q \text{ if $z \in [1, \infty]\times \{0\} $ } ,\\
 		& u(z) \in T_{\sigma_\epsilon(\tau)(s) }^*Q \text{ if $z=(s,0) \in \del Z$ lies between $(0,0)$ and $(1,0)$},\\
 		& u(z) \in T_{q_\epsilon} \text{if $z \in [-\infty,0]\times \{0\}$ , }\\
 		& \lim_{s \to -\infty} u(s,\cdot) = q_\epsilon, \lim_{s\to \infty}u(s, \cdot) = q_0, \text{(asymptotic condition)}\\
 		\end{aligned}
 		\right \}
 	\end{equation}
 \end{defin}

In particular, when $\epsilon = 0$, we have the moduli space of curves $\Disc(\sigma_0, q_0)$ with constant boundary condition in $Q$ and $T_{q_0}^*Q$, which is a copy of $I^k$.

 \begin{figure}\label{fig: moving_boundary_condition}
 	\centering
 	\begin{tikzpicture}
 	\newcommand*{\tinyradius}{.03}
 	\draw[line width = 0.4] (3,0) to (-3,0);
 	\coordinate[label=above: $Q$]() at (0,1);
 	\draw[line width =0.4] (3,1) to (-3,1);
 	\coordinate[label=above:\small{$(du-X_H\otimes \gamma)^{0,1}=0$ }]() at (0,0);
 	
 	\coordinate[label=below:$T^*_{q_{\epsilon}} Q$]() at(-2.5,0);
 	\coordinate[label=below:$T^*_{q_0}Q$]() at (2.5,0);
 	\coordinate[label=below: $T^*_{\sigma_\epsilon(\tau)(s)}Q$]() at (0,0);
 	\draw[fill=blue] (-1.5,0) circle(\tinyradius);
 	\draw[fill=blue](1.5,0) circle (\tinyradius);
 	\draw[red, line width=0.6] (-1.5,0) to (1.5,0);

 	\end{tikzpicture}
 \end{figure}
 A picture of the boundary condition for the above Equation (\ref{eqn:family_Lagrangian_boundary}) is illustrated in Figure (\ref{fig: moving_boundary_condition}).
 
 
 For the discussion of later sections we want to re-parametrize the interval by $[0, 1/2]$ instead of $[0,1]$; this can be done easily by picking a smooth function $g: [1/2] \to [0,1]$ such that $g$ is locally constant near $1/2$. We will use $R$ to parametrize this half unit interval.
 
\begin{defin}\label{def: path_of_lagrangian_boundary}
	The family  $\{\sigma_\epsilon\}_{\epsilon \in [0,1)}$ also induces a map $\Phi: \del Z \times [0,1/2] \times I^k \to \Lag(T^*Q)$. 
\end{defin}
 Note this $\Phi$ can be extended with that in Equation (\ref{eqn: Lagrangian moving boundary for gluing}) to define a family of smooth maps on the whole unit interval.

\subsection{Moduli spaces with partial moving Lagrangian boundary condition}\label{subsec: Moduli space NN}
We have constructed two families of moduli spaces of pseudo holomorphic curves in $T^*Q$, the first being the glued curve of $\GG(\sigma, R)$, and the second being the family $\Disc(\sigma_\epsilon, q_0)$. We take the disjoint union of these two families and show that this gives a moduli space of J-holomorphic curves with partial Lagrangian moving boundary condition which gives the cobordism connecting
\[
\bigsqcup_{x\in \Chord(T_{q_0}^*Q, T_{q_1}^*Q)}\Cal{H}(x, q_0, q_1) \times \Disc(\sigma, x),
\]
to the moduli space with constant boundary condition $\Disc(\sigma_0, q_0)$.

\begin{defin}[Moduli space with partial moving Lagrangian boundary conditions]\label{def:NN}
	For a generic family $J_{\tau}$ and a Hamiltonian $H$ on $T^*Q$, the moduli space \textit{$\NN(\sigma,J_{\tau}, H)$}
 of pseudo-holomorphic curves with partial Lagrangian moving boundary condition is the disjoint union of tuples $(u, \tau, R)$ where $u$ is a map from $Z$ to $T^*Q$ that satisfies the following conditions:
	\begin{equation}
	\left \{
	\begin{aligned}
	& (du-X_H\otimes \gamma)^{0,1}=0, \text{for closed 1-form $\gamma \in \Omega(Z)$, such that $\gamma = dt$ when $s \to \pm\infty$  }\\
	& u(z) \in Q \text{ for all $R$ and if $z \in \R \times \{1\} \subseteq \del Z$ lies on the upper boundary of $Z$ },\\
	& u(z) \in T_{q_0}^*Q \text{ for all $R$ and if $z \in [1, \infty]\times \{0\} $ } ,\\
	& u(z) \in \Phi((s,0), \tau, R) \text{ if $z=(s,0) \in \del Z$ lies between $(0,0)$ and $(1,0)$},\\
	& u(z) \in T_{\sigma_R(\tau)(1)}^*Q \text{ if $z \in [-\infty,0]\times \{0\}$ , }\\
	& \lim_{s \to -\infty} u(s,\cdot) = \sigma_R(\tau)(1), \lim_{s\to \infty}u(s, \cdot) = q_0, \text{(asymptotic condition)}\\
	\end{aligned}
	\right \}
	\end{equation}
\end{defin}
To show that $\NN(\sigma, J_{\tau}, H)$ is a manifold with corners of the appropriate dimension for a generic family $J_{\tau}$, we first show that its interior $\NN^*(\sigma, J_{\tau}, H)$ is the zero set of a section of an infinite dimensional vector bundle whose differential is a Fredholm operator that is regular for generic $J_{\tau}$.
It is a standard argument similar to that in \cite[Chapter 3]{MS04}. However, we need to consider a different bundle from that discussed over there. We briefly describe the argument in the following part of this section.

First we define the base of this infinite dimensional vector bundle in the context of Definition \ref{def: path_of_lagrangian_boundary}.
\begin{defin}
	For a family of smooth maps  $\Phi: \del Z \times [0,1] \times I^k = \R \times \{0,1\}\times [0,1] \times I^k \to \Lag(T^*Q)$ and $\tau \in I^k$,  and any $k,p$ such that $kp>2$,  let $\BB^{k,p}_\Phi$ denote the space of tuples $(u, \tau, R)$  where $R \in [0,1]$, $\tau \in I^k$ and $u \in W^{k,p}(Z, T^*Q)$ such that
	$u(s, k) \in \Phi(s,k, R, \tau)$ where $k \in \{0,1\}$. 
	In other words, $\BB^{k,p}_\Phi$ is a family of maps from the Riemann Surface $Z$ to $T^*Q$ with boundary lying in the Lagrangians given by $\Phi(\cdot, R, \tau)$.
\end{defin}
We argue that the above space is a Banach bundle over unit interval $I$ parametrized by $R$. The fiber at each $R$ is $\BB_{\Phi_{R}}$, the space of pairs $(u, \tau)$ such that $u|_{\del Z} \in L_{\Phi_{R, \tau}}$. The tangent space of the fiber $\BB_{\Phi_R}$ over $R \in I$ at $(u,\tau, R) \in \BB_\Phi$ is the space:
\[
 T_{(u,\tau, R)} \BB_{\Phi_R} = W^{k,p}(Z\times I^k, u^*T(T^*Q) )
\]
of all families of smooth vector fields $\xi(z) \in T_{u(z)} T^*Q$ along $u$ parametrized by $I^k$.
 
The trivialization of the bundle over any small neighborhood of $(R, \tau)$ goes as follows: 

For any $\sigma \in C_k(\Omega_{q_0, q_1}Q)$, we may realize the path of Lagrangians from $T^*_{q_0}Q$ to $T^*_{q_1} Q$ as the result of a flow of Hamiltonians $W_\tau$ parametrized by $I^k$ such that for any $\tau \in I^k$, the Hamiltonian $W_\tau$ has linear bounds with respect to the distance given by a certain metric from the zero section outside of a compact set. In other words, for given $\sigma$, there exists large enough constants $r_0, C$ such that 
\begin{equation}
	|W_\tau(q, p)| \leq C\cdot |p| \text{ for $|p| \geq r_0$.}
\end{equation}
Thus the Lagrangian boundary condition $T^*_{q_s}Q = \tilde{\sigma}(\tau)(s)$ is just the time-s flow $\phi_{W_\tau}^s (T^*_{q_0}Q)$ of $T^*_{q_0}Q$ of this Hamiltonian $W_\tau$. 
Lemma \ref{lem: smoothness_in_s} and the definition of $\sigma_\epsilon$ shows that for each $\tau$,  the path of Lagrangians given by the extra parameter $R$,
\[
\Phi_\tau:\del Z \times (0,1/2) \to \Lag(T^*Q),
\] 
maps $((s,0), R)$ to $ \phi_{W_\tau}^{R\cdot \psi(s)} (T^*_{q_0}Q)$, where $\psi$ was defined in Definition \ref{defin:bump_function}. Therefore, we shall define the local trivialization map using a time-$\delta$ flow of $W_\tau$. Note that such a diffeomorphism is further parametrized by $s \in \R$, where $s$ denote the infinite coordinate of the strip $Z = \R \times I$.

\begin{defin}[local trivialization]\label{def: local trivialization}
	For any $R \in (0,1)$ and small enough $\delta$ such that $R+ \delta \in (0, 1)$, the local trivialization map $\phi_{R,\tau, s}(\delta)$ on the space of pseudo-holomorphic maps with Lagrangian moving boundary conditions is given by a time-$\delta$ flow of the Hamiltonian $W_\tau$: \label{eqn:local_trivialization}
\[
   \phi_{R,\tau, s}(\delta): \BB_{\Phi_{R, \tau}} \to \BB_{\Phi_{R+ \delta, \tau}}, u(s, t) \mapsto \phi_{W_\tau}^{\delta \cdot \psi(s)}(u(s,t)),
\]
where $\psi$ is the function defined in Definition \ref{defin:bump_function}.
\end{defin}

The following lemma guarantees the above definition is valid.
\begin{lem}
	$\phi_{R,\tau, s}(\delta)$ sends $\BB_{\Phi_{R, \tau}}$ to $\BB_{\Phi_{R+ \delta, \tau}}$.
\end{lem}

\begin{proof}
	For any $u \in \BB_{\Phi_{R, \tau}}$, Definition \ref{def: sigma_epsilon} gives the boundary condition $u(s,0) \in \tilde{\sigma}(\tau)(R\cdot \psi(s))$, which, from the discussion above Definition \ref{def: local trivialization}, is just $\phi_{W_\tau}^{R\psi(s)}(T_{q_0}^*Q)$. We need to show that $\phi_{R,\tau, s}(\delta)$ maps the above path of Lagrangians into the path given by $\phi_{W_\tau}^{(R+\delta)\psi(s)}(T_{q_0}^*Q)$. 
	By Definition \ref{def: local trivialization}, the following equation holds for any  $u \in \BB_{\Phi_{R, \tau}}$
	\begin{align*}
	 	&\phi_{R,\tau, s}(\delta)(u(s,0)) = \phi_{W_\tau}^{\delta \cdot \psi(s)}(u(s,0)) \\
	 	&\in \phi_{W_\tau}^{\delta \cdot \psi(s)}\circ \phi_{W_\tau}^{R \cdot \psi(s)}(T^*_{q_0}Q) = \phi_{W_\tau}^{(\delta + R)\psi(s)} (T^*_{q_0}Q) = \Phi_{R+ \delta}(s).
	\end{align*}
	Thus $\phi_{R, s}(\delta)$ maps $\BB_{\Phi,R}$ into $\BB_{\Phi_{R+ \delta}}$. Since Hamiltonian flow is a diffeomorphism, this gives an identification of the above two spaces. 
\end{proof}
Therefore, $\phi_{R, \tau,S}$ gives a local trivialization of $\BB_{\Phi_{R,\tau}}$ over $(R-\delta, R+ \delta)$.
\begin{rem}
	$B_{\Phi}$ can further be realized as a vector bundle over $I^k \times I$ whose fiber consists of maps $u$ satisfying $u|_{\del Z} \in L_{\Phi_{R, \tau}}$.
\end{rem}





\subsection{A brief discussion of transversality}\label{sec: transversality argument brief discussion}

To show that $\NN^*(\sigma, J_{\tau})$, the interior of $\NN(\sigma, J_{\tau})$, is a manifold of the appropriate dimension for a generic family $J_{\tau}$, we need to show that it is the zero section of a Fredholm operator of an infinite dimension vector bundle.

Recall that in Section \ref{sec:chain_map}, we defined a Banach manifold $\JJ_{Z,k}(T^*Q)$. It denotes a choice of $k$-family of smooth almost complex structures of contact type outside of a compact set of $T^*Q$ (here we may regard $M^{in}$ as the unit cotangent bundle $S^*Q$.) that is also parametrized by $Z$ and such that for any $\tau \in I^{k}$, the corresponding $Z$-dependent almost complex structure
\[J_{\tau, z} \in \JJ(T^*Q)
\] only depends on $t$ when $s\to \pm \infty$ near the two infinite ends $\xi^1, \xi^2$ of $Z$ parametrized by the $s, t$ coordinates and that they agree. The choice of such a $k$ family follows from the construction given in Section \ref{sec:chain_map}; therefore, it also satisfies the regularity criterion.

Following the standard argument given in \cite[Section 3.2]{MS04} and that in \cite{Fl88}, we consider an infinite dimensional vector bundle $\EE$ over $\BB_\Phi \times \JJ_{Z,k}(T^*Q)$.
 The fiber of $\EE$ at $(u,\tau, R,J_{\tau})$ is the space
 \[
  \EE_{(u, \tau,R, J_\tau)} = W^{k-1, p}(Z\times I^{k+1}, \Lambda^{0,1}\otimes_{J_{\tau}} u^*T(T^*Q)),
\]
 a family parametrized by $\tau$ and $ R$ of $J_{\tau}$-antilinear 1-forms on $Z$ with values in $u^*T(T^*Q)$. 
\begin{lem}
	$\EE$  is a Banach manifold and infinite dimensional vector bundle over $\BB_\Phi \times \JJ_{Z,k}(T^*Q)$.
\end{lem}
\begin{proof}[sketch of proof]
The Banach manifold structure of $\EE$ can be constructed as follows. Fix any smooth metric on $T^*Q$ to identify a neighborhood $\NN(u_{R, \tau})$ of an element $u_{R, \tau}: Z \to T^*Q$ in $\BB^{k,p}_{\Phi_{R,\tau}}$ with a neighborhood of zero in the Banach space $W^{k,p}(Z, u_{R,\tau}^*T(T^*Q))$ via $\xi \mapsto \exp_{u_{R, \tau}} (\xi)$. 
This gives coordinate charts on $\BB^{k,p}_{\Phi_{R, \tau}}$ with smooth transition maps. Note that we identified $\BB_\Phi$ as a vector bundle over $I^{k+1}$ with fiber $\BB^{k,p}_{\Phi_{R,\tau}}$ and base $I^{k+1}$ in the previous section, thus there is a coordinate chart $\NN(u, \tau, R)$ by taking the product of a local coordinate chart $\NN(\tau, R)$ with $\NN(u_{R, \tau}) $. We trivialize the bundle $\EE$ over such a coordinate chart by using a connection $\tilde{\nabla}$ that preserves the family of almost complex structure $J_{\tau}$. 
This constructs local trivializations over open neighborhoods $\NN(u, \tau, R)$ in $\BB_{\Phi}\times \{J_\tau\}$. To extend this over neighborhoods of the form $\NN(u, \tau, R) \times \NN(J_\tau)$ first trivialize over a neighborhood $\{(u, \tau, R)\} \times \NN(J_\tau)$ via the isomorphism:
\[
\Lambda^{0,1}\otimes _{J_{\tau}} u_R^*T(T^*Q) \to \Lambda^{0,1}\otimes_{J_{\tau'}}u_R^*T(T^*Q): \alpha \mapsto \frac{1}{2}(\alpha + J_{\tau'}\circ \alpha \circ j)  
\]
and then extend this trivialization over each slice $\NN(u,\tau, R)\times \NN(J_\tau)$ using parallel transport as before. 
Here $\NN(J_\tau)$ can be obtained by taking small perturbations in $End(T(T^*Q))$ satisfying the conditions in Section (\ref{sec:almost_complex_structure_in_general}) and then taking exponentials with respect to these perturbations around $J_\tau$.
\end{proof}
From now on we use $\UU$ to denote the map from $Z\times I^{k+1} \to T^*Q$, namely $\UU(z, \tau, R) = u(z)$ for $u \in \BB_{\Phi_{R,\tau}}$.

The complex antilinear part of $du$ defines a section $\sS:\BB_\Phi \to \EE$ of this vector bundle:
\[
\sS(\UU) = (\UU, \dbar_{J_{\tau}}(u)),
\]
which is a Fredholm operator as we explain below.
From the definition of $\NN(\sigma)$ above, we realize that its interior $\NN^*(\sigma)$ is just the zero set of this section $\sS$.
\[
\xymatrix{
	&&\EE \ar[d] \\
	& \NN(\sigma)\ar@{^{(}->}[r]&\BB_\Phi \ar^{p}[d] \ar@/^1pc/|-{\dbar}[u] \\
   &	&I
}
\]

In order to show that $\NN^*(\sigma)$ is a manifold of appropriate dimension, we must show that $\sS$ is transverse to the zero section. This means that the image of the differential $d\sS(\UU): T_{\UU}\BB_\Phi \to T_{((\UU), 0)}\EE$ is complementary to the tangent space $T_{\UU}\BB_\Phi$ of the zero section for every $\UU \in \NN(\sigma)$. Given $\UU$, we denote by  
\[
D_{\UU}:= W^{k,p}(Z \times I^{k+1}, u^*T(T^*Q) ) \to W^{k-1,p}(Z\times I^{k+1}, \Lambda^{0,1}\otimes u^*T(T^*Q))
\]
the composition of the differential $d\sS: T_{(u, \tau, R)}\BB \to T_{(u,\tau, R, 0)}\EE$ with the projection
 \[
\pi_u: T_{(\UU, 0)}\EE = T_{\UU}\BB_\Phi \oplus \EE_{\UU} \to \EE_{\UU}.
\]

The operator $D_{\UU}$ is called the vertical differential of the section $\sS$ of the section $\sS$ at $\UU$. The exact definition is as follows:

$D_{\UU}$ can be extended to general maps $Z \times I^{k+1}\to T^*Q$, regardless of $J_{\tau}-$holomorphic or not. However, it depends on a choice of splitting of the tangent space $T_{(\UU,\dbar_J(u))}\EE$ into horizontal and vertical subspaces. Such a splitting corresponds to a connection on $T(T^*Q)$. We assume the connection preserves the almost complex structure so that the fibers of $\EE$ are invariant under point wise parallel transport.  The connection 
\begin{equation}
\tilde{\nabla}_v X := \nabla_vX - \frac{1}{2} J_{\tau} (\nabla_vJ_{\tau})X
\end{equation}
induced by the connection $\nabla$ of the metric works.
Given $\xi \in W^{k,p}(Z \times I^{k+1}, u^*T(T^*Q))$, let 
\begin{equation}
\Psi_\UU(\xi): \UU^*T(T^*Q) \to \exp_\UU(\xi)^*T(T^*Q)
\end{equation}
denote the complex bundle isomorphism, given by parallel transport with respect to $\tilde{\nabla}$ along the geodesics $s \mapsto \exp_{\UU(z,\tau,R)}(s\xi_{z,\tau,R}).$ We define the map 
\begin{align}
&\F : W^{k,p}(Z\times I^{k+1}, u^*T(T^*Q))\to W^{k-1,p}(Z\times I^{k+1},\Lambda^{0,1}\otimes u^*T(T^*Q))\\
 &\F_{\UU}(\xi): = \Psi_\UU(\xi)^{-1}\dbar_{J_{\tau}}(\exp_\UU(\xi)).
 \end{align}
 For general map $\UU:Z\times I^{k+1} \to T^*Q$, $D_{\UU}$ is defined as
 \begin{equation}
 D_{\UU}\xi := D\F_u(0)\xi.
 \end{equation}
Similar to \cite[Proposition (3.1.1)]{MS04}, we see that
\begin{equation}\label{eq: explicit_form_of_Du}
	D_{\UU}\xi = \frac{1}{2}(\nabla \xi + J_{z, \tau} \nabla \xi \circ j_Z ) - \frac{1}{2} J_{z,\tau} (\nabla_\xi J)(\UU) \del_J(\UU),
\end{equation}
 where $j_Z$ acts on any tangent vector $\zeta = (\zeta_z, \zeta_\tau, \zeta_R) \in T_{(z,\tau, R)} (Z\times I^{k+1}) $ by $(j_Z (\zeta_z), \zeta_\tau,\zeta_R)$.
 The above equation implies that $D_{\UU}$ is a real linear Cauchy-Riemann operator and hence a Fredholm operator as explained in \cite[Appendix C]{MS04} or \cite[Section 8]{AbSe09}. For a generic infinitesimal perturbation $K_{\tau, z} \in End(TM, J_{\tau,z}, \omega)$, The resulting linearized operator at $J_{\tau, z}\exp(-J_{\tau,z}K_{\tau, z})$ with the expanded dense Banach subspace $\Cal{K}$ of infinitesimal perturbations given by
 \begin{align}
   \Cal{K}\oplus W^{k,p}(Z\times I^{k+1}, \UU^*T(^*Q)) \to W^{k-1, p}(Z\times I^{k+1}, \UU^*(T^*Q))\\
 	(K, \xi) \mapsto \frac{1}{2} K_{z,\UU(z,\tau,R)} \circ (d\UU - \xi \otimes \gamma) \circ j_Z + D_\UU \xi,
  \end{align}
  is also Fredholm.
 Thus we arrive at the following proposition:
 \begin{prop}
 	For a generic family of almost complex structure $J_{\tau}$, $\F$ is a Fredholm operator of index $ \mu(q_0) - \mu(q_1)+|\sigma| +1 $ and the interior $\NN^*(\sigma, J_{\tau})$ is a smooth manifold of same dimension. 
 \end{prop}
 
 \begin{rem}
 Transversality would come from the surjectivity of $D_{\UU}$ for every $\UU \in \NN^*(\sigma)$ for a given regular $J_{\tau}$ as a $W^{k,p}$-neighborhood of zero in $\F_\UU^{-1}(0)$ is diffeomorphic to a $W^{k,p}$-neighborhood of $(u,\tau,R)$ via the map $\xi \mapsto \exp_\UU(\xi)$ and that $\F_\UU^{-1}(0)$ intersects a small neighborhood of zero in a smooth submanifold due to the surjectivity of $D_{\UU}$ and standard Sard-Smale argument.   The details of this argument is similar to the proof of transversality shown in \cite[Chapter 3]{MS04}.
 
  \end{rem}



%
%
%
%


\subsection{Homotopy equivalence}
%

To show the homotopy equivalence of the identity functor $\mathit{Id}$ and $f\circ F\circ \eta$ on $C_{*}(\Omega_{q_0, q_1}Q)$, we construct maps from the cubical chains of maps in $Q$ to itself:
\begin{equation}\label{eqn: homotopy_map}
	Hp: C_{*}(\Omega_{q_0, q_1}Q) \to C_{*+1}(\Omega_{(q_0, q_1)}Q),
\end{equation}
induced by  $\NN(\sigma)$ such that $ (f\circ F\circ \eta - \mathit{Id})_{*} = d \circ Hp + Hp \circ d$ in 
$ C_{*}(\Omega_{q_0, q_1}Q)$.

\begin{rem}
The geometric idea is that for every path in $\sigma$, concatenating $\Disc(\sigma_\epsilon)$ with the remainder of the path after the $\epsilon$-th part, we can homotope to the original $(f\circ F\circ \eta)_*(\sigma)$ that is the set $\bigcup_{x \in \Chord(T^*_{q_0},T^*_{q_1} )}\Disc(\sigma, x) \times \Cal{H}(x, q_0, q_1)$, represented by the first picture in Figure (\ref{fig: Gluing of two moduli spaces}).
\end{rem}

From the previous subsection and Sard-Smale argument we know that $\NN^*(\sigma, J_{\tau})$ is a smooth manifold of dimension $|\sigma| + \mu(q_0) - \mu(q_1) + 1$ for a generic family of almost complex structures $J_{\tau}$.
We list the boundary components of its compactification $\NN(\sigma, J_\tau)$ below.
 For any $\sigma' \subseteq \del \sigma$, we obtain a map:
\begin{equation}\label{eqn:boundary_from_sigma_itself}
	\NN(\sigma', J_{\tau}) \to \NN(\sigma, J_{\tau}).
\end{equation}

 Also, by gluing a half disc with output $x$ and input $\sigma$ to a half disc with input $x$, we get the map
\begin{equation}\label{eqn:gluing_end_of_NN}
\Disc(\sigma, x) \times \Cal{H}(q_0, x, q_1) \to \NN(\sigma,J_{\tau}).
\end{equation} 
Maps with the above form for all possible $x \in \Chord(T^*_{q_0}Q, T^*_{q_1}Q)$ is identified with the inverse image of the boundary at $R = 1$ of the unit interval $[0,1]$ that gives the $1$-parameter family of holomorphic maps in the definition of $\NN(\sigma, J_{\tau})$.

Similarly, taking the inverse image of the boundary at $R =0$, we obtain the following map:
\begin{equation}\label{eqn:identity_end_of_NN}
	 \Disc(Id_{q_0}, q_0)  \to \NN(\sigma, J_{\tau}).
\end{equation}

\begin{prop}\label{prop: boundary_components_of_NN}
	
	Fixing the Floer data and a generic choice of family of almost complex structures $J_{\tau}$, $\NN(\sigma, J_{\tau})$ is a smooth manifold with corners and its codimension $1$ boundary components consist of the images of the inclusion maps given by equation (\ref{eqn:boundary_from_sigma_itself}), (\ref{eqn:gluing_end_of_NN}), (\ref{eqn:identity_end_of_NN}), namely
	
	\begin{align}\label{eqn:boundary_components_of_homotopy}
	&\coprod_{x} \Disc(\sigma, x) \times \Cal{H}(q_0, x, q_1),  \\
	&\Disc(\sigma_0, q_0) = \Disc(\mathit{Id}_{q_0}, q_0), \\
	&\NN( \sigma') \text{ for $\sigma' \subseteq \del \sigma$}.
	\end{align}	\qed
\end{prop}
Consider the family of evaluation map:
\begin{equation}
	ev_{\tau, R}: \NN(\sigma)_{\tau, R} \to \Omega(q_0, q_1),
\end{equation}
 which for every $R, \tau$ takes every pseudo-holomorphic strip $u:Z \to T^*Q$ to the path along $Q$ between $q_0$ and $q_R$ obtained by restricting to the component of $\del Z$ that stays within $Q$
  and the concatenation 
  with the remainder of the original path from $q_R$ to $q_1$ that is given by $\sigma$. 
As a whole this gives a map:
\[
ev: \NN(\sigma) \to C^\infty(I^{k+1},\Omega(q_0, q_1))
\]
parametrized by $R,\tau$ and induce a map $ev_*:  C_{*}(\NN(\sigma) ) \to C_{*+k+1}(\Omega_{(q_0, q_1)}Q)$.

The map in equation (\ref{eqn:gluing_end_of_NN}) gives a commutative diagram:
\begin{equation}
\xymatrix{
	&\Disc(\sigma, x) \times \Cal{H}(q_0, x, q_1) \ar[r] \ar^{ev}[d] & \NN(\sigma) \ar^{ev}[d] \\
	& C_{k}(\Omega_{(q_0, q_1)}Q) \ar[r] &C_{k+1}(\Omega_{(q_0, q_1)}Q),	
}
\end{equation}
where the above horizontal map is the inclusion of a boundary stratum and the left vertical map is projection onto the second factor and the remaining two maps are evaluations.

Similarly, the map in equation (\ref{eqn:identity_end_of_NN}) gives the following commutative diagram:
\begin{equation}
\xymatrix{
	&  \Disc(Id_{q_0}, q_0)  \ar[r] \ar^{ev}[d] & \NN(\sigma) \ar^{ev}[d] \\
	&   C_{k}(\Omega_{(q_0, q_1)}Q)  \ar[r] & C_{k+1}(\Omega_{(q_0, q_1)}Q) ,	
}
\end{equation}
where the first horizontal map is the inclusion of the boundary stratum $\Disc(Id_{q_0}, q_0)$; the left vertical map is also the evaluation map on the component of the boundary of $Z$ that stays inside $Q$; the bottom horizontal map is concatenation with the path given by $\sigma(\tau)$. 

The following commutative diagram is obtained by equation (\ref{eqn:boundary_from_sigma_itself}) in a similar fashion:
\begin{equation}
	\xymatrix{
		&  \NN(\sigma')  \ar[r] \ar^{ev}[d] & \NN(\sigma) \ar^{ev}[d] \\
		&   C_{k}(\Omega_{(q_0, q_1)}Q)  \ar[r] & C_{k+1}(\Omega_{(q_0, q_1)}Q).	
	}
\end{equation}

Notice that $\NN(\sigma, J_{\tau})$ is by definition a family of moduli spaces parametrized by $R \in [0,1]$. Taking $R = 1$ we get the boundary strata represented by Equation (\ref{eqn:gluing_end_of_NN}), given by standard gluing techniques; when $R = 0$, this is the boundary given by equation (\ref{eqn:identity_end_of_NN}); we also incorporate the boundary component that arise from the boundary chains $\del \sigma$ of chains $\sigma$ as in Equation (\ref{eqn:boundary_from_sigma_itself}). Thus the following proposition.

\begin{lem}\label{lemma: homotopy_by_evaluation}
	There exist fundamental chains $[\NN(\sigma)] \in C_{*}(\NN(\sigma))$ such that the assignment induced by the evaluation map
	\[
	Hp(\sigma) = ev_*([\NN(\sigma)])
	\]
	defines the homotopy described in equation (\ref{eqn: homotopy_map}).
\end{lem}
\begin{proof}
The strategy is to build up from chains of lower dimensions, starting with the fundamental chains $[\NN( \sigma')]$ for $\sigma' \subseteq \del \sigma$.
Let us consider moduli spaces $\NN(\sigma)$ whose boundary only has codimension $1$ strata, which must therefore be products of closed manifolds. By taking the product of the fundamental chains of the factors in each boundary stratum, we obtain a chain in $C_*(\NN(\sigma))$ with appropriate signs as before:
\begin{equation}\label{eqn:boudary_cycles_of_homotopy}
	\sum_{x \in \Chord(T_{q_0}^*Q, T_{q_1}^*Q)} (-1)^{\dagger} [\Disc(\sigma, x)] \times [\Cal{H}(q_0, x, q_1)]\\
	 + \sum_{\sigma' \subseteq \del \sigma}[\NN( \sigma')] + [\Disc(\sigma_0, q_0)].
\end{equation}

Lemma (\ref{lemma: homotopy_by_evaluation}) implies that the above equation is a cycle. We now define the fundamental chain 
\[  [\NN(\sigma)]
\]
to be any chain whose boundary is equation (\ref{eqn:boudary_cycles_of_homotopy}).

By definition of homotopy on chains, we realize that up to signs which we shall ignore, 
   proving Hp gives a homotopy on $C_*(\Omega_{(q_0, q_1)}Q)$ between $\mathit{Id}$ and $f\circ F\circ \eta$, is equivalent to showing that $(f\circ F\circ \eta - \mathit{Id})(\sigma) = \del \circ Hp(\sigma) + Hp(\del \sigma) \in C_*(\Omega_{(q_0, q_1)}Q)$. 
   
    We note $\del \circ Hp(\sigma)$ consists of the image of the boundary components listed in equation (\ref{eqn:boundary_components_of_homotopy}), the first of which corresponds to $f\circ F\circ \eta$, the second corresponds to $\mathit{Id}_*$ and the third corresponds to $\sum_{\sigma' \subset \del \sigma} ev_*([\NN( \sigma')]) = Hp(\del \sigma)$.
\end{proof}

\section{Appendix: $C^0$-estimates} \label{Appendix: C^0_bound}



In this section we denote by $M^{in}$ a Liouville domain with contact boundary and $M$ its completion as in Equation \eqref{eq:completion} , namely $M = M^{in} \cup \del M^{in} \times [1, \infty)$ with a collar neighborhood of $\del M^{in}$ inside $M^{in}$ modeled by $\del M^{in} \times (0,1]$. Let us use $r$ to denote the radial coordinate of the completion and use $\alpha = \theta|_{\del M^{in}}$ to denote the contact form on $\del M^{in}$.

In order to show that our moduli space of pseudo-holomorphic curves inside $M$ is compact, we need to prove Gromov compactness by restricting to a compact subset of $M$, which comes from a $C^0$ bound on $\rho=r\circ u$. Here $u$ is the pseudo-holomorphic curve from  $\D \setminus \{\xi_1, \xi_2\}$ to $M$, where $\D$ is the unit disc. We denote this domain of $u$ by $\Sigma$ in this section, and it's easy to see that this is a strip
\[
\Sigma \simeq \R \times [0,1].
\]
Let $z = (s, t)$ denote the coordinate for $\R\times I$.
 
A $C^0$-bound is achieved if $\rho$ is constraint within a compact subset of in $M$.
Suppose the moduli space of pseudo-holomorphic strip has boundary conditions given by a family of moving Lagrangians $L_s$ for $s \in [0,1]$ in the lower boundary of $\Sigma$; this family of Lagrangians in $M$ has Legendrian boundaries, namely the infinite ends of each $L_s$ is a product of Legendrian submanifolds of $\del M^{in}$ and an open interval:
 \[
 L_s \cap \del M^{in} \times (0,\infty) =\Lambda_s \times (0,\infty),
 \]
  near $\del M^{in}$, where $\Lambda_s$ is a family of Legendrian submanifolds generated by a Hamiltonian $H_0^s$ parametrized by $s$ on $\del M^{in}$. In later discussions, we just write $H_0$ if there is no confusion. Thus there exists a vector field $X_{H_0}$ in $\del M^{in}$ such that
   \[
  i_{X_{H_0}} d\alpha=-dH_0,
  \]
  and $\Lambda_s$ is the time-$s$ flow of $X_{H_0}$ in $\del M^{in}$. A calculation similar to that of exercise 3.48 in \cite{MS95} gives 
  \begin{equation} \label{eqn:$X_{H_0}$}
  X_{H_0}=H_0\cdot Y+ Z,
 \end{equation}
  where $Y$ is the Reeb vector field on $\del M^{in}$ and $Z$ lies in the contact distribution $\kappa =\ker(\alpha)$ such that 
  \[
  i_Z d \alpha|_{\kappa} = dH_0|_{\kappa}. 
  \]
Now consider a Hamiltonian $H^l$ on $M$ that is linear with respect to the radius on cylindrical ends $\del M^{in}\times (1,\infty)$,  given by $H^l (q,r): =r\cdot  H_0(q)$ for $q \in \del M^{in}$. This gives rise to a vector field on the conical end 
\begin{equation}\label{eqn:$X_{H^l}$}
X_{H^l}(q,r) = (X_{H_0}(q), -r dH_0(Y)\partial_r) \text{ when $r \gg 0$.}
\end{equation}
\begin{rem}
	This Hamiltonian $H^l$ is the one denoted by $W$ in Section \eqref{subsec: Moduli space NN}.
\end{rem}
As for the properties of the Hamiltonian that enters in the definition of the wrapped Fukaya category, we consider functions $H$ on $M$ parametrized by $\Sigma$ such that:
\begin{align} \label{eq:class-of-h}
& \parbox{30em}{$H$ is positive everywhere. Moreover, $H (r,q) = c(s)\cdot r^2$ on the infinite cone for some smooth $c:\R \to \R_+$.} 
\end{align}
Let $X_H$ be the Hamiltonian vector field of $H$. On the infinite cone it is a multiple of the Reeb vector field:
\[
X_H = 2c(s) r\cdot (0,Y),
\]
 where $Y$ is the Reeb vector field associated to the contact form $\alpha = \theta|\partial M^{in}$ and the first coordinate corresponds to the component in the direction of $\del_r$. By restriction on $\del M^{in}$, which corresponds to $r = 1$, one sees that $X_H$ satisfies $X_H|_{\partial M^{in} }=c(1)Y$.

On both infinite ends of the domain $\Sigma$ there are asymptotic Hamiltonian chords $x_0$ and $x_1$ and the curves satisfy the in-homogeneous Cauchy Riemann equation (\ref{eq:CauchyRiemannEq}), which we write here for the reader's convenience:
\begin{align}\label{eq:condition_on_u}
& \parbox{30em}{The map $u: \Sigma \to M$ satisfies $(du-X_H\otimes \gamma)^{0,1}=0,$ where $u(s,1)$ lies in a fixed Lagrangian $K$, $ u(s,0) \in L_0$ when $s\gg0$, $u(s,0) \in L_1$ when $s\ll0$ and $u(s,t) \in L_s $ for $s\in[0,1]$ and $\gamma$ is a closed $1$-form on $\Sigma$ such that $\gamma = dt$ near both infinite ends and that $\gamma = dt$ in a small neighborhood of $[0,1]\times \{ 0\}$. }
\end{align} 
Namely we restrict the moving boundary Lagrangian condition to an interval on the lower boundary of $\Sigma$.

\begin{figure}\label{fig:Maximum_principle_for_glued_disc}
	\centering
	\begin{tikzpicture}
	\newcommand*{\tinyradius}{.03}
	\begin{scope}[shift = {(8, -.5)}]

	\draw[line width = 0.4] (3,1.5) to (-3,1.5);
	\coordinate[label=above: $K$]() at (0,1.5);
	\draw[line width =0.4] (3,1.5) to (-3,1.5);
	\draw [->] (3.4,.5) to (2.9,.5);
	\draw[->](-2.5,.5) to (-3,.5);
	\coordinate[label=above:\small{$(du-X_H\otimes \gamma)^{0,1} = 0$ }]() at (0,0.5);
	
	\draw[line width = 0.4] (3,0) to (-3, 0);
	\coordinate[label=below: \tiny{$L_0$}]() at (1,0);	
	\draw[fill=blue](.5, 0) circle (\tinyradius);
	\draw [red, line width=.6](.5,0) to (-.5,0);
	\coordinate[label=below: \tiny{$L_{\sigma(\tau)(s)}$}] () at (0,0);
	\draw[fill=blue](-.5, 0) circle (\tinyradius);
	\coordinate[label=below:\tiny{$L_1$}]() at (-1,0);
	
	\end{scope}

	\draw [->, line width=1] (4.5,0) to (3,0);
	\draw (-2, 1) circle (.05);
	\draw(-2, -.5) circle (.05);
	\draw (-2,1) to (-2,-.5);
	\draw (-2,1) to (1,1);
	\draw(-2,-.5) to (1,-.5);
	\draw [red](1,1)[out=(0:2), in =(0:2)] to (1,-.5);
	\coordinate[label=left:$K$]() at (-2,0);
	\coordinate [label=right:$L_{\sigma(\tau)}$]() at (2.5,0.3);
	\draw [->] (0,-.3) to (0,.6) node[left]{$x_H$};
	\end{tikzpicture}
	\caption{Maximum Principle for glued disc}
\end{figure}
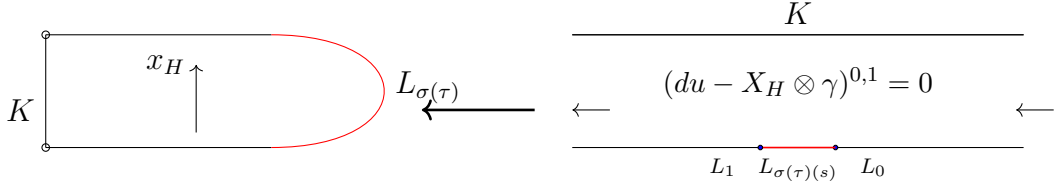

With the above geometric settings, we state the following $C^0$ estimate for moduli space of pseudo-holomorphic curves with moving Lagrangian boundary with Legendrian ends inside a Liouville manifold.
\begin{thm}\label{thm: C^0 bound on strip}
	Given a $2$-dimensional disk $\Sigma$ with punctures $\{ \xi_1,\xi_2 \}$ on the boundary, every element in the moduli space of pseudo-holomorphic curves in an open Liouville manifold $M$ from the domain $\Sigma$ and moving Lagrangian boundary conditions generated by the flow of a Hamiltonian with linear growth rate at the cylindrical ends has a priori $C^{0}$ bound. Namely, there is a choice of Hamiltonian data $H, J$ and closed $1$-form $\gamma$ on $\Sigma$ satisfying the conditions of \ref{eq:condition_on_u} such that all solutions of the inhomogeneous pseudo-holomorphic equation 
	\begin{equation}
	(du - X_H\otimes \gamma)^{0,1} = 0
	\end{equation}
	stay within a compact subset of $M$.
\end{thm}

\begin{proof}
We use the Hopf maximum principle on a function that is monotone and bijective with $\rho$ to show that when the radial coordinate is too large, the maximum can't live on the moving part of the boundary of the domain $\Sigma$ nor the interior, thus forcing any solution of the equation to be within a chosen compact space inside $M$. The proof consists of showing non-existence of local maximum at the moving boundary and in the interior.



First, we work with an almost complex structure $J$ of contact type which satisfies Equation \eqref{eq:contact type} near the cylindrical ends and $\gamma = dt$:
\begin{align*} \partial_s \log(\rho )
&=\partial_s(\log(r \circ u))\\
&=\frac{dr}{r}(\del_s u)
= \frac{dr}{r} \circ(-J)(\partial_t u-  X_H)\\
&=\frac{\theta}{r}(\partial_t u-2c\rho X_\rho)
=\alpha(\partial_t u)-2c\rho\alpha(X_\rho)\\
&=\alpha(\partial_t u)-2c(s)\rho.
\end{align*}
The third equality comes from the $J$-holomorphic curve equation; the fourth equality comes from the condition \eqref{eq:contact type} and the last equality is due to the fact that $X_\rho= Y$.
Similarly,
\begin{align*}
\partial_t \log(\rho)&=
\frac{dr}{r}\cdot (-J) \cdot J \cdot (\partial_t u)
=\alpha(J \partial_t u)\\
&=\alpha(-\partial_s u + J cX_\rho)
=\alpha(-\partial_s u)
\end{align*} 

Since $H_0$ is smooth on the compact set $\del M^{in}$ we have a bound on $H_0$, let $\mu$ denote twice the maximum value, 
\begin{equation}\label{eqn: $mu$}
\mu:= 2\cdot \text{max}H_0.
\end{equation}
Consider the following auxiliary function:
\begin{equation}\label{eq:auxilary function}
h_\mu(s,t)=
\begin{cases}
e^{\frac{-s^2}{1-s^2}}\cdot (\frac{t^2}{2} -t )\mu 			& s\in(-1,0)\\
(\frac{t^2}{2}-t)\mu			& s\in [0,1]\\
e^{\frac{-(s-2)^2}{1-(s-2)^2}}\cdot (\frac{t^2}{2}-t) \mu  & s \in (1,2) \\
0				&otherwise\\
\end{cases}
\end{equation}
\begin{rem}
	The above construction has 2 components, the first part is $(\frac{t^2}{2} -t )\mu$ with the property that  $d^c h(\del_s)\geq \mu$ on all of the lower boundary of $\Sigma$ with $d^c h(\del_s)$ equals $\mu$ when $s\in[0,1]$ (the portion of moving boundary illustrated in red in figure (\ref{fig:Maximum_principle_for_glued_disc})) and $d^c h(\del_s)$ vanishes on the upper boundary of $\Sigma$ which, is mapped into a constant Lagrangian $K$.  The second part is just a mollifier function so that $h_\mu(s,t)$ vanishes near infinite ends.
\end{rem}
We want to show that there is a $C^0$-bound on the function $\log(\rho)-h_\mu$ to give a $C^0$-bound on $\rho$ since $h_\mu$ itself is bounded above on $\Sigma$, thus an upper bound on $\log(\rho)-h_\mu$ would imply an upper bound on $\rho$; to this end we apply the Hopf maximum principle to $\log(\rho)-h_\mu$. 
To this end we present the following two lemmas:

\begin{lem}\label{lemma: max not on boundary}
	Suppose there is an auxiliary function $h_\mu$ defined in Equation \eqref{eq:auxilary function} and a pseudo-holomorphic strip $u: \Sigma \to M$ with moving boundary Lagrangian conditions on a Liouville domain $M$ satisfying the condition in Theorem \ref{thm: C^0 bound on strip}. When the radial coordinate $r$ is large enough, the maximum value of $\log(\rho) - h_\mu$ does not lie on the boundary of the domain $\Sigma$.
\end{lem}

\begin{lem}\label{lemma: max not in interior}
	Suppose we are equipped with a pseudo-holomorphic strip as in Lemma \ref{lemma: max not on boundary}, when $r$ is large enough, the maximum value of $\log(\rho) - h_\mu$ does not lie on the interior of the domain $\Sigma$. 
\end{lem}
Thus, from the above two lemmas, there is an upper bound on the radial coordinate of the image of the strip and stays within a compact set.
 
\begin{proof}[Proof of Lemma \ref{lemma: max not on boundary}]
We first look at the partial derivatives of $h_\mu$ with respect to $s$ and $t$.
 \[
\del_t h_\mu=
\begin{cases}
e^{\frac{-s^2}{1-s^2}}\cdot (t-1)\mu 			& s\in(-1,0)\\
(t-1)\mu			& s\in [0,1]\\
e^{\frac{-(s-2)^2}{1-(s-2)^2}}\cdot (t-1)\mu  & s \in (1,2) \\
0				&otherwise\\
\end{cases}
\]

\[
\del_s h_\mu=
\begin{cases}
e^{\frac{-s^2}{1-s^2}}\cdot \frac{-2s}{(1-s^2)^2}(\frac{t^2}{2}-t)\mu 			& s\in(-1,0)\\
0			& s\in [0,1]\\
e^{\frac{-(s-2)^2}{1-(s-2)^2}}\cdot \frac{-2(s-2)}{(1-(s-2)^2)^2}(\frac{t^2}{2}-t) \mu  & s \in (1,2) \\
0				&otherwise\\
\end{cases}
\]
Thus we obtain the following equations:
\begin{equation}
\begin{split}
d^c \log(\rho )&=d\log(\rho)\circ j  \\
&=-\del_s \log(\rho) dt +\del_t \log(\rho)ds\\
&=(-\alpha(\del_t u)+2c(s)\rho) dt -\alpha(\del_s u) ds\\
&=-u^*\alpha +2c(s)\rho \cdot dt  \label{eqn: dchmu}\\
d^ch_\mu &=-\del_s h_\mu dt +\partial_t h_\mu ds	
\end{split}
\end{equation}
Let $\xi$ denote $\del_s u(s, 0)$, the vector field that is mapped to $M $ from 
the lower boundary of $\Sigma$, note $\xi$ is tangent to the Hamiltonian flow of $L_s$ when $s\in[0,1]$. We may decompose $\xi$ into three components:
\begin{equation}\label{eqn:boundary_tangent_vector}
\xi=a\del_r+X_{H_0}+Z',
\end{equation}
where the first component is a radial vector, $X_{H_0}$ is the Hamiltonian vector field on $\del M^{in}$ generated by $H_0$ and $Z'$ is a vector in the tangent space of the Legendrian given by the intersection of the moving Lagrangian and $\del M^{in}$ at $u(s, 0)$, namely $Z' \in T \Lambda_s \subseteq \kappa=\ker(\theta)$. Recall from  equation (\ref{eqn:$X_{H_0}$}) that $X_{H_0}=H_0\cdot Y +Z$ where $Z\in \kappa$ also. 

From Equations (\ref{eqn: dchmu}) and (\ref{eqn:boundary_tangent_vector}), it is easy to see:
\begin{align*}
(d^c\log\rho-d^ch_\mu)(\partial_s)
&= -\alpha(\xi)+ 2c(s) dt (\partial_s) -\del_s h_\mu dt(\del_s) -\del_t h_\mu\\
&=-\alpha(a\cdot \partial_r+ X_{H_0}+Z')+0 + 0-\del_t h_\mu\\
&=-\alpha(H_0Y) -\del_t h_\mu\\
&=-H_0-\del_t h_\mu=-H_0+(1-t)\mu\cdot f(s).
\end{align*}
Note $0\leq f(s) \leq 1$, and  $H_0(q)$ is strictly less than $\mu$ for any $q \in \del M^{in}$ by definition of $\mu$ in equation (\ref{eqn: $mu$}), we get $(d^c\log \rho-d^ch_\mu)(\partial_s) > 0$ when $t=0$;  in other word, the gradient of $\log\rho - h_\mu$ paired with a inner normal vector on the lower part of $\del \Sigma$ where the Lagrangian boundary moves from $L_0$ to $L_1$ is positive.

If the maximum value appears on $(s,0)$ where $s \notin [0,1]$, say $s <0$. Then the maximum value appears on the fixed Lagrangian $L_0$. We may apply the inverse Hamiltonian flow of $H^l$ on the lower boundary to obtain a strip with constant Lagrangian boundary condition:
 \begin{equation}
 \tilde{u}(s,t) = 
 \begin{cases}
 u(s,t) \text{ when $s <0$, }\\
 \phi_{-H^l}^{(1-t)s}(u(s,t)) \text{ $s \in [0,1]$,}\\
 \phi_{-H^l}^{1-t}(u(s,t)) \text{ when $s >1$.}
 \end{cases}
 \end{equation}
 From the $C^0$-estimate on the radial coordinate derived by energy estimates given by Lemma 7.2 in \cite{Ab12}, we know that the maximum value of $\log(\rho) - h_\mu$is bounded a priori.
 
When $t=1$, if $K$ is compact, $C^0$-bound of $\rho$ is automatic; otherwise $K$ is a Lagrangian with Legendrian boundary near the cylindrical ends, thus $(d^c\log \rho-d^ch_\mu)(-\partial_s)\equiv0$ as $u(s,1)\in K$ and $K\cap \del M \times [1,\infty)=\Lambda\times[1,\infty) $ for some $\Lambda$ Legendrian on $\del M$. Since $ H_0 < \mu$, the maximum of $\log\rho-h_\mu$ can't occur on the boundary of $\Sigma$ by the Hopf Lemma, which claims that if the maximum occurs at the boundary for $\log\rho - h_\mu$, then the partial derivative on the outward normal direction would be strictly positive.
\end{proof}

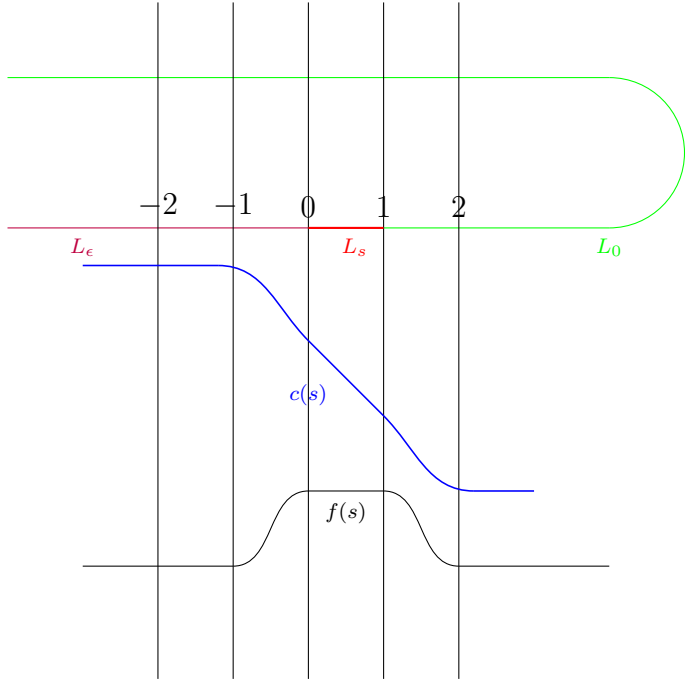
\begin{figure}\label{fig: function_requirements}
	\begin{tikzpicture}
	\draw [green] (-4, 2) to (4,2);
	\draw [green](4,0) arc (-90:90: 1);
	\draw [green](4,0) to (1,0);
	\draw [purple] (1/4,0) to (-4,0);
	\coordinate [label= below: \textcolor{green}{\tiny{$L_0$}}] () at (4,0);
	\coordinate [label= below: \textcolor{purple}{\tiny{$L_\epsilon$}}] () at (-3,0);
	
	\coordinate[label = above: $-2$]() at (-2, 0);
	\draw (-2, -6) to (-2, 3);
	\coordinate[label = above: $-1$]() at (-1, 0);
	\draw (-1, -6) to (-1, 3);
	\coordinate[label = above: $0$]() at (0, 0);
	\draw (0,-6) to (0,3);
	\coordinate[label = above: $1$]() at (1, 0);
	\draw (1,-6) to (1,3);
	\coordinate[label = above: $2$]() at (2, 0);
	\draw (2,-6) to (2, 3);
	\draw [red, line width=0.8] (1,0) to (0, 0);
	\coordinate[label = below: \textcolor{red}{\tiny{$L_s$} } ] () at (2/3, 0);
	\begin{scope}[shift ={(0,-.5)}]

	\draw [blue, line width = 0.6] (3,-3) to (2.2,-3);
	\draw [blue,line width = 0.6] (2.2,-3) to [out = 180, in = -45] (1,-2) ;
	\draw [blue, line width = 0.6] (1,-2) to (0,-1);
	\draw[blue, line width = 0.6] (0,-1) [out = 135, in = 0] to (-1.2, 0);
	\draw [blue, line width = 0.6] (-1.2,0) to (-3, 0);
	\coordinate[label = above: \textcolor{blue}{\tiny{$c(s)$}}] () at (0,-2);
	\end{scope}
	
	\begin{scope}[shift ={(0,-4.5)}]
	\draw (-3,0) to (-1,0);
	\draw (-1,0) to [out = 0, in = 180] (0,1);
	\draw (0,1) to (1,1);
	\draw (1,1) to [out =0, in = 180] (2,0);
	\draw (2,0) to (4,0);
	\coordinate [label = below:\tiny{$f(s)$}] () at (0.5, 1);
	
	\end{scope}

	\end{tikzpicture}
	\caption{$H=c(s)\cdot r^2$ and the bump function.}
\end{figure}

\begin{proof}[Proof of Lemma \ref{lemma: max not in interior}]
We now use the Laplacian to show that the maximum can't lie in the interior:
\begin{equation}\label{eq:Laplacian}
\begin{split}
\Delta(\log\rho-h_\mu)&=-dd^c(\log\rho-h_\mu)(\del_s, \del_t)
=u^*d\alpha -\partial_s[2c(s)\rho] ds\wedge dt-\Delta h_\mu \\
&=d\alpha(\del_su, \del_t u)-2c'(s)\rho - 2c(s)\del_s\rho-\Delta h_\mu \\
&=d\alpha(\del_su, J\del_s u+ cX_\rho)-2c'(s)\rho -2c(s)\del_s\rho - \Delta h_\mu \\
&=d\alpha(\del_su, J\del_s u)-2c'(s)\rho- 2c(s)\del_s\rho - \Delta h_\mu
\end{split}
\end{equation}
Note we had a decomposition of the tangent space
\begin{equation}
	T_{(r,y)}M = \kappa \oplus \R \cdot Y \oplus \R \del_r,
\end{equation}
 thus if we let $\del_s u=a\cdot Y+ b\del_r +Z$ where $Z \in \kappa$ as before. Then $J \del_s u= -ar\del_r+\frac{b}{r}\cdot Y+J \cdot Z$. 
Also $\omega=d(r\cdot \alpha)=dr\wedge \alpha+ rd\alpha$.
Therefore 
\begin{equation}
\begin{split}
d\alpha(\del_s u, J \del_s u)&=d\alpha(a\cdot Y+ b\del_r +Z, -ar\del_r+\frac{b}{r}\cdot Y+J \cdot Z)\\
&=d\alpha(Z, J\cdot Z)=\frac{(\omega-dr\wedge \alpha)(Z, J\cdot Z) }{r}\\
&=\frac{\omega(Z, J\cdot Z)}{r}=\frac{|Z|^2_J}{r}
\end{split}
\end{equation}
From the above equation we get 
\begin{equation}
\begin{split}
\Delta(\log\rho-h_\mu)&=d\alpha(\del_su, J\del_s u)-2 c'(s) \rho - 2c(s)\del_s\rho -\Delta h_\mu\\
&=\frac{|Z|_J^2}{r}-2c'(s)\rho- 2c(s)\del_s\rho -  \Delta h_\mu
\end{split}
\end{equation}
We notice that in applying the Hopf maximum principle, we can ignore the term $\del_s \rho$ at the maximum point. Therefore we only care about $\frac{|Z|_J^2}{r}- 2c'(s)\rho - \Delta h_\mu$.

When $s\in[0,1]$, $\Delta h_\mu=\mu$, thus the above term becomes 
\begin{equation*}
	\frac{|Z|^2}{r}-2c'(s) \rho-\mu.
\end{equation*} 
When $c'(s)$ is very negative such that $2c'(s) +\mu < 0$, if $\rho \geq 1$, the maximum of $\log\rho-h_\mu$ can't be in the interior, since in that case at the interior maximum point,$\Delta(\log\rho - h_\mu) <0$, yet $0> \Delta(\log\rho-h_\mu)+2c'(s)+\mu=|Z|^2\geq 0$ a contradiction. Thus $\log\rho - h_\mu$ on the image of the pseudo holomorphic strip lies inside $M^{in}$.  

When $s \in [-1,0]$,
\begin{equation*} 
\Delta (\log\rho-h)= |Z|^2-c'(s)-(1+\frac{4s^2-2(1-s^2)^2+8s^2(s^2-1)}{(1-s^2)^4}(\frac{t^2}{2}-t) ) e^{\frac{-s^2}{1-s^2}}\mu.
\end{equation*}
Some Calculation shows that $\frac{4s^2-2(1-s^2)^2+8s^2(s^2-1)}{(1-s^2)^4}\cdot e^{\frac{-s^2}{1-s^2}} \in[-200, 200]$ for $s\in[-1,0]$; while $-\frac{1}{2}\leq(\frac{t^2}{2}-t))\leq0$.
Similarly for $s\in [1,2]$.

Thus if we impose the condition $c'(s)\leq -100 |\text{max}_QH_0|$, we get upper bound of $\rho$ on all of $\Sigma$. We denote the negative constant on the right hand side of our inequality by $-A$.

Therefore, we may pick a constant $C$ large enough such that $C>2A$.

Let  $f(s)$ be a bump function that vanishes outside $[-1, 2]$ and is equal to $1$ between $0$ and $1$ as shown in Figure (\ref{fig: function_requirements}).
For example we can require $f(s)$ to be:
\begin{equation}
f(s) = 
\begin{cases}
0 & s\in (-\infty, -1]\\
\exp(-\frac{s^2}{1-s^2}) & s \in (-1,0) \\
1 & s \in [0,1]\\
\exp(-\frac{(s-2)^2}{1-(s-2)^2} ) & s\in [1,2]\\
0 & s \in [2,\infty],\\
\end{cases}
\end{equation}
which happens to be the same function in the definition of $h_\mu$.
Thus we may define the Hamiltonian $H_s$ to be any smooth function whose restriction to the cylindrical end $\del M^{in} \times (0,\infty)$ is of the following form:
\begin{equation}\label{eq:quadratic_Hamiltonian}
	H_s(q, r) = (C -A \cdot s) r^2 \cdot f(s),
\end{equation} 
whenever $r \in [1, \infty)$.
Namely the function $c(s)$ in condition (\ref{eq:class-of-h}) is given by $(C- A\cdot s^2) \cdot f(s)$, thus $c'(s)$ would be $-A$ when $s \in [0,1]$ and vanishes outside $[-1,2]$. This definition makes sure that for the moving Lagrangian boundary, namely when $s \in [0,1]$, the Hamiltonian $H$ on cylindrical end is of the form $c(s)\cdot r^2$ for $c'(s)$ sufficiently negative.

The Hamiltonian defined in equation (\ref{eq:quadratic_Hamiltonian}) satisfies the requirement of condition (\ref{eq:class-of-h}) and the previous calculations would follow through.
\end{proof}

In this way we have shown that the maximum of $\log(\rho) - h_\mu$ on the pseudo-holomorphic strip can't lie on the boundary of $\Sigma$ for large enough $r$ and can't lie in the interior of $\Sigma$ either whenever the radial coordinate is larger than $1$. Since when $s \to \pm \infty$, the ends of the strip converges to some intersection points of different (constant) Lagrangians with Legendrian boundary and thus they have to be within $M^{in}$. We get a $C^0$ bound that depends on $H_0$ on the radial coordinate of any holomorphic strip satisfying equation (\ref{eq:condition_on_u}).

\end{proof}

\section{Acknowledgment}
	I would like to thank My advisor Mohammed Abouzaid, who pointed this particular direction of research to me and offered great guidance all the way along. I would also like to thank Pei-ken Hung, who spent hours explaining the analytical details; Dusa McDuff, Nate Bottman, Yoel Groman, Zack Sylvan, Kyler Siegel, Oleg Lazerav, Simon Brendle, Semon Rezchikov offered great insights during helpful discussions. I would also like to thank the organizers of the Kylerec workshop, and especially Thomas Kragh, for giving me a special chance to learn the current techniques and present closely related ideas in nearby Lagrangian conjecture.

\doublespacing
\bibliographystyle{abbrv}
\bibliography{paper}

\end{document}